\documentclass[11pt,twoside,a4paper]{article}
\usepackage{color,amscd,amsmath,amssymb,amsthm,latexsym,stmaryrd}
\usepackage{mathabx}
\usepackage{shuffle}
\usepackage[latin1]{inputenc}
\usepackage[T1]{fontenc}   
\usepackage[english]{babel}
\usepackage{tikz,tkz-tab}
\newcommand{\tdun}[1]{\begin{picture}(10,5)(-2,-1)
\put(0,0){\circle*{2}}
\put(3,-2){\tiny #1}
\end{picture}}

\newcommand{\tddeux}[2]{\begin{picture}(12,5)(0,-1)
\put(3,0){\circle*{2}}
\put(3,5){\circle*{2}}
\put(3,0){\line(0,1){5}}
\put(6,-2){\tiny #1}
\put(6,3){\tiny #2}
\end{picture}}

\newcommand{\tdtroisun}[3]{\begin{picture}(22,12)(-8,-1)
\put(3,0){\circle*{2}}
\put(6.5,7){\circle*{2}}
\put(-1,7){\circle*{2}}
\put(-2.8,0){\Large $\vee$}
\put(5,-2){\tiny #1}
\put(9,5){\tiny #2}
\put(-8,5){\tiny #3}
\end{picture}}

\newcommand{\tdtroisdeux}[3]{\begin{picture}(12,15)(-2,-1)
\put(0,0){\circle*{2}}
\put(0,5){\circle*{2}}
\put(0,10){\circle*{2}}
\put(0,0){\line(0,1){5}}
\put(0,5){\line(0,1){5}}
\put(3,-2){\tiny #1}
\put(3,3){\tiny #2}
\put(3,9){\tiny #3}
\end{picture}}

\input{xy}
\xyoption{all}

\title{Twisted bialgebras, cofreeness and cointeraction}
\date{}
\author{Lo\"\i c Foissy\\ \\
{\small \it Fédération de Recherche Mathématique du Nord Pas de Calais FR 2956}\\
{\small \it Laboratoire de Mathématiques Pures et Appliquées Joseph Liouville}\\
{\small \it Université du Littoral Côte d'Opale-Centre Universitaire de la Mi-Voix}\\ 
{\small \it 50, rue Ferdinand Buisson, CS 80699,  62228 Calais Cedex, France}\\ \\
{\small \it email: foissy@univ-littoral.fr}}

\theoremstyle{plain}
\newtheorem{theo}{Theorem}[section]
\newtheorem{lemma}[theo]{Lemma}
\newtheorem{cor}[theo]{Corollary}
\newtheorem{prop}[theo]{Proposition}
\newtheorem{defi}[theo]{Definition}

\theoremstyle{remark}
\newtheorem{remark}{Remark}[section]
\newtheorem{notation}{Notations}[section]
\newtheorem{example}{Example}[section]

\newcommand{\K}{\mathbb{K}}
\newcommand{\N}{\mathbb{N}}

\newcommand{\comp}{\mathbf{Comp}}
\newcommand{\set}{\mathbf{Set}}
\newcommand{\vect}{\mathbf{Vect}}
\newcommand{\I}{\mathcal{I}}
\newcommand{\cont}{\mathbf{cont}}
\newcommand{\calA}{\mathcal{A}}
\newcommand{\calB}{\mathcal{B}}
\newcommand{\calC}{\mathcal{C}}
\newcommand{\calP}{\mathcal{P}}
\newcommand{\calQ}{\mathcal{Q}}
\newcommand{\calR}{\mathcal{R}}
\newcommand{\calcomp}{\mathcal{C}omp}
\newcommand{\gr}{\mathcal{G}r}
\newcommand{\calgr}{\mathcal{G}r'}
\newcommand{\caltop}{\mathcal{T}op}
\newcommand{\calpos}{\mathcal{P}os}
\renewcommand{\cot}{co\mathcal{T}}
\newcommand{\com}{\mathbf{Com}}
\newcommand{\fun}{\mathcal{K}}
\newcommand{\fdeux}{\widehat{\mathcal{K}}}
\newcommand{\WQSym}{\mathbf{WQSym}}
\newcommand{\QSym}{\mathbf{QSym}}
\newcommand{\bfh}{\mathbf{H}}
\newcommand{\fqsym}{\mathcal{O}rd}
\newcommand{\FQSym}{\mathbf{FQSym}}
\newcommand{\chara}{\mathrm{Char}}
\newcommand{\Endo}{\mathrm{End}}
\newcommand{\tdelta}{\tilde{\Delta}}
\newcommand{\mor}{\mathrm{Mor}}
\newcommand{\HO}{\mathbf{HO}}
\newcommand{\ho}{\mathrm{ho}}
\newcommand{\CE}{\mathrm{CE}}
\newcommand{\id}{\mathrm{Id}}
\newcommand{\cc}{\mathrm{cc}}
\newcommand{\cl}{\mathrm{cl}}
\newcommand{\Vect}{\mathrm{Vect}}
\newcommand{\qsh}{\mathrm{QSh}}
\begin{document}

\maketitle

\begin{abstract}
We study twisted bialgebras and double twisted bialgebras, that is to say bialgebras in the category of linear species,
or in the category of species in the category of coalgebras. 
We define the notion of cofree twisted coalgebra and generalize Hoffman's quasi-shuffle product, obtaining in particular
a twisted bialgebra of set compositions $\calcomp$. Given a special character, this twisted bialgebra satisfies a terminal property,
generalizing the one of the Hopf algebra of quasisymmetric functions proved by Aguiar, Bergeron and Sottile.

We give $\calcomp$ a second coproduct, making it a double twisted bialgebra, and prove that it is a terminal object
in the category of double twisted bialgebras. 
Actions of characters on morphisms allow to obtain every twisted bialgebra morphisms from
a connected double twisted bialgebra to $\calcomp$.

These results are applied to examples based on graphs and on finite topologies,
obtaining species versions of the chromatic symmetric series and chromatic polynomials,
or of the Ehrhart polynomials. Moreover, through actions of monoids of characters, we obtain a
twisted bialgebraic interpretation of the duality principle.
\end{abstract}


\tableofcontents

\section*{Introduction}

Species, introduced by Joyal \cite{Joyal} in the 80's, are functors from the category
of finite sets with bijective maps to the category of vector spaces.
They form a monoidal category and a natural idea is to define and construct algebras, coalgebras and bialgebras 
in this category: this is for example explored by Aguiar and Mahajan in \cite{Aguiar,Aguiar3}. 
In the present paper, following \cite{PatrasReutenauer}, these objects will be qualified of \emph{twisted}.
A bunch of functors from species to vector spaces send twisted algebras, coalgebras or bialgebras
to algebras, coalgebras or bialgebras: we shall use here the full Fock functor $\fun$ and the bosonic Fock functor
$\fdeux$, in the terminology of \cite{Aguiar}. 
Our aim in this paper is to mimic several known results on bialgebras involving cofreeness,
shuffles and (co)interactions in the frame of twisted bialgebras. \\

We shall start in the first section with reminder on species, and examples of twisted algebras, coalgebras and bialgebras,
that is to say algebras, coalgebras and bialgebras in the category of species. A simple example is given by the species
$\com$: for any finite set $A$, $\com[A]=\K$; the product and the coproduct of $\com$ are given by:
\begin{align*}
m_{A,B}(1\otimes 1)&=1\in \com[A\sqcup B],\\
\Delta_{A,B}(1)&=1\otimes 1\in \com[A]\otimes \com[B].
\end{align*}
Two more complex examples will be particularly studied in the sequel:
firstly, the twisted bialgebra of graphs $\calgr$., which product is given by the disjoint union of graphs, and which coproduct
by extraction of subgraphs. For example, if $A$, $B$ and $C$ are finite sets:
\begin{align*}
\Delta_{A,B\sqcup C}(\tdtroisdeux{$A$}{$B$}{$C$})&=\tdun{$A$}\otimes \tddeux{$B$}{$C$},&
\Delta_{B\sqcup C,A}(\tdtroisdeux{$A$}{$B$}{$C$})&=\tddeux{$B$}{$C$}\otimes \tdun{$A$},\\
\Delta_{B,A\sqcup C}(\tdtroisdeux{$A$}{$B$}{$C$})&=\tdun{$B$}\otimes \tdun{$A$}\tdun{$C$},&
\Delta_{A\sqcup C,B}(\tdtroisdeux{$A$}{$B$}{$C$})&=\tdun{$A$}\tdun{$C$}\otimes \tdun{$B$},\\
\Delta_{C,A\sqcup B}(\tdtroisdeux{$A$}{$B$}{$C$})&=\tdun{$C$}\otimes \tddeux{$A$}{$B$},&
\Delta_{A\sqcup B,C}(\tdtroisdeux{$A$}{$B$}{$C$})&=\tddeux{$A$}{$B$}\otimes \tdun{$C$}.
\end{align*} 
This twisted bialgebra is commutative and cocommutative.
Secondly, the twisted bialgebra of finite topologies, or,  by Alexandroff's theorem \cite{Alexandrov}, of quasi-posets,
which product is given by the disjoint union of finite topologies, and which coproduct is given
by the extraction of open subsets. For example, if $A$, $B$ and $C$ are nonempty finite sets:
\begin{align*}
\Delta_{A,B\sqcup C}(\tdtroisdeux{$A$}{$B$}{$C$})&=\tdun{$A$}\otimes \tddeux{$B$}{$C$},&
\Delta_{B\sqcup C,A}(\tdtroisdeux{$A$}{$B$}{$C$})&=0,\\
\Delta_{B,A\sqcup C}(\tdtroisdeux{$A$}{$B$}{$C$})&=0,&
\Delta_{A\sqcup C,B}(\tdtroisdeux{$A$}{$B$}{$C$})&=0,\\
\Delta_{C,A\sqcup B}(\tdtroisdeux{$A$}{$B$}{$C$})&=0,&
\Delta_{A\sqcup B,C}(\tdtroisdeux{$A$}{$B$}{$C$})&=\tddeux{$A$}{$B$}\otimes \tdun{$C$}.
\end{align*} 
We here represent finite topologies or quasi-posets by their Hasse graphs.
This twisted bialgebra is commutative and not cocommutative. \\

We deal with the cofree twisted coalgebra $\cot(\calP)$ 
(co)generated by a species $\calP$ (Definition \ref{defi5}) in the second section. It satisfies a universal property, justifying
the adjective "cofree" (Theorem \ref{theo6}). Two particular examples will be used:
the first one, $\calcomp$, based on set compositions, is generated by the augmentation ideal of $\com$;
 the second one, $\fqsym$, based on total orders, is generated by a species concentrated in degree $1$.
The universal property is used to extend Hoffman's quasi-shuffle product to the twisted case:
if $\calP$ is a twisted algebra, then the cofree twisted coalgebra $\cot(\calP_+)$ of its augmentation ideal $\calP_+$
inherits a product $\squplus$ making it a twisted bialgebra; when the product of $\calP_+$ is trivial,
one obtains the simpler shuffle product $\shuffle$. This is applied to $\calcomp$
and to $\fqsym$. For example, in $\calcomp$, if $A$, $B$, $C$ and $D$ are nonempty finite sets:
\begin{align*}
(A)\squplus (B)&=(A,B)+(B,A)+(A\sqcup B),\\
(A,B)\squplus (C)&=(A,B,C)+(A,C,B)+(C,A,B)+(A\sqcup B,C)+(A,B\sqcup C),\\
(A,B)\squplus (C,D)&=(A,B,C,D)+(A,C,B,D)+(C,A,B,D)\\
&+(A,C,D,B)+(C,A,D,B)+(C,D,A,B)\\
&+(A,B\sqcup C,D)+(A\sqcup C,B,D)+(A\sqcup C,D,B)\\
&+(A,C,B\sqcup D)+(C,A,B\sqcup D)+(C,A\sqcup D,B)\\
&+(A\sqcup C,B\sqcup D),
\end{align*}
whereas:
\begin{align*}
(A)\shuffle (B)&=(A,B)+(B,A),\\
(A,B)\shuffle (C)&=(A,B,C)+(A,C,B)+(C,A,B),\\
(A,B)\shuffle (C,D)&=(A,B,C,D)+(A,C,B,D)+(C,A,B,D)\\
&+(A,C,D,B)+(C,A,D,B)+(C,D,A,B).
\end{align*}
These shuffle bialgebras are, in some sense, universal for commutative cofree twisted
coalgebras, see Theorem \ref{theo18}.\\

We extend in the third section  the notion of characters to twisted bialgebras (Definition \ref{defi14}). 
As in the classical case, the characters form a monoid $\chara(\calP)$ for the convolution product $*$ associated to $\Delta$.
If $\calP$ is a connected twisted bialgebra, that is to say if $\calP[\emptyset]$ is one-dimensional,
this monoid has a richer structure: for any $q\in \K$, there is a map $\lambda\longrightarrow \lambda^q$
such that:
\begin{align*}
&\forall q,q'\in \K,\:\forall \lambda \in \chara(\calP),&
\lambda^q*\lambda^{q'}&=\lambda^{q+q'},&(\lambda^q)^{q'}&=\lambda^{qq'}.
\end{align*}
Taking $q=-1$, one obtains an inverse of $\lambda$: in this case, $\chara(\calP)$ is a group.
Moreover, the twisted bialgebra $\calcomp$ inherits a special character $\varepsilon'$ making it a terminal object
in a suitable category of connected twisted bialgebras with a character (Theorem \ref{theo17}): 
this is the twisted version of Aguiar, Bergeron and Sottile's theorem on combinatorial Hopf algebras \cite{Aguiar2}.\\

All these constructions are based on the Cauchy tensor product of species.
We turn to the Hadamard tensor product of species in the fourth section.
This allows to define a second kind of bialgebras in the category of species (Definition \ref{defi19}).
We shall be especially interested in species with both sort of structures and with a nice cointeraction
between them, mimicking the notion of bialgebras in cointeraction we now recall.

A pair of bialgebras in cointeraction is a pair $(A,m_A,\Delta)$ and $(B,m_B,\delta)$ of bialgebras
with a (right) coaction $\rho:A\longrightarrow A\otimes B$ of $B$ on $A$ such that the product, coproduct,
counit and unit of $A$ are morphisms of right comodules. In other words:
\begin{itemize}
\item $\rho(1_A)=1_A\otimes 1_B$ and for any $a,b\in A$, $\rho(ab)=\rho(a)\rho(b)$:
$\rho$ is an algebra morphism.
\item $(\Delta\otimes \id)\circ \rho=m_{1,3,24}\circ (\rho\otimes \rho) \circ \Delta$, where
\[m_{1,3,24}:\left\{\begin{array}{rcl}
A\otimes B\otimes A\otimes B&\longrightarrow&A\otimes A\otimes b\\
a_1\otimes b_1\otimes a_2\otimes b_2&\longrightarrow&a_1\otimes a_1\otimes b_1b_2.
\end{array}\right.\]
\item For any $a\in A$, $(\varepsilon_A\otimes \id)\circ \rho(a)=\varepsilon_A(a)1_B$.
\end{itemize}
For example, the polynomial algebra $\K[X]$ is given two multiplicative coproducts by:
\begin{align*}
\Delta(X)&=X\otimes 1+1\otimes X,&\delta(X)&=X\otimes X.
\end{align*} 
Then $(\K[X],m,\Delta)$ and $(\K[X],m,\delta)$ are in cointeraction with the coaction $\rho=\delta$.
In numerous examples, $(A,m)=(B,m)$ and $\rho=\delta$; 
see \cite{CEFM,Foissychrom,FoissyEhrhart,ManchonBelhaj} for combinatorial
examples based on families of graphs, oriented graphs or posets, and applications of the cointeraction.

This notion is extended to species, obtaining double twisted bialgebras (Definition \ref{defi21}).
For example, the species $\com$ is a double twisted bialgebra: it has a second coproduct $\delta$, such that for any finite set $A$:
\[\delta_A:\left\{\begin{array}{rcl}
\com[A]&\longrightarrow&\com[A]\otimes \com[A]\\
1&\longrightarrow&1\otimes 1.
\end{array}\right.\]
 The set of characters of a double twisted bialgebra $\calP$ holds a second convolution product $\star$ associated to $\delta$,
making it a monoid. Moreover:
\begin{align*}
&\forall \lambda,\mu,\nu\in \chara(\calP),&
(\lambda*\mu)\star \nu&=(\lambda\star \nu)*(\mu\star\nu).
\end{align*}
In other words, $(\chara(\calP),\star)$ acts on $(\chara(\calP),*)$ by monoid endomorphisms.
For any twisted bialgebra $\calQ$, this monoid also acts on the set $\mor_B(\calP,\calQ)$ of twisted bialgebra morphisms
from $\calP$ to $\calQ$ (Proposition \ref{prop26}), giving a right action:
\[\left\{\begin{array}{rcl}
\mor_B(\calP,\calQ)\times \chara(\calP)&\longrightarrow&\mor_B(\calP)\\
(\phi,\lambda)&\longrightarrow&\phi \leftarrow \lambda.
\end{array}\right.\]
In particular, we obtain an injective morphism of monoids:
\[\chi_\calP:\left\{\begin{array}{rcl}
(\chara(\calP),\star)&\longrightarrow&(\mor_B(\calP,\calP),\circ)\\
\lambda&\longrightarrow&\id_{\calP}\leftarrow \lambda.
\end{array}\right.\]
We prove that if $\calP$ is a commutative double twisted  bialgebra,
then the twisted cofree coalgebra $\cot(\calP_+)$ is a double twisted bialgebra, with the quasi-shuffle product $\squplus$,
the deconcatenation coproduct $\Delta$ and another coproduct $\delta$ defined by a extraction-contraction process
(Proposition \ref{prop23}). This is directly applied to $\calcomp$ in Corollary \ref{cor24}.
For example, if $A$, $B$ and $C$ are nonempty finite subsets:
\begin{align*}
\delta(A)&=(A)\otimes (A),\\
\delta(A,B)&=(A,B)\otimes ((A)\squplus (B))+(A\sqcup B)\otimes (A,B),\\
\delta(A,B,C)&=(A,B,C)\otimes ((A)\squplus (B)\squplus (C))+(A,B\sqcup C)\otimes ((A)\squplus (B,C))\\
&+(A\sqcup B,C)\otimes ((A,B)\squplus (C))+(A\sqcup B\sqcup C)\otimes (A,B,C).
\end{align*}
We then prove a universal property for this twisted double bialgebra in Theorem \ref{theo29}:
if $\calP$ is a twisted double bialgebra satisfying a condition of connectivity,
then there exists a unique morphism $\phi$ of twisted double bialgebra from $\calP$ to $\calcomp$.
Moreover, all the morphisms of twisted bialgebras from $\calP$ to $\calcomp$ can be deduced from $\phi$
by Theorem \ref{theo17}, using the action $\leftarrow$ of the monoid $\chara(\calP)$. \\

Graphs and finite topologies are given a structure of double twisted bialgebras in the next two sections.
For example, in $\calgr$, If $A$, $B$, $C$ are nonempty finite sets:
\begin{align*}
\delta(\tdun{$A$})&=\tdun{$A$}\otimes \tdun{$A$},\\
\delta(\tddeux{$A$}{$B$})&=\tddeux{$A$}{$B$}\otimes \tdun{$A$}\tdun{$B$}+
\tdun{$A\sqcup B$} \hspace{6mm} \otimes \tddeux{$A$}{$B$},\\
\delta(\tdtroisun{$A$}{$C$}{$B$})&=\tdtroisun{$A$}{$C$}{$B$}\otimes \tdun{$A$}\tdun{$B$}\tdun{$C$}
+\tddeux{$A\sqcup B$}{$C$}\hspace{6mm}\otimes \tddeux{$A$}{$B$}\tdun{$C$}+
\tddeux{$A\sqcup C$}{$B$}\hspace{6mm}\otimes \tddeux{$A$}{$C$}\tdun{$B$}+
\tdun{$A\sqcup B\sqcup C$}\hspace{12mm}\otimes \tdtroisun{$A$}{$C$}{$B$}.
\end{align*}
In $\caltop$:
\begin{align*}
\delta(\tdun{$A$})&=\tdun{$A$}\otimes \tdun{$A$},\\
\delta(\tddeux{$A$}{$B$})&=\tddeux{$A$}{$B$}\otimes \tdun{$A$}\tdun{$B$}+
\tdun{$A\sqcup B$} \hspace{6mm} \otimes \tddeux{$A$}{$B$},\\
\delta(\tdtroisun{$A$}{$C$}{$B$})&=\tdtroisun{$A$}{$C$}{$B$}\otimes \tdun{$A$}\tdun{$B$}\tdun{$C$}
+\tddeux{$A\sqcup B$}{$C$}\hspace{6mm}\otimes \tddeux{$A$}{$B$}\tdun{$C$}+
\tddeux{$A\sqcup C$}{$B$}\hspace{6mm}\otimes \tddeux{$A$}{$C$}\tdun{$B$}+
\tdun{$A\sqcup B\sqcup C$}\hspace{12mm}\otimes \tdtroisun{$A$}{$C$}{$B$},\\
\delta(\tdtroisdeux{$A$}{$B$}{$C$})&=\tdtroisdeux{$A$}{$B$}{$C$}\otimes \tdun{$A$}\tdun{$B$}\tdun{$C$}
+\tddeux{$A\sqcup B$}{$C$}\hspace{6mm}\otimes \tddeux{$A$}{$B$}\tdun{$C$}+
\tddeux{$A$}{$B\sqcup C$}\hspace{6mm}\otimes \tdun{$A$}\tddeux{$B$}{$C$}+
\tdun{$A\sqcup B\sqcup C$}\hspace{12mm}\otimes \tdtroisdeux{$A$}{$B$}{$C$}.
\end{align*}
The unique double twisted bialgebra morphism $\phi_{chr}:\calgr\longrightarrow \calcomp$ 
is related to the noncommutative chromatic formal series of \cite{Gebhard}
(Proposition \ref{prop32}). For example:
\begin{align*}
\phi_{chr}(\tdun{$A$})&=(A),\\
\phi_{chr}(\tddeux{$A$}{$B$})&=(A,B)+(B,A),\\
\phi_{chr}(\tdtroisdeux{$A$}{$B$}{$C$})&=(A,B,C)+(A,C,B)+(B,A,C)+(B,C,A)+(C,A,B)+(C,B,A)\\
&+(A\sqcup C,B)+(B,A\sqcup C).
\end{align*}
The unique double twisted bialgebra morphism $\phi_{ehr}:\calgr\longrightarrow \calcomp$ 
 is related to linear extensions and to the Ehrhart quasisymmetric function (Theorem \ref{theo41}). For example:
\begin{align*}
\phi_{ehr}(\tdun{$A$})&=(A),\\
\phi_{ehr}(\tddeux{$A$}{$B$})&=(A,B),\\
\phi_{ehr}(\tdtroisun{$A$}{$C$}{$B$})&=(A,B,C)+(A,C,B)+(A,B\sqcup C),\\
\phi_{ehr}(\tdtroisdeux{$A$}{$B$}{$C$})&=(A,B,C).
\end{align*}
We also prove in this section that $\caltop$ is a cofree twisted coalgebra, so is isomorphic to a
quasi-shuffle twisted bialgebra (Theorem \ref{theo42}). \\

The seventh section is devoted to the study of homogeneity. We consider now graded twisted bialgebras,
that is to say twisted bialgebra $\calP$ admitting a decomposition
\[\calP=\bigoplus_{n\geq 0}\calP_n,\]
compatible with the product and coproduct of $\calP$ (Definition \ref{defi43}). 
For example, $\calgr$ admits a graduation, given by the number of vertices of the graphs, 
and $\caltop$ admits a graduation, given by the number of classes of the quasi-posets,
or equivalently the number of vertices of the Hasse graphs. If $\calP$ is graded, it admits a family $(\iota_q)_{q\in \K}$
of bialgebra endomorphisms, such that for any $q$, $q'\in \K$:
\[\iota_q\circ \iota_{q'}=\iota_{qq'}.\]
Conversely, we define in Proposition \ref{prop44} such a family $(\theta_q)_{q\in \K}$ of endomorphisms of $\calcomp$, given by:
\[\theta_q=\id_{\calcomp}\leftarrow \varepsilon'^q.\]
These morphisms imply a graduation of $\calcomp$. For example:
\begin{align*}
\calcomp_1[A]&=\Vect((A)),\\
\calcomp_2[A]&=Vect\left((A_1,A_2)+\frac{1}{2}(A_1\sqcup A_2),\: (A_1,A_2)\in \comp[A]\right),\\
\calcomp_3[A]&=Vect\left(\begin{array}{c}
(A_1,A_2,A_3)+\frac{1}{2}(A_1\sqcup A_2,A_3)+
\frac{1}{2}(A_1, A_2\sqcup A_3)+\frac{1}{6}(A_1\sqcup A_2\sqcup A_3),\\
 (A_1,A_2,A_3)\in \comp[A])
 \end{array}\right).
\end{align*}
For any graded connected twisted bialgebra $\calP$, we prove that the set of homogeneous morphisms
from $\calP$ to $\calcomp$ is in one-to-one correspondence with the set of morphisms from $\calP_1$ to $\com$
(Corollary \ref{cor47}). As applications, and we describe all homogeneous twisted bialgebra morphisms
from $\calgr$ or from $\caltop$ to $\calcomp$ (Proposition \ref{prop48} and \ref{prop56}). 
For graphs, we shall especially consider a  one-parameter family of them, denoted by $(\phi_q)_{q\in \K}$. For example:
\begin{align*}
\phi_q(\tdun{$A$})&=q(A),\\
\phi_q(\tddeux{$A$}{$B$})&=q^2(A,B)+q^2(B,A)+q^2(A\sqcup B)\\
\phi_q(\tdtroisun{$A$}{$C$}{$B$})&=q^3(A,B,C)+q^3(A,C,B)+q^3(B,A,C)+q^3(B,C,A)+q^3(C,A,B)+q^3(C,B,A)\\
&+q^3(A\sqcup B,C)+q^3(A\sqcup C,B)+q^3(B\sqcup C,A)\\
&+q^3(A,B\sqcup C)+q^3(B,A\sqcup C)+q^3(C,A\sqcup B)+q^3(A\sqcup B\sqcup C).
\end{align*}
Playing with the graduation, we deform the morphism $\phi_{chr}$, putting:
\begin{align*}
&\forall q\in \K,& \phi_{chr_q}&=\theta_q^{-1}\circ \phi_{chr} \circ \iota_q.
\end{align*} 
We prove the existence of a character $\lambda_{chr_q}$ such that:
\[\phi_{chr_q}=\phi_1\leftarrow \lambda_{chr_q}.\]
This rather complicated character is invertible for the $\star$ product, and for any graph $G$,
$\lambda_{chr_q}^{-1}(G)$ is a power of $q$ (Corollary \ref{cor53}). We finally obtain  in Corollary \ref{cor53}
a diagram of twisted bialgebras morphisms
\[\xymatrix{\calgr\ar[rr]^{\Gamma} \ar[rd]_{\phi_{chr_{-1}}}&&\calgr\ar[ld]^{\phi_{chr_1}}\\
&\calcomp&}\]
where $\Gamma$ is a twisted bialgebra automorphism related to a character defined by acyclic orientations of graphs.
Similar results are obtained for finite topologies.  \\ 

We apply the Fock functors to these constructions in the eighth section. 
Both of them send twisted bialgebras to bialgebras, and the bosonic one sends double twisted bialgebras 
to bialgebras in cointeraction (Theorem \ref{theo62}).
Note that this second point generally does not work for the full Fock functor by a lack of commutativity. We obtain:
\begin{itemize}
\item For $\com$: both Fock functors give $\K[X]$, with its usual product and coproducts. 
\item For $\fqsym$: the full Fock functor gives Malvenuto and Reutenauer's Hopf algebra $\FQSym$ of permutations
\cite{DHT,MR2}, whereas the bosonic Fock functor gives again the polynomial algebra $\K[X]$.
\item For $\calcomp$: the full Fock functor gives the Hopf algebra of packed words (or of set compositions) $\WQSym$
\cite{NovelliThibon,NovelliThibon2,NovelliThibon3} with its internal coproduct; the bosonic Fock functor
give the Hopf algebra of quasi-symmetric functions $\QSym$.
\item For the species of graphs $\calgr$, we obtain the Hopf algebras $\bfh_{\calgr}$ and
$H_{\calgr}$ described in \cite{Foissychrom}.
The induced morphism $\phi$ is Stanley's chromatic symmetric function \cite{Stanley2}.
\item For the species of finite topologies $\caltop$, we obtain the Hopf algebras $\bfh_{\caltop}$
and $H_{\caltop}$ described in \cite{FoissyEhrhart}.
The induced morphism $\phi$ is the (strict) Ehrhart quasisymmetric function.
\end{itemize}
We also prove in Theorem \ref{theo68} that the image of a a cofree twisted bialgebra by the full Fock functor 
is free and cofree (getting back a known result on $\FQSym$ and $\WQSym$).\\

The last section is devoted to applications of these results on graphs and finite topologies.
We also prove that there exists a unique morphism $H:\QSym\longrightarrow \K[X]$ of double bialgebra,
which we use to recover the chromatic polynomial $P_{chr}$ and the (strict) Ehrhart polynomial $P_{ehr}$. 
We finally obtain the following morphims of bialgebra morphisms:
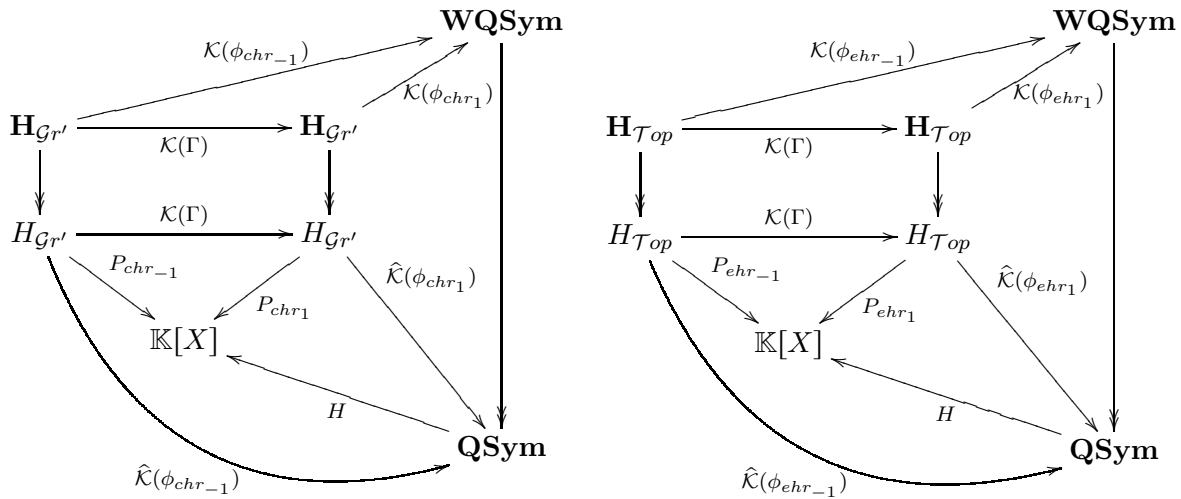
\begin{align*}
&\xymatrix{&&&\WQSym\ar@{->>}[dddd]\\
\bfh_{\calgr}\ar[rr]_{\fun(\Gamma)}\ar@{->>}[d]\ar[rrru]^{\fun(\phi_{chr_{-1}})}
&&\bfh_{\calgr}\ar@{->>}[d]\ar[ru]_{\fun(\phi_{chr_1})}&\\
H_{\calgr}\ar[rr]^{\fun(\Gamma)}\ar[dr]^{P_{chr_{-1}}}\ar@/_4pc/[ddrrr]_{\fdeux(\phi_{chr_{-1}})}
&&H_{\calgr}\ar[ld]^{P_{chr_1}}\ar[rdd]^<(.2){\fdeux(\phi_{chr_1})}&\\
&\K[X]&&\\
&&&\QSym\ar[ull]^H}
&\xymatrix{&&&\WQSym\ar@{->>}[dddd]\\
\bfh_{\caltop}\ar[rr]_{\fun(\Gamma)}\ar@{->>}[d]\ar[rrru]^{\fun(\phi_{ehr_{-1}})}
&&\bfh_{\caltop}\ar@{->>}[d]\ar[ru]_{\fun(\phi_{ehr_1})}&\\
H_{\caltop}\ar[rr]^{\fun(\Gamma)}\ar[dr]^{P_{ehr_{-1}}}\ar@/_4pc/[ddrrr]_{\fdeux(\phi_{ehr_{-1}})}
&&H_{\caltop}\ar[ld]^{P_{ehr_1}}\ar[rdd]^<(.2){\fdeux(\phi_{ehr_1})}&\\
&\K[X]&&\\
&&&\QSym\ar[ull]^H}
\end{align*}

\begin{notation} \begin{itemize}
\item Let $\K$ be a commutative field of characteristic zero. All the objects of this text
(vector spaces, algebras, coalgebras, bialgebras$\ldots$) will be taken over $\K$.

\item For any $n\in \N$, we denote by $\underline{n}$ the set $\{1,\ldots,n\}$. In particular, 
$\underline{0}=\emptyset$.
\item For any $n\in \N$, we denote by $\mathfrak{S}_n$ the $n$-th symmetric group. 
\item If $A$ is a finite set, we denote by $\sharp A$ its cardinality.
\end{itemize}
\end{notation}

\section{Bialgebras in the category of species}

\subsection{Reminders and examples}

We refer to \cite{Bergeron,Joyal} for a more complete exposition of species.
We denote by $\set$ the category of finite sets with bijections and by $\vect$ the category of vector spaces over $\K$
with linear maps. A linear species is a functor from $\set$ to $\vect$. 
In other words, a species $\calP$ is given by:
\begin{itemize}
\item For any finite set $A$, a vector space $\calP[A]$.
\item For any bijection $\sigma:A\longrightarrow B$ between two finite sets, a linear map
 $\calP[\sigma]:\calP[A]\longrightarrow \calP[B]$.
\end{itemize}
The following properties are satisfied:
\begin{itemize}
\item For any finite set $A$, $\calP[\id_A]=\id_{\calP[A]}$.
\item For any finite sets, $A,B,C$ and any bijections $\sigma: A\longrightarrow B$ and $\tau:B\longrightarrow C$,
\[\calP[\tau \circ \sigma]=\calP[\tau]\circ \calP[\sigma].\]
\end{itemize}
If $\calP$ and $\calQ$ are two species, a morphism of species $f:\calP\longrightarrow \calQ$ is 
a natural transformation from $\calP$ to $\calQ$, that is, for any finite set $A$, $f[A]$ is a linear map from 
$\calP[A]$ to $\calQ[A]$, such that for any bijection $\sigma:A\longrightarrow B$ between two finite sets, 
the following diagram commutes:
\[\xymatrix{\calP[A]\ar[r]^{\calP[\sigma]} \ar[d]_{f[A]}&\calP[B]\ar[d]^{f[B]}\\
\calQ[A]\ar[r]_{\calQ[\sigma]}&\calQ[B]}\]

\begin{example}
\begin{enumerate}
\item We define a species $\com$ by the following:
\begin{itemize}
\item For any finite set $A$, $\com[A]=\K$. 
\item For any bijection $\sigma:A\longrightarrow B$, $\com[\sigma]=\id_\K$.
\end{itemize} 
\item For any finite set $A$, let us denote by $\comp[A]$ the set of set compositions of $A$, that is to say 
finite sequences $(A_1,\ldots,A_k)$ of nonempty subsets of $A$, such that  $A_1\sqcup \ldots \sqcup A_k=A$.
Let $\calcomp[A]$ be the space generated by $\comp[A]$.
If $\sigma:A\longrightarrow B$ is a bijection:
\[\calcomp[\sigma]:\left\{\begin{array}{rcl}
\calcomp[A]&\longrightarrow&\calcomp[B]\\
(A_1,\ldots,A_k)\in \comp[A]&\longrightarrow&(\sigma(A_1),\ldots,\sigma(A_k))\in \comp[B].
\end{array}\right.\]
For example, if $a,b$ are two different elements of a set $A$:
\begin{align*}
\calcomp[\emptyset]&=\Vect(\emptyset),\\
\calcomp[\{a\}]&=\Vect(\{a\}),\\
\calcomp[\{a,b\}]&=\Vect((\{a\},\{b\}),(\{b\},\{a\}),(\{a,b\}))
\end{align*}
\item We define the species of graphs $\gr$ by the following:
\begin{itemize}
\item For any finite set $A$, $\gr[A]$ is the vector space generated by graphs $G$ which vertex set $V(G)$ is equal to $A$.
\item For any bijection $\sigma:A\longrightarrow B$, $\gr[\sigma]$ sends the graph $G=(A,E(G))$,
where $E(G)$ is the set of edges of $A$, to the graph $\gr[\sigma](G)=G'=(B,E(G'))$, with:
\[E(G')=\{\{\sigma(a),\sigma(b)\}, \{a,b\}\in E(G)\}.\]
For example, if $a,b$ are two different elements of a set $A$:
\begin{align*}
\gr[\emptyset]&=\Vect(\emptyset),&
\gr[\{a\}]&=\Vect(\tdun{$a$}),&
\gr[\{a,b\}]&=\Vect(\tdun{$a$}\tdun{$b$}, \tddeux{$a$}{$b$}).
\end{align*}
\end{itemize}
\item We define the species of posets $\calpos$ by the following:
\begin{itemize}
\item For any finite set $A$, $\calpos[A]$ is the vector space generated by posets $P$ which vertex set is $A$,
that is to say pairs $P=(A,\leq_P)$, such that $\leq_P$ is a partial order on $A$.
\item For any bijection $\sigma:A\longrightarrow B$, $\calpos[\sigma]$ sends the poset $P=(A,\leq_P)$,
 to the poset $\calpos[\sigma](P)=(B,\leq_P')$, with:
\begin{align*}
&\forall a,b\in B,&a\leq_P' b\Longleftrightarrow \sigma^{-1}(a)\leq_P \sigma^{-1}(b).
\end{align*}
\end{itemize}
Posets will be represented by their Hasse graphs. For example, if $a,b$ are two different elements of a set $A$:
\begin{align*}
\calpos[\emptyset]&=\Vect(\emptyset),\\
\calpos[\{a\}]&=\Vect(\tdun{$a$}),\\
\calpos[\{a,b\}]&=\Vect(\tdun{$a$}\tdun{$b$},\tddeux{$a$}{$b$},\tddeux{$b$}{$a$}).
\end{align*}
\end{enumerate}\end{example}

\begin{notation} 
If $\calP$ and $\calQ$ are species, we define the composition of the species $\calP$ and $\calQ$ by:
\[\calP\circ \calQ[A]=\bigoplus_{\mbox{\scriptsize $I$ partition of $A$}}
\calP[I]\otimes \left(\bigotimes_{X\in I} \calQ[X]\right).\]
We shall especially consider the case where $\calQ=\com$:
\[\calP\circ \com[A]=\bigoplus_{\mbox{\scriptsize $I$ partition of $A$}} \calP[I].\]
\end{notation}

\begin{example}
\begin{enumerate}
\item We put $\calgr=\gr\circ \com$. For any finite set $A$, $\calgr[A]$ 
is the vector space generated by graphs which set of vertices is a partition of $A$. For example, if $A=\{a,b\}$:
\[\calgr[A]=\Vect(\tdun{$\{a\}$}\hspace{2mm}\tdun{$\{b\}$}\hspace{2mm},\tddeux{$\{a\}$}{$\{b\}$}\hspace{2mm},
\tdun{$\{a,b\}$}\hspace{4mm}).\]
If $\sigma:A\longrightarrow B$ is a bijection, the map $\calgr[\sigma]:\calgr[A]\longrightarrow \calgr[B]$ 
is given by the action of $\sigma$ on the elements of $A$.
\item The species $\caltop$ of finite topologies is given by the following:
\begin{itemize}
\item For any finite set $A$, $\caltop[A]$ is the vector space generated by the set of topologies on $A$,
that is to say subsets $T$ of the set of subsets of $A$ with the following properties:
\begin{enumerate}
\item $\emptyset$ and $A$ belong to $T$.
\item If $X,Y\in T$, then $X\cup Y \in T$ and $X\cap Y \in T$.
\end{enumerate}
\item If $\sigma:A\longrightarrow B$ and $T$ is a topology on $A$, then:
\[\caltop[\sigma](T)=\{\sigma(X),X\in T\}\]
\end{itemize}
A quasi-poset is a pair $(A,\leq_P)$, where $A$ is a set and $\leq_P$ a quasi-order on $A$,
that is to say a transitive and reflexive relation on $A$.
By Alexandroff's theorem \cite{Alexandrov,Stong}, for any finite set $A$,
there is a bijection between topologies on $A$ and quasi-posets $P=(A,\leq_P)$.
This bijection associates to a quasi-order $\leq$ on $A$ the set $Top(\leq)$ of its open sets: 
$X\subseteq A$ is $\leq$-open if:
\begin{align*}
&\forall a,b\in A,&a\leq b\mbox{ and }a\in X\Longrightarrow b\in X.
\end{align*}
In the sequel, we shall identify in this way finite topologies and quasi-posets.
If $T=(A,\leq_T)$ is a quasi-poset, one defines an equivalence on $A$ by:
\[a \sim_T b\mbox{ if } a\leq_T b\mbox{ and }b\leq_T a.\]
The quotient space $A/\sim_T$  is given a partial order $\overline{\leq}_T$ by:
\[X\overline{\leq}_T Y\mbox{ if }\forall a\in X,\: \forall b\in Y,\: a\leq_T b.\]
Therefore, $(A/\sim_T ,\overline{\leq}_T)$ is a poset. 
This implies that the species $\caltop$ and $\calpos \circ \com$ are isomorphic. 
For example, if $A=\{a,b\}$, we represent finite topologies on $A$ by the Hasse graph of $\overline{\leq}_T$,
the vertices being decorated by the corresponding equivalence class of $\sim_T$:
\[\caltop[A]=\Vect(\tdun{$\{a\}$}\hspace{2mm}\tdun{$\{b\}$}\hspace{2mm},\tddeux{$\{a\}$}{$\{b\}$}\hspace{2mm},
\tddeux{$\{b\}$}{$\{a\}$}\hspace{2mm},\tdun{$\{a,b\}$}\hspace{4mm}).\]
The open sets of these topologies are given in the following table:
\[\begin{array}{|c|c|c|c|c|}
\hline&\emptyset&\{a\}&\{b\}&\{a,b\}\\
\hline\tdun{$\{a\}$}\hspace{2mm}\tdun{$\{b\}$}\hspace{2mm}&\times&\times&\times&\times\\
\hline \tddeux{$\{a\}$}{$\{b\}$}\hspace{2mm}&\times&&\times&\times\\
\hline \tddeux{$\{b\}$}{$\{a\}$}\hspace{2mm}&\times&\times&&\times\\
\hline \tdun{$\{a,b\}$}\hspace{4mm}&\times&&&\times\\
\hline\end{array}\]
\end{enumerate}\end{example}

\begin{notation}\begin{enumerate}
\item If $G$ is a graph, the number of its vertices is denoted by $\deg(G)$, and the number of its connected components
is denoted by $\cc(G)$.
\item If $T$ is a finite topology on a set $A$, we denote by $CL(T)$ the quotient space $A/\sim_T$ and by $\cl(T)$ its cardinal
(which is also the number of vertices of the Hasse graph of $T$), and we denote by $\cc(T)$ then number of
its connected components (which is also the number of connected components of the Hasse graph of $T$).
\end{enumerate}\end{notation}

\subsection{Algebras in the category of species}

Let $\calP$, $\calQ$ be two species. The Cauchy tensor product of species $\calP\otimes \calQ$ is defined as this:
\begin{itemize}
\item For any finite set $A$:
\[\calP\otimes \calQ[A]=\bigoplus_{I\subseteq A} \calP[I]\otimes \calQ[A\setminus I].\]
\item For any bijection  $\sigma:A\longrightarrow B$ between two finite sets:
\[\calP\otimes \calQ[\sigma]:\left\{\begin{array}{rcl}
\calP\otimes \calQ[A]&\longrightarrow&\calP\otimes \calQ[A]\\
x\otimes y \in \calP[I]\otimes \calQ[A\setminus I]&\longrightarrow&
\underbrace{\calP[\sigma_{\mid I}](x)\otimes \calQ[\sigma_{\mid A\setminus I}](y)}
_{\in \calP[\sigma(I)]\otimes \calQ[\sigma(A\setminus I)]}.
\end{array}\right.\]
\end{itemize}
For any species $\calP$, $\calQ$ and $\calR$, $\calP\otimes (\calQ\otimes \calR)
=(\calP\otimes \calQ)\otimes \calR$. The unit is the species $\I$:
\[\I[A]=\begin{cases}
\K\mbox{ if }A=\emptyset,\\
(0)\mbox{ otherwise}.
\end{cases}\]
For any species $\calP$, $\calP\otimes \I=\I\otimes \calP=\calP$.\\

If $f:\calP\longrightarrow \calP'$ and $g:\calQ\longrightarrow \calQ'$ are morphisms of species, 
then $f\otimes g:\calP\otimes \calQ\longrightarrow \calP'\otimes \calQ'$ is a morphism of species:
\[f\otimes g[A]:\left\{\begin{array}{rcl}
\calP\otimes \calQ[A]&\longrightarrow&\calP'\otimes \calQ'[A]\\
x\otimes y\in \calP[I]\otimes \calQ[A\setminus I]&\longrightarrow&\underbrace{f[I](x)\otimes g[A\setminus I](y)}
_{\in \calP'[I]\otimes \calQ'[A\setminus I]}.
\end{array}\right.\]
For any species $\calP,\calQ$, we define a species isomorphism 
$c_{\calP,\calQ}:\calP\otimes \calQ\longrightarrow \calQ\otimes \calP$ by:
\[c_{\calP,\calQ}[A]:\left\{\begin{array}{rcl}
\calP\otimes \calQ[A]&\longrightarrow&\calQ\otimes \calP[A]\\
x\otimes y\in \calP[I]\otimes \calQ[A\setminus I]&\longrightarrow&\underbrace{y\otimes x.}_
{\in \calQ[A\setminus I]\otimes \calP[I]}
\end{array}\right.\]

\begin{defi}
An algebra in the category of species, or twisted algebra,
is a pair $(\calP,m)$ where $\calP$ is a species and $m:\calP\otimes \calP\longrightarrow \calP$ is a morphism of species 
such that the following diagram commutes:
\[\xymatrix{\calP^{\otimes 3} \ar[r]^{m\otimes \id_\calP}\ar[d]_{\id_\calP\otimes m}&\calP^{\otimes 2}\ar[d]^{m}\\
\calP^{\otimes 2}\ar[r]_{m}&\calP}\]
Moreover, there exists a morphism of species $\iota:\I\longrightarrow \calP$ such that the following diagram commutes:
\[\xymatrix{\I\otimes \calP\ar[r]^{\iota \otimes \id_\calP}\ar[rd]_{\id}&\calP^{\otimes 2}\ar[d]^m
&\calP\otimes \I \ar[l]_{\id_\calP\otimes \iota}\ar[ld]^{\id}\\
&\calP&}\]
We shall say that $(\calP,m)$ is commutative if $m\circ c_{\calP,\calP}=m$.
\end{defi}

In other words, a twisted algebra $\calP$ is given, for any finite sets $A$ and $B$, a map 
$m_{A,B}:\calP[A]\otimes \calP[B]\longrightarrow \calP[A\sqcup B]$ such that:
\begin{enumerate}
\item For any bijections $\sigma:A\longrightarrow A'$ and $\tau:B\longrightarrow B'$ between finite sets, then:
\[m_{A',B'}\circ (\calP[\sigma]\otimes \calP[\tau])=\calP[\sigma \sqcup \tau]\circ m_{A,B}:\calP[A]\otimes \calP[B]
\longrightarrow \calP[A'\sqcup B'].\]
\item For any finite sets $A,B,C$:
\[m_{A\sqcup B,C}\circ (m_{A,B}\otimes \id_{\calP[C]})=m_{A,B\sqcup C}\circ (\id_{\calP[A]}\otimes m_{B,C}).\]
\item There exists an element $1_\calP\in \calP[\emptyset]$ such that for any finite set $A$, and $x\in \calP[A]$:
\[m_{\emptyset,A}(1_A\otimes x)=m_{A,\emptyset}(x\otimes 1_A)=x.\]
\end{enumerate}
Moreover, $\calP$ is commutative if, and only if, for any finite set $A,B$, in $\calP[A\sqcup B]$:
\begin{align*}
&\forall x\in \calP[A], \:\forall y\in \calP[B],&m_{A,B}(x\otimes y)=m_{B,A}(y\otimes x).
\end{align*}

\begin{example} The following examples can be found in \cite{Aguiar3}:
\begin{enumerate}
\item The product of $\K$ induces an algebra structure on $\com$: for any finite sets $A$ and $B$,
for any $\lambda\in \K=\com[A]$, $\mu\in \K=\com[B]$,
\[m_{A,B}(\lambda\otimes \mu)=\lambda \mu \in \K=\com[A\sqcup B].\]
The unit of $\com$ is the unit $1$ of $\K=\com[\emptyset]$.
\item The species $\calcomp$ is an algebra, with the concatenation product:
\[m_{A,B}((A_1,\ldots,A_k)\otimes (B_1,\ldots,B_l))=(A_1,\ldots,A_k,B_1,\ldots,B_l).\]
Its unit is the empty composition $\emptyset$.
\item The species of graphs $\calgr$ is an algebra with the product given by disjoint union of graphs.
The unit is the empty graph. For example, if $A$, $B$, $C$, $D$ and $E$ are finite sets, then:
\[m_{A\sqcup B\sqcup C,D\sqcup E} (\tdtroisun{$A$}{$C$}{$B$}\otimes \tddeux{$D$}{$E$})
=\tdtroisun{$A$}{$C$}{$B$}\tddeux{$D$}{$E$}.\]
\item The species $\caltop$ is a twisted algebra with the disjoint union product.
\end{enumerate}\end{example}

\begin{notation}
Let $\calP$ be a twisted algebra. If $(A_1,\ldots,A_k)\in \comp[A]$, we inductively define $m_{A_1,\ldots,A_k}$
from $\calP[A_1]\otimes \ldots \otimes \calP[A_k]$ into $\calP[A]$ by:
\begin{align*}
m_A&=\id_{\calP[A]},\\
m_{A_1,\ldots,A_k}&=m_{A_1\sqcup A_2,A_3,\ldots,A_k}\circ (m_{A_1,A_2}\otimes \id_{\calP[A_3]}\otimes \ldots \otimes
\id_{\calP[A_k]}).
\end{align*}
By associativity, if $(A_1,\ldots,A_{k+l})\in \comp[A]$:
\[m_{A_1\sqcup \ldots\sqcup A_k,A_{k+1}\sqcup \ldots \sqcup A_{k+l}}
\circ (m_{A_1,\ldots,A_k}\otimes m_{A_{k+1},\ldots,A_{k+l}})=m_{A_1,\ldots,A_{k+l}}.\]
\end{notation}

\subsection{Coalgebras and bialgebras in the category of species}

\begin{defi}
A coalgebra in the category of species, or twisted coalgebra,
is a pair $(\calP,\Delta)$ where $\calP$ is a species and $\Delta:\calP\longrightarrow \calP\otimes \calP$ 
is a morphism of species such that the following diagram commutes:
\[\xymatrix{\calP\ar[r]^{\Delta}\ar[d]_{\Delta}&\calP^{\otimes 2}\ar[d]^{\id_\calP\otimes \Delta}\\
\calP^{\otimes 2}\ar[r]_{\Delta \otimes \id_\calP}&\calP^{\otimes 3}}\]
Moreover, there exists a morphism of species $\varepsilon:\calP\longrightarrow \I$ such that the following diagram commutes:
\[\xymatrix{\I\otimes \calP\ar[r]^{\id}&\calP\ar[d]^\Delta&\calP\otimes \I \ar[l]_{\id}\\
&\calP^{\otimes 2}\ar[lu]^{\varepsilon \otimes \id_\calP}\ar[ru]_{\id_\calP\otimes \varepsilon}&}\]
We shall say that $(\calP,\Delta)$ is cocommutative if $c_{\calP,\calP}\circ \Delta=\Delta$.
We shall say that the coalgebra $(\calP,\Delta)$ is connected if $\calP[\emptyset]$ is one-dimensional.
\end{defi}

In other words, a twisted coalgebra $\calP$ is given, for any finite sets $A$ and $B$,
a map $\Delta_{A,B}:\calP[A\sqcup B]\longrightarrow \calP[A]\otimes \calP[B]$ such that:
\begin{enumerate}
\item For any bijections $\sigma:A\longrightarrow A'$ and $\tau:B\longrightarrow B'$ between finite sets, then:
\[(\calP[\sigma]\otimes \calP[\tau])\circ \Delta_{A,B}=\Delta_{A',B'}\circ \calP[\sigma\sqcup \tau]:
\calP[A\sqcup B]\longrightarrow \calP[A']\otimes \calP[B'].\]
\item For any finite sets $A,B,C$:
\[(\Delta_{A,B}\otimes \id_{\calP[C]})\circ \Delta_{A\sqcup B,C}=(\id_{\calP[A]}\otimes \Delta_{B,C})\circ \Delta_{A,B\sqcup C}.\]
\item There exists a linear map $\varepsilon_\calP:\calP[\emptyset]\longrightarrow \K$ 
such that for any finite set $A$:
\[(\varepsilon_\calP\otimes \id_{\calP[A]})\circ \Delta_{\emptyset,A}=(\id_{\calP[A]}\otimes \varepsilon_\calP)\circ 
\Delta_{A,\emptyset}=\id_{\calP[A]}.\]
\end{enumerate}

\begin{example}\begin{enumerate}
\item The species $\com$ is given a coalgebra structure by the following: for any finite sets $A,B$,
\begin{align*}
&\forall \lambda \in \K=\com[A\sqcup B],&\Delta_{A,B}(\lambda)=\lambda(1\otimes 1).
\end{align*}
The counit is the identity of $\K$.
\item The species $\calcomp$ is given a coalgebra structure by:
\[\Delta_{A,B}(A_1,\ldots,A_k)=\begin{cases}
(A_1,\ldots,A_p)\otimes (A_{p+1},\ldots,A_k) &\mbox{ if there exists a $p$ (necessarily unique)}\\
&\mbox{such that $A_1\sqcup \ldots \sqcup A_p=A$},\\
0\mbox{ otherwise}.
\end{cases}\]
The counit is given by $\varepsilon[\emptyset]=1$.
\end{enumerate}\end{example}

\begin{remark}
\begin{enumerate}
\item If $(\calP,m_\calP)$ and $(\calQ,m_\calQ)$ are two algebras in the category of species, 
then $\calP\otimes \calQ$ is too. Its product is given by:
\[m_{\calP\otimes \calQ}=(m_\calP\otimes m_\calQ)\circ (\id_\calP\otimes c_{\calP,\calQ}\otimes \id_\calQ).\]
Its unit is $1_\calP\otimes 1_\calQ$.
\item Dually, if $(\calP,\Delta_\calP)$ and $(\calQ,\Delta_\calQ)$ are two coalgebras in the category of species, 
then $\calP\otimes \calQ$ is too. Its coproduct is given by:
\[\Delta_{\calP\otimes \calQ}
=(\id_\calP\otimes c_{\calP,\calQ}\otimes \id_\calQ)\circ (\Delta_\calP\otimes \Delta_\calQ).\]
Its counit is $\varepsilon_\calP\otimes \varepsilon_\calQ$.
\end{enumerate}\end{remark}

\begin{lemma} \label{lemme3}
Let us assume that the twisted coalgebra $(\calP,\Delta)$ is connected. 
There exists a unique $1_{\calP}\in \calP[\emptyset]$
such that $\varepsilon(1_\calP)=1$. Moreover, for any finite set $A$, for any $x\in \calP[A]$:
\begin{align*}
\Delta_{A,\emptyset}(x)&=x\otimes 1_\calP,&
\Delta_{\emptyset,A}(x)&=1_\calP\otimes x.
\end{align*}
\end{lemma}

\begin{proof}
For any finite set $A$, $(\varepsilon\otimes \id)\circ \Delta_{\emptyset,A}=\id_{\calP[A]}$, 
so $\varepsilon\neq 0$. As $\calP[\emptyset]$ is one dimensional, $\varepsilon:\calP[\emptyset]\longrightarrow \K$
is bijective. Let $1_\calP$ be the unique element of $\calP[\emptyset]$ such that $\varepsilon(1_\calP)=1$.
For any $x\in \calP[A]$, there exists $y\in \calP[A]$ such that $\Delta_{\emptyset,A}(x)=1_\calP\otimes y$. Then:
\[(\varepsilon\otimes \id)\circ \Delta_{\emptyset,A}(x)=x=y,\]
so $\Delta_{\emptyset,A}(x)=1_\calP\otimes x$. Similarly, $\Delta_{A,\emptyset}(x)=x\otimes 1_\calP$.
\end{proof}

\begin{notation}
Let $\calP$ be a twisted coalgebra. If $A_1\sqcup \dots\sqcup A_k=A$, 
we inductively define $\Delta_{A_1,\ldots,A_k}$ by:
\begin{align*}
\Delta_A&=\id_{\calP[A]},\\
\Delta_{A_1,\ldots,A_{k+1}}&=(\Delta_{A_1,\ldots,A_k}\otimes \id_{\calP[A_{k+1}]})\circ 
\Delta_{A_1\sqcup \ldots \sqcup A_k,A_{k+1}} \mbox{ if }k\geq 1.
\end{align*}
By coassociativity, for any finite sets $A_1,\ldots,A_{k+l}$:
\[(\Delta_{A_1,\ldots,A_k}\otimes \Delta_{A_{k+1},\ldots,A_{k+l}})\circ \Delta_{A_1\sqcup \ldots \sqcup A_k,
A_{k+1}\sqcup \ldots \sqcup A_{k+l}}=\Delta_{A_1,\ldots,A_{k+l}}.\]
\end{notation}

As for "classical" bialgebras:

\begin{defi}
Let $\calP$ be a species, both a twisted  algebra $(\calP,m_\calP)$ and a twisted coalgebra $(\calP,\Delta_\calP)$. 
The following conditions are equivalent:
\begin{enumerate}
\item $\varepsilon_\calP:\calP\longrightarrow \I$ and $\Delta_\calP:\calP\longrightarrow \calP\otimes \calP$ 
are algebra morphisms.
\item $\iota:\I\longrightarrow \calP$ and $m:\calP\otimes \calP\longrightarrow \calP$ are coalgebra morphisms.
\end{enumerate}
If this holds, we shall say that $(\calP,m_\calP,\Delta_\calP)$ is a bialgebra in the category of species
or a twisted bialgebra \cite{Aguiar,PatrasReutenauer}.
\end{defi}

The compatibility between the product and coproduct can also be written in this way: 
if $A$ is a finite set and $A=I\sqcup J=I'\sqcup J'$,
\begin{align*}
\Delta_{I,J}\circ m_{I,'J'}&=(m_{I'\cap I,J'\cap I}\otimes m_{I'\cap J,J'\cap J})\circ (\id_{\calP[I\cap I]}
\otimes c_{\calP[I'\cap J],\calP[J'\cap I]}\otimes \id_{\calP[J'\cap J]})\\
&\circ (\Delta_{I'\cap I,I'\cap J}\otimes \Delta_{J'\cap I,J'\cap J}).
\end{align*}

\begin{example}
The species $\com$, with the product and coproduct defined below, is a twisted bialgebra.
\end{example}

\begin{remark}
If $\calP$ is a connected twisted bialgebra, 
that is to say if it is connected as a coalgebra, the elements $1_\calP$ defined
in Lemma \ref{lemme3} is necessarily the unit of the algebra $\calP$.
\end{remark}

\section{Cofreeness}

\subsection{Cofree coalgebras}

\begin{defi} \label{defi5}
Let $\calP$ be a species such that $\calP[\emptyset]=(0)$. The cofree coalgebra (co)generated by $\calP$ is the species:
\[\cot(\calP)=\bigoplus_{n=0}^\infty \calP^{\otimes n},\]
that is to say, for any finite set $A$:
\[\cot(\calP)[A]=\bigoplus_{(A_1,\ldots,A_k)\in \comp[A]} \calP[A_1]\otimes \ldots \otimes \calP[A_k].\]
For any $x_i\in \calP[A_i]$, $i\in \underline{k}$, we denote by $x_1\ldots x_k$ the tensor product of
the elements $x_1,\ldots,x_k$ in $\calP[A_1]\otimes \ldots \otimes \calP[A_k]
\subseteq \cot(\calP)[A_1\sqcup \ldots \sqcup A_k]$.
This species is given a coproduct making it a twisted coalgebra:
for any finite sets $A$, $B$, $(C_1,\ldots,C_k)\in \comp[A\sqcup B]$, $x_i \in \calP[A_i]$,
\[\Delta_{A,B}(x_1\ldots x_k)=\begin{cases}
x_1\dots x_i \otimes x_{i+1}\ldots x_k\mbox{ if there exists an $i$ (necessarily unique)} \\
\hspace{4cm} \mbox{such that }A_1\sqcup \ldots\sqcup A_i=A,\\
0\mbox{ otherwise}.
\end{cases}\]
We denote by $\pi:\cot(\calP)\longrightarrow \calP$ the canonical projection vanishing 
on $\calP^{\otimes k}$ if $k\neq 1$.
\end{defi}

More generally, we shall say that a twisted coalgebra is cofree if it is isomorphic to a 
coalgebra $\cot(\calP)$.

\begin{theo}\label{theo6}
Let $\calP$ be a species such that $\calP[\emptyset]=(0)$ and let $\calC$ be a connected twisted coalgebra. 
For any species morphism $\phi:\calC\longrightarrow \calP$,
there exists a unique coalgebra morphism $\Phi:\calC\longrightarrow \cot(\calP)$ such that the following diagram commutes:
\[\xymatrix{\calC\ar[r]^{\Phi} \ar[rd]_{\phi}&\cot(\calP)\ar[d]^{\pi}\\
&\calP}\]
Moreover, for any finite set $A$:
\begin{align}
\label{eq1} \Phi[A]&=\sum_{(A_1,\ldots,A_k)\in \comp[A]} (\phi[A_1]\otimes \ldots \otimes \phi[A_k])
\circ \Delta_{A_1,\ldots,A_k}.
\end{align}
\end{theo}

\begin{proof}
\textit{Existence}. Let $\Phi$ be the species morphism defined by (\ref{eq1}). 
For any finite set $A$:
\[\pi[A]\circ \Phi[A]=\phi[A]\circ \Delta_A=\phi[A],\]
so $\pi\circ \Phi=\phi$. Let $A$ and $B$ be two finite sets. Let us prove that:
\[\Delta_{A,B}\circ \Phi[A\sqcup B]=(\Phi[A]\otimes \Phi[B])\circ \Delta_{A,B}.\]
If $A=B=\emptyset$, it is enough to apply this two maps on $1_\calC$:
\begin{align*}
\Delta_{\emptyset,\emptyset}\circ \Phi(\emptyset\sqcup \emptyset)(1_\calC)=1\otimes 1
=(\Phi[\emptyset]\otimes \Phi[\emptyset])\circ \Delta_{\emptyset,\emptyset}(1_\calC).
\end{align*}
If $A=\emptyset$ and $B\neq \emptyset$, for any $x\in \calC(A)$:
\begin{align*}
\Delta_{\emptyset,B}\circ \Phi[B](x)=1\otimes \Phi[B](x)
=(\Phi[\emptyset] \otimes \Phi[B])\circ \Delta_{\emptyset,B}(x).
\end{align*}
A similar computation works if $B=\emptyset$ and $A\neq \emptyset$.
Let us assume now that $A,B \neq \emptyset$.
\begin{align*}
\Delta_{A,B}\circ \Phi[A\sqcup B]&=\Delta_{A,B}\circ \left(\sum_{(A_1,\ldots,A_k)\in \comp[A\sqcup B]} 
\phi^{\otimes k}\circ\Delta_{A_1,\ldots,A_k}\right)\\
&=\Delta_{A,B}\circ \left(\sum_{\substack{(A_1,\ldots,A_p)\in \comp[A],\\ (B_1,\ldots,B_q)\in \comp[B]}}
\phi^{\otimes (p+q)}\circ \Delta_{A_1,\ldots,A_p,B_1,\ldots,B_q}\right)\\
&=\sum_{\substack{(A_1,\ldots,A_p)\in \comp[A],\\ (B_1,\ldots,B_q)\in \comp[B]}}
\phi^{\otimes p}\circ \Delta_{A_1,\ldots,A_k}\otimes \phi^{\otimes q}\circ \Delta_{B_1,\ldots,B_q}\\
&=(\Phi[A]\otimes \Phi[B])\circ \Delta_{A,B}.
\end{align*}
So $\Phi$ is a coalgebra morphism.\\

\textit{Unicity}. By definition of $\Delta$, for any nonempty set $A$, in $\cot(\calP)[A]$:
\[\bigcap_{(I,J)\in \comp[A]} \ker(\Delta_{I,J})=\calP[A].\]
Let $\Phi$ and $\Phi'$ be two coalgebra morphisms satisfying the required conditions.
Let us prove that $\Phi[A]=\Phi'[A]$ by induction on $\sharp A$. If $A=\emptyset$,
as $1$ is the unique nonzero element of $\cot(\calP)[\emptyset]$ such that 
$\Delta_{\emptyset,\emptyset}(1)=1\otimes 1$:
\[\Phi[\emptyset](1_\calC)=\Phi'[\emptyset](1_\calC)=1.\]
As $\calC[\emptyset]=\K 1_\calC$, $\Phi[\emptyset]=\Phi'[\emptyset]$. 
Let us assume the result at all ranks $<\sharp A$. For any $(I,J)\in \comp[A]$:
\begin{align*}
\Delta_{I,J}\circ \Phi[A]=(\Phi[I]\otimes \Phi[J])\circ \Delta_{I,J}
=(\Phi'[I]\otimes \Phi'[J])\circ \Delta_{I,J}=\Delta_{I,J}\circ \Phi'[A].
\end{align*}
Hence, $\Phi[A]-\Phi'[A]$ takes its values in $\calP[A]$, so:
\[\Phi[A]-\Phi'[A]=\pi[A]\circ \Phi[A]-\pi[A]\circ \Phi'[A]=\phi[A]-\phi[A]=0.\]
Therefore, $\Phi=\Phi'$. \end{proof}

\begin{example} Let us consider the species $\calP$ defined by:
\[\calP[A]=\begin{cases}
\K\mbox{ if }\sharp A=1,\\
(0)\mbox{ otherwise}.
\end{cases}\]
The cofree coalgebra $\cot(\calP)$ is denoted by $\fqsym$. For any finite set $A$, 
noticing that compositions of $A$ which parts have cardinality $1$ are in bijection with bijections 
from $A$ to $\underline{\sharp A}$:
\[\fqsym(A)=\bigoplus_{\substack{f:A\longrightarrow \underline{\sharp A},\\
\mbox{\scriptsize bijection}}} \K^{\otimes \sharp A}.\]
Hence, $\fqsym(A)$ has a basis indexed by bijections from $A$ to $\underline{\sharp A}$, or equivalently
a basis indexed by total orders on $A$. If $A=I\sqcup J$ and $\leq$ is a total order on $A$:
\[\Delta_{I,J}(\leq)=\begin{cases}
\leq_{\mid I}\otimes \leq_{\mid J}\mbox{ if for any $x\in I$, $y\in J$, $x\leq y$},\\
0\mbox{ otherwise}.
\end{cases}\]
\end{example}

\subsection{Quasishuffle products}

\begin{notation}
\begin{enumerate}
\item Let $n_1,\ldots,n_k$ be integers. Let us denote by $\qsh(n_1,\ldots,n_k)$ the set of 
$(n_1,\ldots,n_k)$-quasi-shuffles, that is to say surjective map
$\sigma:\underline{n_1+\ldots+n_k}\longrightarrow \underline{\max(\sigma)}$, such that for all $i\in \underline{k}$, 
$\sigma_{\mid\{n_1+\ldots+n_{i-1}+1,\ldots,n_1+\ldots+n_i\}}$ is non decreasing.
Bijective quasi-shuffles are called shuffles; the set of $(n_1,\ldots,n_k)$-shuffles is denoted by 
$Sh(n_1,\ldots,n_k)$.
\item Let $\calP$ be a species. We consider the species $\calP_+$ defined by:
\[\calP_+[A]=\begin{cases}
(0)\mbox{ if }A=\emptyset,\\
\calP[A]\mbox{ if }A\neq \emptyset.
\end{cases}\]
Note that if $\calP$ is a twisted algebra, then $\calP_+$ is stable under the product of $\calP$.
\item Let $\calP$ be a twisted algebra.
Let $(A_1,\ldots,A_k)\in \comp[A]$, where $A$ is a finite set, 
and let $\sigma:\underline{k}\longrightarrow \underline{\max(\sigma)}$ be a surjection.
For any $x_i\in \calP[A_i]$, $i\in \underline{k}$, we put:
\[\sigma\rightarrow x_1 \ldots x_k
=\left( \prod^{m}_{i,\sigma(i)=1}x_i\right)\ldots \left(\prod^m_{i,\sigma(i)=\max(\sigma)} x_i\right),\]
where each product is taken in $\calP$. This is an element of $\calP^{\otimes \max(\sigma)}[A]$.
\end{enumerate} \end{notation}

The construction of quasi-shuffle products \cite{Hoffman,FoissyPatrasShuffle} extends to the category of species:

\begin{theo}
Let $\calP$ be a twisted algebra. 
The cofree coalgebra $\cot(\calP_+)$ is given an associative product $\squplus$: 
if, for any $i\in \underline{k+l}$, $x_i \in \calP[A_i]$, then
\[x_1\ldots x_k\squplus x_{k+1}\ldots x_{k+l}=\sum_{\sigma \in \qsh(k,l)} \sigma \rightarrow (x_1\ldots x_{k+l}).\]
Then $(\cot(\calP_+),\squplus,\Delta)$ 
is a twisted bialgebra. Moreover, $\squplus$ is commutative if, and only if,
$m$ is commutative.
\end{theo}

\begin{proof}
Let $\calC$ be the coalgebra $\cot(\calP_+)\otimes \cot(\calP_+)$. Then $\calC[\emptyset]=\K 1\otimes 1$:
$\calC$ is connected. 
We shall identify the unit $1$ of $\K$ (also unit of $\cot(\calP_+)$) and the unit of $\calP$.
We consider the species morphism $\phi:\calC\longrightarrow \calP_+$ defined by:
\begin{align*}
\phi_{\mid \calP^{\otimes k}\otimes \calP^{\otimes l}}&=\begin{cases}
m\mbox{ if }(k,l)\in\{(0,1),(1,0),(1,1)\},\\
0\mbox{ otherwise.}
\end{cases}
\end{align*}
There exists a unique coalgebra $\squplus:\calC\longrightarrow \cot(\calP_+)$ such that $\pi\circ \squplus=\phi$.
Let us prove that $\squplus$ is associative. 
Let us consider the coalgebra morphisms $\squplus\circ (\squplus\otimes \id)$ 
and $\squplus\circ (\id \otimes \squplus)$ from $\cot(\calP_+)^{\otimes 3}$
to $\cot(\calP_+)$. For any $k,l,p\geq 0$:
\begin{align*}
\pi\circ \squplus\circ(\squplus\otimes \id)_{\mid \calP^{\otimes k}\otimes \calP^{\otimes l}\otimes \calP^{\otimes p}}
&=\begin{cases}
m\circ (m\otimes \id)\mbox{ if $k,l,p\leq 1$ and $(k,l,p)\neq (0,0,0)$},\\
0\mbox{ otherwise};
\end{cases}\\
\pi\circ \squplus\circ(\id\otimes \squplus)_{\mid \calP^{\otimes k}\otimes \calP^{\otimes l}\otimes \calP^{\otimes p}}
&=\begin{cases}
m\circ (\id\otimes m)\mbox{ if $k,l,p\leq 1$ and $(k,l,p)\neq (0,0,0)$},\\
0\mbox{ otherwise}.
\end{cases}
\end{align*}
As $m$ is associative, $\pi\circ \squplus\circ(\squplus\otimes \id)=\pi\circ \squplus\circ(\id\otimes \squplus)$.
By unicity in Theorem \ref{theo6}, $\squplus\circ(\squplus\otimes \id)=\squplus\circ(\id\otimes \squplus)$.
So $(\cot(\calP_+),\squplus,\Delta)$ is a twisted bialgebra.\\

Let $x_i\in \calP[A_i]$ for all $i$. We put $A=A_1\sqcup \ldots \sqcup A_{k+l}$. By Theorem \ref{theo6}:
\begin{align*}
&\squplus(x_1\ldots x_k\otimes x_{k+1}\ldots x_{k+l})\\
&=\sum_{(B_1,\ldots,B_n)\in \comp[A]} \phi^{\otimes n}\circ \Delta_{A_1,\ldots,A_n}(x_1\ldots x_k \otimes
x_{k+1}\ldots x_{k+l})\\
&=\sum_{\substack{1=i_0\leq i_1\leq \ldots \leq i_p=k,\\1=j_0\leq j_1\leq \ldots \leq j_p=l}}
\phi\left(x_1\ldots x_{i_1}\otimes x_{k+1}\ldots x_{k+j_1}\right)\ldots
\phi\left(x_{i_{p-1}+1}\ldots x_{i_p}\otimes x_{k+j_{p-1}+1}\ldots x_{k+j_p}\right)\\
&=\sum_{\substack{1=i_0\leq i_1\leq \ldots \leq i_p=k,\\1=j_0\leq j_1\leq \ldots \leq j_p=l,\\
\forall s,\: i_{s+1}-i_s\leq 1,\\
\forall t,\: j_{t+1}-i_t\leq 1}}
\phi\left(x_1\ldots x_{i_1}\otimes x_{k+1}\ldots x_{k+j_1}\right)\ldots
\phi\left(x_{i_{p-1}+1}\ldots x_{i_p}\otimes x_{k+j_{p-1}+1}\ldots x_{k+j_p}\right)\\
&=\sum_{\sigma \in \qsh(k,l)} \sigma \rightarrow (x_1\ldots x_{k+l}).
\end{align*}

Let us assume that $\squplus$ is commutative. For any finite set $A$ and $B$:
\begin{align*}
\squplus_{B,A}\circ c_{\cot(\calP_+)[A],\cot(\calP_+)[B]}&=\squplus_{A,B},\\
\pi[A\sqcup B]\circ \squplus_{B,A}\circ c_{\calP[A],\calP[B]}
&=(\pi[A\sqcup B]\circ \squplus_{A,B} )_{\mid \calP[A]\otimes \calP[B]},\\
m_{B,A}\circ c_{\calP[A],\calP[B]}&=m_{A,B}.
\end{align*}
So $m$ is commutative.\\

Let us assume that $m$ is commutative. Then $\squplus$ and $\squplus\circ  c$ are both coalgebra morphisms from 
$\cot(\calP_+)^{\otimes 2}$ to $\cot(\calP_+)$ and, as $m$ is commutative, $\pi\circ \squplus\circ c=\pi\circ m=\phi$.
By unicity in Theorem \ref{theo6}, $\squplus=\squplus\circ c$. \end{proof}

\begin{example}\begin{enumerate}
\item Let us consider the bialgebra $\com$. The cofree coalgebra $\cot(\com_+)$ 
is a twisted bialgebra. For any finite set $A$:
\[\cot(\com_+)[A]=\bigoplus_{(A_1,\ldots,A_k)\in \comp[A]} \com[A_1]\otimes \ldots \otimes \com[A_k].\]
The element $1\otimes \ldots \otimes 1$ of  $\com[A_1]\otimes \ldots \otimes \com[A_k]$
will be identified with $(A_1,\ldots,A_k) \in \comp[A]$: this defines a species isomorphism
between $\cot(\com_+)$ and $\calcomp$. The product $\squplus$ induced on $\calcomp$ is given by:
\[(A_1,\ldots,A_k)\squplus (B_1,\ldots,B_k)=\sum_{\sigma \in \qsh(k,l)} 
\left(\bigcup_{i \in \sigma^{-1}(1)} A_i\right)\ldots
\left(\bigcup_{i \in \sigma^{-1}(\max(\sigma))} A_i\right).\]
The coproduct is given by:
\[\Delta_{A,B}(A_1,\ldots,A_k)=\begin{cases}
(A_1,\ldots,A_p)\otimes (A_{p+1},\ldots,A_k) &\mbox{ if there exists $p$ (necessarily unique)}\\
&\mbox{such that $A_1\sqcup \ldots \sqcup A_p=A$},\\
0\mbox{ otherwise}.
\end{cases}\]
The counit is given by $\varepsilon_{\calcomp}[\emptyset]=1$.
For example, if $A$, $B$, $C$, $D$ are nonempty finite sets:
\begin{align*}
(A)\squplus (B)&=(A,B)+(B,A)+(A\sqcup B),\\
(A,B)\squplus (C)&=(A,B,C)+(A,C,B)+(C,A,B)+(A\sqcup B,C)+(A,B\sqcup C),\\
(A,B)\squplus (C,D)&=(A,B,C,D)+(A,C,B,D)+(C,A,B,D)\\
&+(A,C,D,B)+(C,A,D,B)+(C,D,A,B)\\
&+(A,B\sqcup C,D)+(A\sqcup C,B,D)+(A\sqcup C,D,B)\\
&+(A,C,B\sqcup D)+(C,A,B\sqcup D)+(C,A\sqcup D,B)\\
&+(A\sqcup C,B\sqcup D).
\end{align*}
\item Let $\calP$ be any species such that $\calP[\emptyset]=(0)$. 
We give it a trivial product $m=0$. If for all $i\in \underline{k+l}$, $x_i\in\calP[A_i]$ and $\sigma \in \qsh(k,l)$,
then $\sigma\rightarrow (x_1\ldots x_{k+l})=0$ if $\sigma$ is not bijective, that is to say if
$\sigma$ is not a shuffle. Hence, the product induced by $m$ on $\cot(\calP)$ is the shuffle product:
\[\shuffle(x_1\ldots x_k\otimes x_{k+1}\ldots x_{k+l})
=\sum_{\sigma \in Sh(k,l)} \sigma \rightarrow (x_1\ldots x_{k+l}).\]
We obtain the shuffle product of \cite{Aguiar3}.  In particular, this holds for $\fqsym$.
If $A$ and $B$ are finite sets and $\leq_A$, $\leq_B$ are total orders on respectively $A$ and $B$,
for any $\sigma \in \qsh(\sharp A,\sharp B)$, $\sigma \rightarrow (\leq_A \leq_B)$
is nonzero if, and only if, $\sigma$ is a bijective, that is to say if, and only if, $\sigma$ is a shuffle.
Consequently:
\[\leq_A\squplus \leq_B=\sum_{\substack{\mbox{\scriptsize $\leq$ total order on $A\sqcup B$},\\ 
\leq_{\mid A}=\leq_A,\: \leq_{\mid_B}=\leq_B}} \leq.\]
\end{enumerate} \end{example}

\begin{prop}
Let $(\calP,m)$ be a twisted algebra. 
The map $\varepsilon':\cot(\calP_+)\longrightarrow \calP$ defined by $\varepsilon'[\emptyset](1)=1_\calP$ and
$\varepsilon'[A]=\pi[A]$ if $A\neq \emptyset$ is an algebra morphism.
\end{prop}

\begin{proof}
As $1$ is the unit of $\squplus$, it is enough to prove that:
\[\varepsilon'[A\sqcup B]\circ \squplus_{A,B}=m_{A,B}\circ (\varepsilon'[A]\otimes \varepsilon'[B]),\]
for any nonempty sets $A$ and $B$. Let $(A_1,\ldots,A_k)\in \comp[A]$, 
$(B_1,\ldots,B_l)\in \comp[B]$, $x_i \in \calP[A_i]$, $x_{k+i}\in \calP[B_i]$. By construction of $\squplus$:
\begin{align*}
\varepsilon'[A\sqcup B]\circ \squplus_{A,B}(x_1\ldots x_k\otimes x_{k+1}\ldots x_{k+l})
&=\begin{cases}
m(x_1\otimes x_2)\mbox{ if }k=l=1,\\
0\mbox{ otherwise};
\end{cases}\\
&=m_{A,B}\circ (\varepsilon'[A]\otimes \varepsilon'[B])(x_1\ldots x_k\otimes x_{k+1}\ldots x_{k+l}).
\end{align*}
So $\varepsilon'$ is a twisted algebra morphism.
\end{proof}

As $\calcomp$ is isomorphic to $\cot(\com_+)$:

\begin{cor}\label{cor9}
The following map is an algebra morphism from $(\calcomp,\squplus)$ to $(\com,m)$:
\[\varepsilon':\left\{\begin{array}{rcl}
\calcomp[A]&\longrightarrow&\com[A]\\
(A_1,\ldots,A_k)&\longrightarrow&\delta_{k,1}.
\end{array}\right.\]
\end{cor}

\begin{theo} \label{theo10}
Let $\calP$ be a twisted algebra and 
let $\calB$ be a connected twisted bialgebra.
For any species morphism $\phi:\calB\longrightarrow \calP$, there exists a unique twisted coalgebra morphism
$\Phi$ from $\calB$ to  $(\cot(\calP_+),\squplus,\Delta)$ such that $\varepsilon'\circ \Phi=\phi$.
For any finite set $A$:
\[\Phi[A]=\sum_{(A_1,\ldots,A_k)\in \comp[A]} \phi^{\otimes k}\circ \Delta_{A_1,\ldots,A_k}.\]
Moreover, $\Phi$ is a twisted bialgebra morphism if, and only if, $\phi$ is  twisted algebra morphism.
\end{theo}

\begin{proof}
Let $\phi'$ defined by:
\[\phi'[A]=\begin{cases}
0\mbox{ if }A=\emptyset,\\
\phi[A]\mbox{ if }A\neq \emptyset.
\end{cases}\]
For any coalgebra morphism $\Phi:\calB\longrightarrow \cot(\calP_+)$, $\Phi(1_{\calB})=1$. 
As $\calB[\emptyset]=\K 1_C$, $\varepsilon'\circ \Phi=\phi$ if, and only if, $\pi\circ \Phi=\phi'$.
By Theorem \ref{theo6}, such a morphism exists and is unique.\\

Let us assume that $\Phi$ is a bialgebra morphism. By composition, $\phi=\varepsilon'\circ \Phi$
is an algebra morphism. Let us assume that $\phi$ is an algebra morphism. Let us consider the coalgebra morphisms
$\Phi \circ m$ and $\squplus\circ (\Phi \otimes \Phi)$, both from $\calB\otimes \calB$ to $\cot(\calP_+)$.
\begin{align*}
\varepsilon'\circ \Phi\circ m&=\phi\circ m,\\
\varepsilon'\circ \squplus\circ (\Phi \otimes \Phi)&=m_\calP\circ (\varepsilon'\otimes \varepsilon')\circ (\Phi \otimes \Phi)\\
&=m_\calP \circ (\phi\otimes \phi)\\
&=\phi\circ m.
\end{align*}
Consequently, $\varepsilon'\circ \Phi\circ m=\varepsilon'\circ \squplus\circ (\Phi \otimes \Phi)$,
so $\pi\circ \Phi\circ m=\pi\circ \squplus\circ (\Phi \otimes \Phi)$. By unicity in Theorem \ref{theo6},
$\Phi\circ m=\squplus\circ (\Phi \otimes \Phi)$, so $\Phi$ is a bialgebra morphism. \end{proof}

\subsection{A criterion of cofreeness}

Let us give here a twisted version of Loday and Ronco's rigidity theorem for bialgebras \cite{LodayRonco}.

\begin{theo} \label{theo11}
Let $(\calP,\Delta)$ be a connected twisted coalgebra. 
We denote by $1$ the unique nonzero element of $\calP[\emptyset]$
such that $\Delta_{\emptyset,\emptyset}(1)=1\otimes 1$. The following propositions are equivalent:
\begin{enumerate}
\item $(\calP,\Delta)$ is isomorphic to a cofree coalgebra $\cot(\calQ)$.
\item There exists a structure of twisted algebra $(\calP,m)$ on $\calP$, of unit $1$, such that
for any finite sets $I$, $J$, $I'$ and $J'$ such that $I\sqcup J=I'\sqcup J'$, for any $x\in \calP[I']$,
$y \in \calP[J']$:
\[\Delta_{I,J}\circ m_{I',J'}(x\otimes y)=
\begin{cases}
(m_{I',I\setminus I}\otimes \id)(x\otimes \Delta_{I\setminus I',J}(y))\mbox{ if }I'\subseteq I,\\
(\id\otimes m_{J\setminus J',J'})(\Delta_{I,J\setminus J'}(x)\otimes y) \mbox{ if }J'\subseteq J,\\
0\mbox{ otherwise}.
\end{cases}\]
\end{enumerate}
\end{theo}

\begin{proof} $1\Longrightarrow 2$. Up to an isomorphism, we assume that the coalgebras $(\calP,\Delta)$ and
$\cot(\calQ)$ are equal. Let us denote by $m$ the concatenation product of $\calP$. 
Let us consider $x_i\in \calQ[A_i]$ for all $i\in \underline{k+l}$. Putting
$I'=A_1\sqcup \ldots \sqcup A_k$ and $J'=A_{k+1}\sqcup \ldots \sqcup A_{k+l}$, if $I\sqcup J=I'\sqcup J'$:
\begin{align*}
&\Delta_{I,J}\circ m_{I',J'}(x_1\ldots x_k\otimes x_{k+1}\ldots x_{k+l})\\
&=\Delta_{I,J}(x_1\ldots x_{k+l})\\
&=\begin{cases}
x_1\ldots x_j \otimes x_{j+1}\ldots x_{k+l}\mbox{ if there exists $j$ such that }A_1\sqcup \ldots \sqcup A_j=I,\\
0\mbox{ otherwise.}
\end{cases}
\end{align*}
Three cases are possible.
\begin{itemize}
\item Such a $j$ exists and $j\leq k$. Then $J'\subseteq J$ and:
\begin{align*}
&\Delta_{I,J}\circ m_{I',J'}(x_1\ldots x_k\otimes x_{k+1}\ldots x_{k+l})\\
&=x_1\ldots x_j \otimes x_{j+1}\ldots x_k.x_{k+1}\ldots x_{k+l}\\
&=(\id\otimes m_{J\setminus J',J'})(\Delta_{I,J\setminus J'}(x_1\ldots x_k)\otimes x_{k+1}\ldots x_{k+l}).
\end{align*}
\item Such a $j$ exists and $j\geq k$. Then $I'\subseteq I$ and:
\begin{align*}
&\Delta_{I,J}\circ m_{I',J'}(x_1\ldots x_k\otimes x_{k+1}\ldots x_{k+l})\\
&=x_1\ldots x_k.x_{k+1}\ldots x_l \otimes x_{j+1}\ldots x_{k+l}\\
&=(m_{I',I\setminus I}\otimes \id)(x_1\ldots x_k\otimes \Delta_{I\setminus I',J}(x_{k+1}\ldots x_{k+l}))
\end{align*}
\item No such $j$ exists. Then $I'\nsubseteq I$, $J'\nsubseteq J$ and:
\begin{align*}
\Delta_{I,J}\circ m_{I',J'}(x_1\ldots x_k\otimes x_{k+1}\ldots x_{k+l})&=0.
\end{align*} \end{itemize}

$2\Longrightarrow 1$. Let $\calQ$ be the subspecies of $\calP$ defined by:
\[\calQ[A]=\{x\in \calP[A],\forall (I,J)\in \comp(A), \Delta_{I,J}(x)=0\}.\]
We shall consider the following morphism:
\[
\Phi:\left\{\begin{array}{rcl}
\cot(\calQ)&\longrightarrow&\calP\\
x_1\ldots x_k\in \calQ[A_1]\otimes \ldots \calQ[A_k]&\longrightarrow&
m_{A_1,\ldots,A_k}(x_1\otimes \ldots \otimes x_k).
\end{array}\right.
\]
Let us first prove that $\Phi$ is a coalgebra morphism. Let $x_1\ldots x_k\in \calQ[A_1]\otimes \ldots \calQ[A_k]$.
\begin{itemize}
\item If $I=A_1\sqcup \ldots \sqcup A_i$ for a certain $i$:
\begin{align*}
\Delta_{I,J}\circ \Phi(x_1\ldots x_k)&=\Delta_{I,J}(x_1\cdot\ldots\cdot x_k)\\
&=\Delta_{I,J}((x_1\cdot \ldots \cdot x_i)\cdot (x_{i+1}\cdot \ldots \cdot x_k))\\
&=(m_{I\emptyset}\otimes \id)(x_1\cdot \ldots \cdot x_i\otimes \Delta_{\emptyset,J}(x_{i+1}\cdot \ldots \cdot x_k))\\
&=(m_{I\emptyset}\otimes \id)(x_1\cdot \ldots \cdot x_i\otimes 1\otimes x_{i+1}\cdot \ldots \cdot x_k)\\
&=x_1\cdot \ldots \cdot x_i\otimes x_{i+1}\cdot \ldots \cdot x_k\\
&=(\Phi \otimes \Phi)\circ \Delta_{I,J}(x_1\ldots x_k).
\end{align*}
\item Otherwise, let $i$ be the greatest integer such that $A_1\sqcup\ldots \sqcup A_i\subseteq I$.
Then $A_{i+1}\cap I\neq \emptyset$.
\begin{align*}
\Delta_{I,J}\circ \Phi(x_1\ldots x_k)&=\Delta_{I,J}(x_1\cdot\ldots\cdot x_k)\\
&=(x_1\cdot \ldots \cdot x_i\otimes 1)\cdot \Delta_{I\setminus A_1\sqcup \ldots \sqcup A_i,J}(x_{i+1}\cdot \ldots
\cdot x_k).
\end{align*}
If $I\nsubseteq A_1\sqcup \ldots A_{i+1}$, this is zero. Otherwise, $I\cap A_{i+1}\neq \emptyset$. Then:
\begin{align*}
\Delta_{I,J}\circ \Phi(x_1\ldots x_k)&=\Delta_{I,J}(x_1\cdot\ldots\cdot x_k)\\
&=(x_1\cdot \ldots \cdot x_i\otimes 1)\cdot \Delta_{I\setminus A_1\sqcup \ldots \sqcup A_i,J}(x_{i+1}\cdot \ldots
\cdot x_k)\\
&=(x_1\cdot \ldots \cdot x_i\otimes 1)\cdot \Delta_{A_{i+1}\cap I,A_{i+1}\cap J}(x_{i+1})
\cdot(1\otimes x_{i+2}\cdot\ldots \cdot x_k).
\end{align*}
By definition of $\calQ$, $\Delta_{A_{i+1}\cap I,A_{i+1}\cap J}(x_{i+1})=0$. Finally:
\begin{align*}
\Delta_{I,J}\circ \Phi(x_1\ldots x_k)=0&=(\Phi\otimes \Phi)\circ \Delta_{I,J}(x_1\ldots x_k).
\end{align*}
\end{itemize}
Hence, $\Phi$ is a coalgebra morphism.\\

Let us now prove that $\Phi$ is injective. Let $x \in \cot(\calQ)[A]$, nonzero, such that $\Phi[A](x)=0$.
We assume that $\sharp A$ is minimal. If $A=I\sqcup J$, with $I\neq \emptyset$ and $J\neq \emptyset$:
\begin{align*}
(\Phi[I]\otimes \Phi[J])\circ \Delta_{I,J}(x)&=\Delta_{I,J}\circ \Phi[A](x)=0.
\end{align*}
By hypothesis on $A$, $\Phi[I]$ and $\Phi[J]$ are injective, so $\Delta_{I,J}(x)=0$, and 
$x\in \calQ[A]$. Then $\Phi[A](x)=x\neq 0$: this is a contradiction. Therefore, $\Phi$ is injective.\\

Let us finally prove that $\Phi$ is surjective. As $\calP[\emptyset]=\K 1$, $\calP[\emptyset]=\Phi[\emptyset]
(\cot(\calP)[\emptyset])$. Let $A$ be a nonempty set. For any $x \in \calP[A]$, we put:
\[k(x)=\max\{l,\exists (I_1,\ldots,I_l)\in \comp(A), \Delta_{I_1,\ldots,I_l}(x)\neq 0\}.\]
Obviously, $k(x)$ exists and $k(x)\leq \sharp A$. Let us prove that $x\in \Phi[A](\cot(\calQ)[A])$
by induction on $k(x)$. If $k(x)=1$, then $x\in \calQ[A]$ and $x=\Phi[A](x)$. If $k=k(x)\geq 2$, we put,
for any $I=(I_1,\ldots,I_k)\in \comp(A)$:
\begin{align*}
\Delta_{I_1,\ldots,I_k}(x)=x_1^{(I)}\otimes \ldots \otimes x_k^{(I)}.
\end{align*}
By maximality of $k$, $x_i^{(I)}\in \calQ$ for any $I$, $i$. 
We consider:
\[y=\sum_{I=(I_1,\ldots,I_k)\in \comp(A)} x_1^{(I)}\cdot \ldots \cdot x_k^{(I)}.\]
Then, if $l>k$ and $(I_1,\ldots,I_l)\in \comp(A)$, $\Delta_{I_1,\ldots,I_l}(y)=0$.
If $I=(I_1,\ldots,I_k)\in \comp(A)$, 
$\Delta_{I_1,\ldots,I_l}(y)=x_1^{(I)}\otimes \ldots \otimes x_k^{(I)}=\Delta_{I_1,\ldots,I_k}(x)$.
Hence, $k(x-y)<k$, so $x-y\in \Phi[A](\cot(\calQ)[A])$. As $y\in \Phi[A](\cot(\calQ)[A])$,
$x \in \Phi[A](\cot(\calQ)[A])$. \end{proof}

\section{Convolution and characters}

\subsection{Convolution}

As for "usual" algebras and coalgebras:

\begin{prop}
Let $(\calA,m_\calA)$ be an algebra and $(\calC,\Delta_\calC)$ be a coalgebra, both in the category of species. 
The space $\Endo(\calC,\calA)$ of species morphisms from $\calC$ to $\calA$ is a monoid for the convolution product $*$:
\begin{align*}
&\forall f,g \in \Endo(\calC,\calA),&f*g=m_\calA\circ (f\otimes g)\circ \Delta_\calC.
\end{align*}
The unit is the morphism $u=\iota_\calA\circ \varepsilon_\calC$.
\end{prop}

\begin{lemma}\label{lemme13}
Let $\calA$ be a commutative twisted algebra and $\calB$ be a twisted bialgebra. 
We denote by $\chara(\calB,\calA)$ the set of algebra morphisms from $\calB$ to $\calA$.
\begin{enumerate}
\item $\chara(\calB,\calA)$ is a submonoid of $(\Endo(\calB,\calA),*)$.
\item Let $f\in \chara(\calB,\calA)$, with an inverse $g\in \Endo(\calB,\calA)$. Then $g\in \chara(\calB,\calA)$.
\end{enumerate}
\end{lemma}

\begin{proof}
1. For any $f,g\in \chara(\calB,\calA)$:
\[f*g(1_{\calB})=f(1_{\calB})g(1_{\calB})=1_\calA1_\calA=1_\calA.\]
Moreover:
\begin{align*}
m_\calA\circ (f*g\otimes f*g)&=m_\calA\circ (m_\calA\otimes m_\calA)\circ (f\otimes g\otimes f\otimes g)
\circ (\Delta_{\calB}\otimes \Delta_{\calB}),\\
f*g\circ m_{\calB}&=m_\calA\circ (f\otimes g)\circ (\Delta_{\calB}\circ m_{\calB})\\
&=m_\calA\circ (f\otimes g)\circ (m_{\calB}\otimes m_{\calB})\circ (\id_{\calB}\otimes c_{\calB,\calB}\otimes \id_{\calB})
\circ (\Delta_{\calB}\otimes \Delta_{\calB})\\
&=m_\calA\circ (m_\calA\otimes m_\calA)\circ (f\otimes f\otimes g\otimes g)\circ  (\id_{\calB}\otimes c_{\calB,\calB}
\otimes \id_{\calB})\circ (\Delta_{\calB}\otimes \Delta_{\calB})\\
&=m_\calA\circ (m_\calA\otimes m_\calA)\circ  (\id_\calA\otimes c_{\calA,\calA}\otimes \id_\calA)
\circ (f\otimes g\otimes f\otimes g)\circ (\Delta_{\calB}\otimes \Delta_{\calB}).
\end{align*}
As $\calA$ is commutative, $m_\calA\circ (f*g\otimes f*g)=f*g\circ m_{\calB}$. So $f*g\in \chara(\calB,\calA)$.\\

2. As $\calB$ is a bialgebra, $\calB\otimes \calB$ is a bialgebra. Let us consider the maps 
$F=f\circ m_{\calB}=m_\calA\circ (f\otimes f)$,
$G=g\circ m_{\calB}$ and $H=m_\calA\circ (g\otimes g)$, all three in $\Endo(\calB\otimes \calB,\calA)$.
\begin{align*}
F*G&=m_\calA\circ (f\otimes g)\circ (m_{\calB}\otimes m_{\calB})\circ (\id_{\calB}\otimes c_{\calB,\calB}\otimes 
\id_{\calB})\circ (\Delta_{\calB}\otimes \Delta_{\calB})\\
&=m_\calA\circ (f\otimes g)\circ \Delta_{\calB}\circ m_{\calB}\\
&=f*g \circ m_{\calB}\\
&=\iota_\calA\circ \varepsilon_{\calB} \circ m_{\calB}\\
&=\iota_\calA\circ m_\I \circ (\varepsilon_{\calB}\otimes \varepsilon_{\calB})\\
&=\iota_\calA\circ \varepsilon_{\calB\otimes \calB}.
\end{align*}
So $G$ is a right inverse of $F$.
\begin{align*}
H*F&=m_\calA\circ (m_\calA\otimes m_\calA)\circ (f\otimes f\otimes g\otimes g)\circ (\id_{\calB}\otimes c_{\calB,\calB}
\otimes \id_{\calB})\circ (\Delta_{\calB}\otimes \Delta_{\calB})\\
&=m_\calA\circ (m_\calA\otimes m_\calA)\circ (\id_\calA\otimes c_{\calA,\calA}\otimes \id_\calA)
\circ (f\otimes g\otimes f\otimes g)\circ (\Delta_{\calB}\otimes \Delta_{\calB})\\
&=m_\calA\circ (m_\calA\otimes m_\calA)\circ(f\otimes g\otimes f\otimes g)\circ (\Delta_{\calB}\otimes \Delta_{\calB})\\
&=m_\calA\circ (f*g \otimes f*g)\\
&=m_\calA\circ (\iota_\calA\otimes \iota_\calA)\circ (\varepsilon_{\calB}\otimes \varepsilon_{\calB})\\
&=\iota_\calA \circ m_\I \circ (\varepsilon_{\calB}\otimes \varepsilon_{\calB})\\
&=\iota_\calA\circ \varepsilon_{\calB\otimes \calB}.
\end{align*}
So $H$ is a left inverse of $F$. Therefore, $H=H*(F*G)=(H*F)*G=G$, so $g\circ m_{\calB}=m_\calA\circ (g\otimes g)$. Finally:
\[f*g(1_{\calB})=1_\calA=f(1_{\calB})g(1_{\calB})=g(1_{\calB}),\]
Hence, $g(1_{\calB})=1_\calA$. \end{proof}

\subsection{Characters}

By Lemma \ref{lemme13}:

\begin{defi} \label{defi14}
Let $\calP$ be a twisted algebra. A character on $\calP$ is an algebra morphism from $\calP$ to $\com$.
The set of characters on $\calP$ is denoted by $\chara(\calP)$ instead of $\chara(\calP,\com)$.
Then $\chara(\calP)$ is given a monoid structure in the following way:
\begin{align*}
&\forall \alpha,\beta \in \chara(\calP),& \alpha*\beta&=m_{\com}\circ (\alpha \otimes \beta)\circ \Delta.
\end{align*}
The unit of this monoid is the counit $\varepsilon:\calP\longrightarrow \I\longrightarrow \com$, where the second
arrow is the canonical injection from $\I$ into $\com$.
\end{defi}

\begin{remark}
We proved in Corollary \ref{cor9} that $\varepsilon'$ is a character of $\calcomp$.
\end{remark}

Let $\calP$ be a twisted bialgebra and let $\lambda:\calP\longrightarrow \com$ be a morphism
such that $\lambda[\emptyset]=0$. Then, for any $k\in \N$,
\[\lambda^{*k}[A]=\sum_{(I_1,\ldots,I_k)\in \comp[A]}\lambda^{\otimes k}\circ \Delta_{I_1,\ldots,I_k}(x).\]
If $k>\sharp A$, this is zero: $\lambda$ is locally nilpotent for the convolution product $*$. We obtain an algebra morphism:
\[\left\{\begin{array}{rcl}
\K[[X]]&\longrightarrow&\Endo(\calP,\calcomp)\\
\displaystyle f(X)=\sum_{n=0}^\infty a_nX^n&\longrightarrow&\displaystyle f(\lambda)=\sum_{n=0}^\infty a_n\lambda^{*n}.
\end{array}\right.\]
If $\calP$ is connected, then for any character $\lambda \in \chara(\calP)$, $(\lambda-\varepsilon)[\emptyset]=0$,
so $f(\lambda-\varepsilon)$ exists for any $f\in \K[[X]]$. In particular, for any $q\in \K$, we put:
\[\lambda^q=(1+(\lambda-\varepsilon))^q=\sum_{n=0}^\infty H_n(q)(\lambda-\varepsilon)^{*n},\]
where for any $n\geq 0$, $H_n$ is the $n$-th Hilbert polynomial:
\[H_n(X)=\frac{X(X-1)\ldots (X-n+1)}{n!}.\]
For any $q,q'\in \K$, $(1+X)^q (1+X)^{q'}=(1+X)^{q+q'}$, so:
\[\lambda^q*\lambda^{q'}=\lambda^{q+q'}.\]
Moreover, if $f_q=(1+X)^q-1$, $(1+f_q(X))^{q'}=(1+X)^{qq'}$, so:
\[(\lambda^q)^{q'}=\lambda^{qq'}.\]

\begin{prop}
Let $\calP$ be a connected twisted bialgebra and let $\lambda \in \chara(\calP)$.
For any $q\in \K$, $\lambda^q\in \chara(\calP)$.
\end{prop}

\begin{proof}
We define a species morphism:
\[M:\left\{\begin{array}{rcl}
L(\calP,\com)&\longrightarrow&L(\calP\otimes \calP,\com)\\
\lambda&\longrightarrow&\lambda \circ m.
\end{array}\right.\]
Let $\lambda,\mu\in L(\calP,\com)$. For any $x\in \calP[A]$, $y\in \calP[B]$:
\begin{align*}
M(\lambda)*M(\mu)(x\otimes y)&=M(\lambda)\otimes M(\mu)
(x^{(1)}\otimes y^{(1)}\otimes x^{(2)}\otimes y^{(2)})\\
&=\lambda(x^{(1)}y^{(1)})\mu(x^{(2)}y^{(2)})\\
&=\lambda*\mu(xy)\\
&=M(\lambda*\mu)(x\otimes y).
\end{align*}
So $M$ is an algebra morphism.
Moreover, if $\lambda \chara(\calP)$, then $M(\lambda)=\lambda \otimes \lambda$, so:
\begin{align*}
M(\ln(1+(\lambda-\varepsilon))&=\ln(1+\lambda\otimes \lambda-\varepsilon\otimes \varepsilon)\\
&=\ln(1+\lambda-\varepsilon)\otimes \varepsilon+\varepsilon\otimes \ln(1+\lambda-\varepsilon),\\
M(\lambda^q)&=\exp(q\ln(1+\lambda-\varepsilon)\otimes \varepsilon+\varepsilon\otimes q\ln(1+\lambda-\varepsilon))\\
&=\exp(q\ln(1+\lambda-\varepsilon)\otimes\exp(q\ln(1+\lambda-\varepsilon)\\\
&=\lambda^q\otimes \lambda^q.
\end{align*}
Hence, if $x\in \calP[A]$, $y\in \calP[B]$, $\lambda^q(xy)=\lambda^q(x)\lambda^q(y)$,
so $\lambda^q\in \chara(\calP)$. \end{proof}

\begin{cor}
Let $(\calP,m)$ be a connected twisted bialgebra. Then $(\chara(\calP),*)$ is a group.
\end{cor}

\begin{proof}
For any $\lambda \in \chara(\calP)$,
$\lambda^{-1}*\lambda=\lambda*\lambda^{-1}=\lambda_0=\varepsilon$,
so $\lambda^{-1}$ is indeed an inverse of $\lambda$ for the convolution.
\end{proof}

Let us now give a species version of Aguiar, Bergeron and Sottile's Theorem on the Hopf algebra 
of quasisymmetric functions \cite{Aguiar2}:

\begin{theo} \label{theo17}
Let $\mathcal{B}$ be a connected twisted bialgebra.
Let $\alpha:\mathcal{B}\longrightarrow \com$ be a species morphism. There exists a unique coalgebra morphism
$\phi:\mathcal{B}\longrightarrow \calcomp$ such that $\varepsilon'\circ \phi=\alpha$.
For any $x\in \mathcal{B}[A]$, with $A\neq \emptyset$:
\begin{align*}
\phi(x)=\sum_{(A_1,\ldots,A_k)\in \comp[A]} 
\alpha^{\otimes k}\circ \Delta_{A_1,\ldots,A_k}(x) (A_1,\ldots,A_k).
\end{align*}
Moreover, $\phi$ is a bialgebra morphism if, and only if, $\phi$ is a character.
\end{theo}

\begin{proof}
Immediate corollary of Theorem \ref{theo10}, as $\calcomp=\cot(\com_+)$.
\end{proof}

\subsection{Subalgebras of shuffle twisted bialgebras}

\begin{theo} \label{theo18}
Let $\calP$ be a twisted, connected and commutative bialgebra. \begin{enumerate}
\item There exists a species $\calQ$ and an injective morphism
of twisted bialgebras from $\calP$ to the shuffle bialgebra $(\cot(\calQ),\shuffle,\Delta)$. 
\item There exists a species $\calQ$ and an isomorphim
of twisted bialgebras from $\calP$ to the shuffle bialgebra $(\cot(\calQ),\shuffle,\Delta)$ if, and only if, 
$\calP$ is cofree.
\end{enumerate}
\end{theo}

\begin{proof}
\textit{First step}. Let $\calP$ and $\calQ$ be species. Then the set of linear morphisms
from $\calP$ to $\calQ$ is a species, defined by:
\begin{enumerate}
\item $L(\calP,\calQ)[A]$ is the space of linear maps from $\calP[A]$ to $\calQ[A]$.
\item If $\sigma:A\longrightarrow B$ is a bijection and $f\in L(\calP,\calQ)[A]$,
$L(\calP,\calQ)[\sigma](f)=\calQ[\sigma]\circ f\circ \calP[\sigma^{-1}]$.
\end{enumerate}
If $\calP$ is a unitary twisted bialgebra, $m:\calP\otimes \calP \longrightarrow \calP$ is a surjective
map of species. Let us denote by $\overline{m}:(\calP\otimes \calP)/\ker(m) \longrightarrow \calP$ the canonically
associated isomorphism of species. If $\calP$ is commutative, then $m$ is a twisted algebra morphism: indeed,
if $x\in \calP[A]$, $y\in \calP[B]$, $x'\in \calP[A']$, $y'\in \calP[B']$:
\[m_{A\sqcup A',B\sqcup B'}((x\otimes y)\otimes (x'\otimes y'))=xx'yy'=xyx'y'
=m_{A,B}(x\otimes y)m_{A',B'}(x'\otimes y').\]
Therefore, $(\calP\otimes \calP)/\ker(m)$ becomes an algebra and $\overline{m}$ is an algebra isomorphism.\\

If now $\calP$ is a commutative twisted bialgebra, we shall consider the morphism of species defined by:
\[M:\left\{\begin{array}{rcl}
L(\calP,\calP)&\longrightarrow&L(\calP\otimes \calP,(\calP\otimes \calP)/\ker(m))\\
f&\longrightarrow&\overline{m}^{-1}\circ f\circ m.
\end{array}\right.\]
In other terms, for any $f:\calP[A]\longrightarrow \calP[A]$, if $A=I\sqcup J$:
\[f\circ m_{I,J}=\overline{m}_{I,J} \circ M[A](f).\]
As $\calP\otimes \calP$ is a twisted coalgebra and $(\calP\otimes \calP)/\ker(m)$ is a twisted algebra,
the species $L(\calP\otimes \calP,(\calP\otimes \calP)/\ker(m))$ is given a convolution product $*$.
Let $f,g\in L(\calP,\calP)[A]$. We put $F=M[A](f)$ and $G=M[A](g)$. 
Then, if $A=I\sqcup J$, $x\in \calP[I]$, $y\in \calP[J]$:
\begin{align*}
(f*g)\circ m_{I,J}(x\otimes y)&=\pi \circ m\circ (f\otimes g)\circ (m\otimes m)\circ (\id \otimes c\otimes \id)\circ 
(\Delta \otimes \Delta)(x\otimes y)\\
&=m\circ (\overline{m}\otimes \overline{m})\circ (F\otimes G)\circ (\id \otimes c\otimes \id)\circ 
(\Delta \otimes \Delta)(x\otimes y).
\end{align*}
Let us denote by $\pi$ the canonical projection from $\calP\otimes \calP$ onto $(\calP\otimes\calP)/\ker(m)$. 
We shall write:
\begin{align*}
F&=\pi\circ(\sum F'_j\otimes F''_j),&G&=\pi\circ(\sum G'_k\otimes G''_k).
\end{align*}
where $F'_j$, $F''_j$, $G'_k$, $G''_k$ are linear endomorphisms of certain $\calP[B]$, with $B\subseteq A$. 
Then, with Sweedler's notation:
\begin{align*}
(f*g)\circ m_{I,J}(x\otimes y)&=f(x^{(1)} y^{(1)})g(x^{(2)}y^{(2)})\\
&=\sum F'_j(x^{(1)})F''_j(y^{(1)})G'_k(x^{(2)})G''_k(y^{(2)})\\
&=\sum F'_j(x^{(1)})G'_k(x^{(2)})F''_j(y^{(1)})G''_k(y^{(2)})\\
&=\sum F'_j*G'_k(x) F''_j*G''_k(y)\\
&=F*G(xy).
\end{align*}
Hence, $M(f*g)=M(f)*M(g)$: $M$ is compatible with $*$.
Moreover, for any $x\in \calP[A]$, $y\in \calP[B]$:
\[\id\circ m(x\otimes y)=xy=\overline{m}\circ\pi(x\otimes y).\]
Hence, $M(\id)=\pi$. \\

\textit{Second step.} Let us denote by $\rho$ the canonical projection from $\calP$ to $\calP_+$.
As $\calP$ is connected, $\id=\nu+\rho$, where $\nu$ is the unit of the convolution product of $L(\calP,\calP)$.
For any finite set $A$, by connectivity of $\calP$:
\[\rho^{*k}[A]=\sum_{(A_1,\ldots,A_k)\in \comp(A)} m_{A_1,\ldots,A_k}\circ \Delta_{A_1,\ldots,A_k}.\]
In particular, if $k>\sharp A$, $\rho^{*k}[A]=0$. We can then consider the morphism defined by:
\[f=\sum_{k=1}^\infty \frac{(-1)^{k+1}}{k}\rho^{*k}=\ln(1+\rho)=\ln(\id).\]
As $M$ is compatible with $*$:
\[M(f)=\ln(1+M(\rho))=\ln(1+(\pi-\nu))=\ln(\pi)=\pi \circ \ln (\id \otimes \id).\]
Moreover, $\id \otimes \id=(\id \otimes \nu)*(\nu \otimes \id)$, by property of the $\ln$ formal series:
\[M(f)=\pi \circ (\ln(\id\otimes \nu)+\ln(\nu\otimes \id)=\pi\circ (\ln(\id)\otimes \nu+\nu\otimes \ln(\id)))
=\pi\circ (f\otimes 1+1\otimes f).\]
In other words, if $x\in \calP[A]$, $y\in \calP[B]$, with $A,B \neq \emptyset$:
\[f(xy)=f(x)\varepsilon(y)+\varepsilon(x)f(y)=0.\]
Let $\calQ$ be the subspecies of $\calP$ defined by:
\[\calQ[A]=\{x\in \calP[A],\:\forall (I,J)\in \comp[A],\:\Delta_{I,J}(x)=0\}.\]
If $x\in \calQ[A]$:
\[f(x)=\sum_{(A_1,\ldots,A_k)\in \comp(A)}\frac{(-1)^{k+1}}{k}
m_{A_1,\ldots,A_k}\circ \Delta_{A_1,\ldots,A_k}(x)=x+0.\]
As a consequence, for any finite set $A$:
\[\calQ[A]\cap m(\calP_+\otimes \calP_+)[A]=(0).\]

\textit{Third step.} For any $n\geq 1$, 
$ m(\calP_+\otimes \calP_+)[\underline{n}]$ is a $\mathfrak{S}_n$-submodule
of $\calP[\underline{n}]$. By semisimplicity, there exits a $\mathfrak{S}_n$-submodule $R_n$ 
of $\calP[\underline{n}]$
such that:
\begin{align*}
\calP[\underline{n}]&=\calQ[\underline{n}]\oplus R_n,&
m(\calP_+\otimes \calP_+)[\underline{n}]&\subseteq R_n.
\end{align*}

For any finite set $A$, of cardinality $n$, let $\sigma$ be a bijection from $\underline{n}$ to $A$.
We put $\calR[A]=\calP[\sigma](R_n)$. This does not depend on the choice of $\sigma$: indeed,
if $\tau$ is another bijection, then $\tau^{-1}\circ \sigma \in \mathfrak{S}_n$,
so $\calP[\tau]^{-1}\circ \calP[\sigma](R_n)=\calP[\tau^{-1}\circ \sigma](V_n)=V_n$,
and finally $\calP[\tau](R_n)=\calP[\sigma](R_n)$. We define in this way a subspecies $\calR$ of $\calP$ such that:
\begin{align*}
\calP&=\calQ\oplus \calR,& m(\calP_+\otimes \calP_+)\subseteq \calR.
\end{align*}
Let $\varpi:\calP\longrightarrow \calQ$ be the canonical projection.
By universal property,  there exists a unique coalgebra morphism $\Phi:\calP\longrightarrow \cot(\calQ)$,
such that $\pi \circ \phi=\varpi$, where $\pi:\cot(\calQ)\longrightarrow \calQ$ is the canonical projection.

Let us assume that $\Phi$ is not injective. Let $x\in \calP[A]$, nonzero, such that $\Phi[A](x)=0$.
We assume that the cardinality of $A$ is minimal. If $(I,J)\in \comp(A)$:
\[(\Phi[I]\otimes \Phi[J])\circ \Delta_{I,J}(x)=\Delta_{I,J}\circ \Phi[A](x)=0.\]
By minimality of $\sharp A$, $\Phi[I]$ and $\Phi[J]$ are injective, so $\Delta_{I,J}(x)=0$.
Hence, $x\in \calQ$, so $\Phi[A](x)=\pi \circ \Phi[A](x)=\varpi(x)=x$: this is a contradiction.
Therefore, $\Phi$ is injective.

Let us consider $\Phi_1=\shuffle \circ (\Phi\otimes \Phi)$
and $\Phi_2=\Phi\circ m$. Both are coalgebra morphims from $\calP\otimes \calP$ to $\cot(\calQ)$.
Moreover, if $x\in \calP[A]$, $y\in \calQ[B]$:
\begin{align*}
\pi\circ \Phi_1[A\sqcup B](x\otimes y)&=\pi[A\sqcup B](\Phi[A](x)\shuffle \Phi[B](y))\\
&=\begin{cases}
\varepsilon(x) \pi\circ \Phi[B](y)\mbox{ if }A=\emptyset,\\
\varepsilon(y)\pi\circ \Phi[B](x) \mbox{ if }B=\emptyset,\\
0\mbox{ otherwise},
\end{cases}\\
&=\begin{cases}
\varepsilon(x) \varpi[B](y)\mbox{ if }A=\emptyset,\\
\varepsilon(y)\varpi[B](x) \mbox{ if }B=\emptyset,\\
0\mbox{ otherwise}.
\end{cases}
\end{align*}
On the other hand, as $\varpi\circ m(\calP_+\otimes \calP_+))=(0)$:
\begin{align*}
\pi\circ \Phi_2[A\sqcup B](x\otimes y)&=\varpi \circ m(x\otimes y)\\
&=\begin{cases}
\varepsilon(x) \varpi[B](y)\mbox{ if }A=\emptyset,\\
\varepsilon(y)\varpi[B](x) \mbox{ if }B=\emptyset,\\
0\mbox{ otherwise}.
\end{cases}
\end{align*}
So $\pi \circ \Phi_1=\pi\circ \Phi_2$: by unicity in the universal property, $\Phi_1=\Phi_2$, so
$\Phi$ is a bialgebra morphism.\\

\textit{Last step}. Let us assume that $\calP$ is cofree, and let us prove that $\Phi$ is surjective.
Let $k\geq 0$, $x_i\in \calQ[A_i]$ for all $i$ and $A=A_1\sqcup \ldots \sqcup A_k$.
Let us prove that $x=x_1\ldots x_k\in \Phi[A](\calP[A])$ by induction on $k$. This is obvious if $k\leq 1$.
If $k\geq 2$, as $\calP$ is cofree, there exists an element $y\in \calP$ such that:
\[\Delta_{A_1,\ldots,A_k}(y)=x_1\otimes \ldots \otimes x_k.\]
Then:
\[\Delta_{A_1,\ldots,A_k}\circ \Phi[A](y)=\Phi[A_i](x_1)\otimes \ldots \otimes \Phi[A_k](x_k)
=x_1\otimes \ldots \otimes x_k.\]
Hence, the induction hypothesis can be applied to $x-\Phi[A](y)$, so $x\in \Phi[A](\calP[A])$. \end{proof}

\section{The Hadamard product of species}

\subsection{Double twisted bialgebras}

\begin{defi}[Hadamard product of species]\label{defi19}
\begin{enumerate}
\item Let $\calP$ and $\calQ$ be two species. 
\begin{itemize}
\item For any finite set $A$, we put $\calP\boxtimes \calQ[A]=\calP[A]\otimes \calQ[A]$.
\item For any bijection $\sigma:A\longrightarrow B$, we put $\calP\boxtimes \calQ[\sigma]
=\calP[\sigma]\otimes \calQ[\sigma]$.
\end{itemize}
Then $\calP\boxtimes \calQ$ is a species.
\item Let $\calP_1$, $\calP_2$, $\calQ_1$, $\calQ_2$ be species, 
and $\phi_1:\calP_1\longrightarrow \calQ_1$, $\phi_2:\calP_1\longrightarrow \calQ_2$ be species morphisms.
We define a morphism $\phi_1\boxtimes \phi_2$ of species from $\calP_1\boxtimes \calQ_1$ 
to $\calP_2\boxtimes \calQ_2$ by:
\[\phi_1\boxtimes \phi_2[A]=\phi_1[A]\otimes \phi_2[A]:\calP_1\boxtimes \calP_2[A]\longrightarrow 
\calQ_1\boxtimes \calQ_2[A].\]
\item Let $\calP$ and $\calQ$ be two species. The following defines a morphism of species from 
$\calP\boxtimes \calQ$ to $\calQ\boxtimes \calP$:
\[\tau_{\calP,\calQ}:\left\{\begin{array}{rcl}
\calP\boxtimes \calQ[A]&\longrightarrow&\calQ\boxtimes \calP[A]\\
x\otimes y&\longrightarrow&y\otimes x.
\end{array}\right.\]
\end{enumerate}
\end{defi}

The species $\com$ is the identity for this tensor product: for any species $\calP$,
\[\calP\boxtimes \com=\com\boxtimes \calP=\calP.\]

\begin{defi}
A twisted bialgebra of the second kind is a family $(\calP,m,\delta)$
where $(\calP,m)$ is a twisted algebra and $\delta:\calP\longrightarrow \calP\boxtimes \calP$ 
is a morphism of species such that: 
\begin{enumerate}
\item The following diagram commutes:
\[\xymatrix{\calP\ar[r]^{\delta}\ar[d]_{\delta}&\calP^{\boxtimes 2}\ar[d]^{\id_\calP\boxtimes \delta}\\
\calP^{\boxtimes 2}\ar[r]_{\delta \boxtimes \id_\calP}&\calP^{\boxtimes 3}}\]
\item There exists a morphism of species $\varepsilon':\calP\longrightarrow \com$
 such that the following diagram commutes:
\[\xymatrix{\com\boxtimes \calP\ar[r]^{\id}&\calP\ar[d]^\delta&\calP\boxtimes \com \ar[l]_{\id}\\
&\calP^{\boxtimes 2}\ar[lu]^{\varepsilon' \boxtimes \id_\calP}\ar[ru]_{\id_\calP\boxtimes \varepsilon'}&}\]
\item $\delta$ is an algebra morphism from $\calP$ to $\calP\boxtimes \calP$, that is to say, for any
finite sets $A,B$:
\[\delta_{A\sqcup B}\circ m_{A,B}=(m_{A,B}\otimes m_{A,B})\circ (\id_{\calP[A]}\otimes c_{A,B}\otimes \id_{\calP[B]})
\circ (\delta_A\otimes \delta_B),\]
and $\delta_\emptyset(1_\calP)=1_\calP\otimes 1_\calP$.
\end{enumerate}\end{defi}

In other words, for any finite set $A$, there exists a coproduct 
$\delta_A:\calP[A]\longrightarrow \calP[A]\otimes \calP[A]$, making $\calP[A]$ a coalgebra of counit $\varepsilon'_A$.
If $\sigma:A\longrightarrow B$ is a bijection, $\calP[\sigma]$ is a coalgebra isomorphism from 
$(\calP[A],\delta_A)$ to $(\calP[B],\delta_B)$. 
In other terms,  these objects are algebras in the category of species in the category of coalgebras, 
that is to say the category of  functors from the category of finite sets with bijections to the category of coalgebras.

\begin{notation}
If $\calP$ is a twisted bialgebra of the second kind, we adopt Sweedler's notation for its coproduct:
if $A$ in a finite set and $x\in \calP[A]$,
\begin{align*}
\delta_A(x)&=x'\otimes x'',\\
(\delta_A\otimes \id_{\calP[A]})\circ \delta_A(x)
&=(\id_{\calP[A]} \otimes \delta_A)\circ \delta_A(x)=x'\otimes x''\otimes x'''.
\end{align*}
\end{notation}

\begin{defi}\label{defi21}
A double twisted bialgebra is a family $(\calP,m,\Delta,\delta)$ such that:
\begin{enumerate}
\item $(\calP,m,\Delta)$ is a twisted bialgebra. Its counit is denoted by $\varepsilon$.
\item $(\calP,m,\delta)$ is a twisted bialgebra of the second kind. Its counit is denoted by $\varepsilon'$.
\item $\Delta$ is a right comodule morphism, that is, for any finite sets $A$, $B$:
\[(\Delta_{A,B}\otimes \id_{\calP[A\sqcup B]})\circ \delta_{A\sqcup B}=m_{1,3,24}\circ 
(\delta_A\otimes \delta_B)\circ \Delta_{A,B},\]
where:
\[m_{1,3,24}:\left\{\begin{array}{rcl}
\calP[A]\otimes \calP[A]\otimes \calP[B]\otimes \calP[B]&\longrightarrow&\calP[A]
\otimes \calP[B]\otimes \calP[A\sqcup B]\\
x\otimes y\otimes z\otimes t&\longrightarrow&x\otimes z\otimes m_{A,B}(y\otimes t).
\end{array}\right.\]

\item The counit $\varepsilon:\calP \longrightarrow \I$ is a right comodule morphism, that is, 
for any $x\in \calP[\emptyset]$:
\[(\varepsilon\otimes \id)\circ \delta_\emptyset(x)=\varepsilon(x)1_\calP\otimes 1_\calP.\]
\end{enumerate}
\end{defi}

\begin{example}
The species $\com$ is a double bialgebra with the coproduct defined on $\com[A]=\K$ by:
\[\delta_A(1)=1\otimes 1.\]
\end{example}

\begin{remark}
Let $(A,m,\Delta,\delta)$ be double twisted bialgebra.
For any finite sets $A$, $B$, $\calP[A]\otimes \calP[B]$ 
is a right $(\calP[A\sqcup B],\delta_{A\sqcup B})$-comodule with the coaction:
\[\rho_{A,B}=\underbrace{(\id^{\otimes 2} \otimes m_{A,B})\circ (\id_{\calP[A]}\otimes \tau_{A,B} \otimes \id_{\calP(B)})}_{=m_{1,3,24}}
\circ (\delta_A\otimes \delta_B).\] 
More generally, if $A_1,\ldots, A_k$ are finite sets, then $\calP[A_1]\otimes \ldots \otimes \calP[A_k]$ is a right 
$\calP[A_1\sqcup \ldots \sqcup A_k]$-comodule, with the coaction defined by:
\[\rho_{A_1,\ldots,A_k}=m_{1,3,\ldots,2n-1,24\ldots 2n}\circ (\delta_{A_1}\otimes \ldots \otimes \delta_{A_n}).\]
For any finite sets $A_1,\ldots,A_k$, $m_{A_1,\ldots,A_k}:\calP[A_1]\otimes \ldots \otimes \calP[A_k]
\longrightarrow \calP[A_1\sqcup \ldots \sqcup A_k]$ and 
$\Delta_{A_1,\ldots,A_k}:\calP[A_1\sqcup \ldots \sqcup A_k]\longrightarrow\calP[A_1]\otimes \ldots\otimes \calP[A_k]$
are right comodule morphisms.
\end{remark}

\begin{prop}
Let $(\calP,m,\Delta,\delta)$ be a connected twisted double bialgebra. Then $m$ is commutative.
\end{prop}

\begin{proof}
Let $A$ and $B$ be finite sets, $x\in \calP[A]$, $y\in \calP[B]$. 
As $\calP[\emptyset]=\K 1_\calP$:
\begin{align*}
\Delta_{\emptyset,A}(x)&=1_\calP\otimes x,&
\Delta_{B,\emptyset}(y)&=y\otimes 1_\calP.
\end{align*}
By the compatibility between the product $m$ and the coproduct $\Delta$:
\begin{align*}
\Delta_{B,A}(xy)=(m_{B,B}\otimes m_{A,A})\circ(\id \otimes c_{\calP[B],\calP[A]}\otimes \id)
(\Delta_{\emptyset,A}(x)\otimes \Delta_{B,\emptyset}(y))=y\otimes x.
\end{align*}
Then:
\begin{align*}
m_{1,3,24}\circ (\delta_B\otimes \delta_A)\circ \Delta_{B,A}(xy)
&=m_{1,3,24}\circ (\delta_B\otimes \delta_A)(y\otimes x)\\
&=y'\otimes x'\otimes y''x''.
\end{align*}
Applying $\varepsilon'[B]\otimes \varepsilon'[A]\otimes \id$:
\[(\varepsilon'[B]\otimes \varepsilon'[A]\otimes \id)\circ m_{1,3,24}
\circ (\delta_B\otimes \delta_A)\circ \Delta_{B,A}(xy)=yx.\]
Moreover:
\begin{align*}
m_{1,3,24}\circ (\delta_B\otimes \delta_A)\circ \Delta_{B,A}(xy)&=
(\Delta_{B,A}\otimes \id)\circ \delta_{A\sqcup B}(xy)\\
&=\Delta_{B,A}((xy)')\otimes (xy)''\\
&=\Delta_{B,A}(x'y')\otimes x''y''\\
&=y'\otimes x'\otimes x''y''.
\end{align*}
Applying $\varepsilon'[B]\otimes \varepsilon'[A]\otimes \id$:
\[(\varepsilon'[B]\otimes \varepsilon'[A]\otimes \id)\circ m_{1,3,24}
\circ (\delta_B\otimes \delta_A)\circ \Delta_{B,A}(xy)=xy.\]
Hence, $xy=yx$: $m$ is commutative. \end{proof}

\begin{prop}\label{prop23}
Let $(\calP,\cdot,\delta_\calP)$ be a connected commutative twisted bialgebra of the second kind.
For any $k\geq 0$, we consider the set $\cont_k$ of pairs $(\sigma,\tau)$, where:
\begin{enumerate}
\item $\sigma:\underline{k}\longrightarrow \underline{\max(\sigma)}$ is a non decreasing surjection.
\item $\tau:\underline{k}\longrightarrow \underline{\max(\tau)}$ is a surjection.
\item For any $i,j\in \underline{k}$, if $i<j$ and $\sigma(i)=\sigma(j)$, then $\tau(i)<\tau(j)$.
\end{enumerate}
There exists a unique coproduct $\delta:\cot(\calP_+)\longrightarrow \cot(\calP_+)\otimes \cot(\calP_+)$ such that:
\begin{enumerate}
\item $(\cot(\calP_+),\squplus,\Delta,\delta)$ is a double twisted bialgebra.
\item $\delta_{\mid \calP}=\delta_\calP$. 
\item For any $(A_1,\ldots, A_k)\in \comp[A]$, $x_i\in \calP[A_i]$:
\[(\varepsilon'\otimes \id)\circ \delta(x_1\ldots x_k)=\varepsilon'(x_1'\cdot\dots\cdot x_k') x_1''\cdot\dots\cdot x_k''.\]
\end{enumerate}
Moreover, for any $(A_1,\ldots, A_k)\in \comp[A]$, $x_i\in \calP[A_i]$:
\begin{align}
\label{eq2}
\delta_A(x_1\ldots x_k)=\sum_{(\sigma,\tau)\in \cont_k} 
\sigma \rightarrow x'_1\ldots x'_k\otimes \tau\rightarrow x''_1\ldots x''_k.
\end{align}
The counit is given by:
\[\epsilon'_{\mid \calP^{\otimes k}}=\begin{cases}
\varepsilon'_\calP\mbox{ if }k=1,\\
0\mbox{ if }k\geq 2.
\end{cases}\]  
Moreover, $\varepsilon':(\cot(\calP_+),m,\delta)\longrightarrow (\calP,m_\calP,\delta_\calP)$ 
is a morphism of bialgebras of the second kind.
\end{prop}

\begin{proof} \textit{Existence}. Let us consider the coproduct defined by (\ref{eq2}).
If $x\in \calP[A]$, as $\cont_1=\{(\id,\id)\}$:
\[\delta_A(x)=x'\otimes x''=\delta_\calP(x).\]
We denote by $1_k:\underline{k}\longrightarrow \underline{1}$ the constant map.
For any $(A_1,\ldots, A_k)\in \comp[A]$, $x_i\in \calP[A_i]$:
\begin{align*}
(\epsilon'\otimes \id)\circ \delta_A(x_1\ldots x_k)&=
\sum_{(1_k,\tau)\in \cont_k}\varepsilon_\calP'(1_k\rightarrow x'_1\ldots x''_k)\tau\rightarrow x''_1\ldots x''_k\\
&=\varepsilon_\calP'(1_k\rightarrow x'_1\ldots x''_k)x''_1\ldots x''_k\\
&=\varepsilon_\calP'(x'_1\cdot \ldots \cdot x'_k)x''_1\ldots x''_k\\
&=\varepsilon_\calP'(x'_1)\ldots \varepsilon'(x'_k)x''_1\ldots x''_k\\
&=x_1\ldots x_k.
\end{align*}
Similarly, $(\id \otimes \epsilon')\circ \delta_A=\id_{\cot(\calP)[A]}$, so $\epsilon'$ is a counit for $\delta$.\\

For any $k,l\geq 0$, we denote by $\cont'_{k,l}$ the set of pairs $(\alpha,\beta)$ such that:
\begin{enumerate}
\item $\alpha:\underline{k+l}\longrightarrow \underline{m}$ is a surjection, non decreasing on $\underline{k}$ 
and on $\underline{k+l}\setminus \underline{k}$.
\item $\beta:\underline{k+l}\longrightarrow \underline{n}$ is a surjection.
\item For any $i,j\in \underline{k+l}$, if $i<j$ and $\alpha(i)=\alpha(j)$, then $\beta(i)<\beta(j)$.
\end{enumerate}
Let $a=a_1\ldots a_k\in \cot(\calP)[A]$ and $b=b_1\ldots b_l\in \cot(\calP)[B]$.
We put:
\begin{align*}
a'\otimes a''&=a'_1\ldots a'_k\otimes a''_1\ldots a''_k,&
b'\otimes b''&=b'_1\ldots b'_k\otimes b''_1\ldots b''_k.
\end{align*} 
\begin{align*}
\delta_{A\sqcup B}\circ m_{A,B}(a\otimes b)&=\sum_{\substack{\alpha \in \qsh(k,l),\\ (\sigma,\tau)\in
\cont_{\max(\alpha)}}} (\sigma \circ \alpha) \rightarrow  (a'b')
\otimes (\tau\circ \alpha)\rightarrow (a''b'')\\
&=\sum_{(\alpha,\beta)\in \cont'_{k,l}} \alpha\rightarrow (a'b')\otimes \beta \rightarrow (a''b'');\\
m_{13,24}\circ (\delta_A\otimes \delta_B)(a\otimes b)&=\sum_{\substack{(\sigma',\tau')\in \cont_k,\\
(\sigma'',\tau'')\in \cont_l,\\\alpha\in \qsh(\max(\sigma'),\max(\sigma'')),\\
\beta \in \qsh(\max(\tau'),\max(\tau''))}}(\alpha\circ (\sigma'\otimes \sigma''))\rightarrow (a'b')
\otimes (\beta \circ (\tau'\otimes \tau''))\rightarrow (a''b'')\\
&=\sum_{(\alpha,\beta)\in \cont'_{k,l}} \alpha\rightarrow (a'b')\otimes \beta \rightarrow (a''b'').
\end{align*}
So $\delta$ is multiplicative.\\

Let us prove that $\Delta$ is a right $(\cot(\calP_+),m,\delta)$-comodule. 
For any $k\geq p \geq 0$, let us denote by $\cont''_{p,k}$ the set of triples $(\alpha,\beta,\gamma)$ such that:
\begin{enumerate}
\item $\alpha:\underline{p}\longrightarrow \underline{l}$ and $\beta:\underline{k}\setminus \underline{p}
\longrightarrow \underline{m}$ are non decreasing bijections.
\item $\gamma:\underline{k}\longrightarrow \underline{n}$ is a surjection.
\item For any $i,j\in \underline{p}$, if $i<j$ and $\alpha(i)=\alpha(j)$, then $\gamma(i)<\gamma(j)$.
\item For any $i,j\in \underline{k}\setminus \underline{p}$, if $i<j$ and $\beta(i)=\beta(j)$, then $\gamma(i)<\gamma(j)$.
\end{enumerate}
Let $a=x_1\ldots x_k \in \cot(\calP_+)[A\sqcup B]$. If there exists a (unique) $p\leq k$ such that 
$x_1\ldots x_i \in \cot(\calP_+)[A]$, we put
$a'\otimes a''\otimes a'''=x'_1\ldots x'_p\otimes x'_{p+1}\ldots x''_k\otimes x''_1\ldots x''_k$.
Then:
\begin{align*}
(\Delta_{A,B}\otimes \id)\circ \delta_{A\sqcup B}(a)
&=\sum_{\substack{(\sigma,\tau)\in \cont_k,\\ \sigma(p)<\sigma(p+1)}} \sigma_{\mid \underline{p}}\rightarrow a'
\otimes \sigma_{\mid \underline{k}\setminus \underline{p}}\rightarrow a''
\otimes \tau \rightarrow a'''\\
&=\sum_{(\alpha,\beta,\gamma)\in \cont''_{p,k}} \alpha \rightarrow a'\otimes \beta 
\rightarrow a''\otimes \gamma \rightarrow a''';\\
m_{1,3,24}\circ (\delta_A\otimes \delta_B)\circ \Delta_{A,B}(a)
&=\sum_{\substack{(\sigma',\tau')\in \cont_p,\\(\sigma'',\tau'')\in \cont_{k-p},\\\alpha \in \qsh(\max(\tau'),\max(\tau''))}}
\sigma'\rightarrow a'\otimes \sigma''\rightarrow a''\otimes (\alpha \circ(\tau'\cdot\tau''))\rightarrow a'''\\
&=\sum_{(\alpha,\beta,\gamma)\in \cont''_{p,k}} \alpha \rightarrow a'\otimes \beta \rightarrow a''\otimes \gamma \rightarrow a'''.
\end{align*}
Otherwise, both are equal to $0$. So $\Delta$ is a right comodule morphism. \\

Let us prove that $\delta$ is coassociative. We work on $\cot(\calP_+)[A]$ and proceed
by induction on $\sharp A$. If $A=\emptyset$, then:
\[(\delta_\emptyset\otimes \id)\circ \delta_\emptyset(1)
(\id\otimes \delta_\emptyset)\circ \delta_\emptyset(1)=1\otimes 1\otimes 1.\]
Let us assume the result at all ranks $<\sharp A$. Let $(I,J)\in \comp[A]$. For any $x\in \cot(\calP_+)[A]$,
putting $\Delta_{I,J}(x)=x_I\otimes x_J$:
\begin{align*}
&(\Delta_{I,J}\otimes \id\otimes \id)\circ (\delta_A\otimes \id)\circ \delta_A(x)\\
&=(m_{1,3,24}\otimes \id)\circ (\delta_I\otimes \delta_J\otimes \id)\circ (\Delta_{I,J}\otimes \id)\circ \delta_A(x)\\
&=(m_{1,3,24}\otimes \id)\circ (\delta_I\otimes \delta_J\otimes \id)\circ m_{1,3,24}\circ (\delta_I\otimes \delta_J)
\circ \Delta_{I,J}(x)\\
&=(x_I)'\otimes (x_J)'\otimes (x_I)''(x_J)''\otimes (x_I)'''(x_J)''',\\ \\
&\Delta_{I,J}\otimes \id\otimes \id)\circ (\id\otimes \delta_A)\circ \delta_A(x)\\
&=(\id \otimes \id \otimes \delta_A)\circ m_{1,3,24}\circ (\delta_I\otimes \delta_J)\circ \Delta_{I,J}(x)\\
&=(x_I)'\otimes (x_J)'\otimes \delta_A((x_I)''(x_J)'')\\
&=(x_I)'\otimes (x_J)'\otimes (x_I)''(x_J)''\otimes (x_I)'''(x_J)'''.
\end{align*}
Hence, $(\delta_A\otimes \id)\circ \delta_A(x)-(\id\otimes \delta_A)\circ \delta_A(x)$ belongs to:
\[\bigcap_{(I,J)\in \comp[A]} \ker(\Delta_{I,J})\otimes \cot(\calP_+)[A]=\calP[A]\otimes \cot(\calP_+)[A].\]
If $x\in \calP^{\otimes k}(A)$:
\begin{align*}
(\pi\otimes \id\otimes \id)\circ (\delta_A\otimes \id)\circ \delta_A(x)&=
\sum_{(\sigma,\tau)\in\cont_k} x'\otimes \sigma\rightarrow x''\otimes \tau\rightarrow x'''\\
&=(\id\otimes \delta_A)\circ (\pi\otimes \id)\circ \delta_A(x)\\
&=(\pi\otimes \id\otimes \id)\circ (\id\otimes \delta_A)\circ \delta_A(x).
\end{align*}
Consequently, $\delta_A$ is coassociative.\\

\textit{Unicity}. Let $\delta'$ be another coproduct with the same properties.
Let us show that $\delta_A=\delta'_A$ by induction on $\sharp A$. This is obvious if $A=\emptyset$.
Otherwise, if $(I,J)\in \comp[A]$:
\begin{align*}
(\Delta_{I,J}\otimes \id)\circ \delta_A&=m_{1,3,24}\circ (\delta_I\otimes \delta_J)\circ \Delta_{I,J}\\
&=m_{1,3,24}\circ (\delta'_I\otimes \delta'_J)\circ \Delta_{I,J}\\
&=(\Delta_{I,J}\otimes \id)\circ \delta'_A.
\end{align*}
Hence, $\delta_A-\delta'_A$ takes its values in $\calP[A]\otimes \cot(\calP_+)[A]$. As $(\varepsilon'\otimes \id)\circ\delta_A
=(\varepsilon'\otimes \id)\circ \delta'_A$, $\delta_A=\delta'_A$.\\

We already proved that $\varepsilon'$ is an algebra morphism. Let $x\in \calP_+[A]^{\otimes k}$, with $k\geq 1$. 
Then:
\begin{align*}
\delta_A\circ \varepsilon'(x)&=\begin{cases}
0\mbox{ if }k \geq 2,\\
\delta_\calP(x)\mbox{ if }k=1;
\end{cases}\\
(\varepsilon'\otimes \varepsilon')\circ \delta_A(x)&=\begin{cases}
\varepsilon(1_k\longrightarrow x')\otimes \varepsilon(1_k\longrightarrow x'')\mbox{ if }(1_k,1_k)\in \cont'_1,\\
0\mbox{ otherwise}
\end{cases}\\
&=\begin{cases}
0\mbox{ if }k \geq 2,\\
\delta_\calP(x)\mbox{ if }k=1.
\end{cases}
\end{align*}
Therefore, $\delta_A\circ \varepsilon'=(\varepsilon'\otimes \varepsilon')\circ \delta_A$.
\end{proof}

\begin{example}
We denote by $\cdot$ the product of $\calP$ and we use Sweedler's notation for the coproduct of $\calP$.
If $x\in \calP[A]$, $y\in \calP[B]$, $z\in \calP[C]$:
\begin{align*}
\delta(x)&=x'\otimes x'',\\
\delta(xy)&=x'y'\otimes x''\squplus y''+x'\cdot y'\otimes x'' y''\\
&=x'y'\otimes (x''y''+y''x''+x''\cdot y'')+x'\cdot y'\otimes x'' y'',\\
\delta(xyz)&=x'y'z'\otimes x''\squplus y''\squplus z''+x'(y'\cdot z')\otimes x''\squplus y''z''\\
&+(x'\cdot y')z'\otimes x''y''\squplus z''+ x'\cdot y'\cdot z'\otimes x''y''z''.
\end{align*}
\end{example}

\begin{cor} \label{cor24}
The bialgebra $(\calcomp,\squplus,\Delta)$ is made a double bialgebra with the coproduct given by the following:
if $(A_1,\ldots,A_k)\in \comp[A]$,
\begin{align*}
&\delta_A(A_1,\ldots,A_k)\\
&=\sum_{(\sigma,\tau)\in \cont_k} \sigma \rightarrow (A_1,\ldots,A_k)\otimes \tau\rightarrow (A_1,\ldots,A_k)\\
&=\sum_{1\leq i_1<\ldots<i_p<k}(A_1\sqcup\ldots \sqcup A_{i_1},\ldots,A_{i_p+1}\sqcup\ldots \sqcup A_k)
\otimes (A_1,\ldots,  A_{i_1})\squplus \ldots\squplus (A_{i_p+1},\ldots,  A_k).
\end{align*}
The counit is given by $\varepsilon'(A_1,\ldots, A_k)=\delta_{k,1}$.
\end{cor}

\begin{proof}
This comes from $\calcomp=\cot(\com_+)$. Note that in this case, $\varepsilon'=\epsilon'$.
\end{proof}

\begin{example} In $\calcomp$, if $A$, $B$ and $C$ are finite sets:
\begin{align*}
\delta(A)&=(A)\otimes (A),\\
\delta(A,B)&=(A,B)\otimes ((A)\squplus (B))+(A\sqcup B)\otimes (A,B),\\
\delta(A,B,C)&=(A,B,C)\otimes ((A)\squplus (B)\squplus (C))+(A,B\sqcup C)\otimes ((A)\squplus (B,C))\\
&+(A\sqcup B,C)\otimes ((A,B)\squplus (C))+(A\sqcup B\sqcup C)\otimes (A,B,C).
\end{align*}
\end{example}

\subsection{Characters of a double bialgebra}

\begin{prop}
Let $\calP=(\calP,m,\Delta,\delta)$ be a double twisted bialgebra. The set of characters $\chara(\calP)$ 
has a second convolution product $\star$, making it a monoid:
\begin{align*}
&\forall f,g\in \chara(\calP),&f\star g=m_{\com}\circ (f\otimes g)\circ \delta.
\end{align*}
Its unit is $\varepsilon'$. For any $f,g,h \in \chara(\calP)$:
\[(f*g)\star h=(f\star h)*(g\star h).\]
We denote by $M_\calP$ the monoid $(\chara(\calP),\star)$.
\end{prop}

\begin{proof}
As $\delta$ is an algebra morphism, if $f,g\in \chara(\calP)$, then $f\star g \in \chara(\calP)$. 
\begin{align*}
(f*g)\star h&=(f\otimes g\otimes h)\circ (\Delta \otimes \id_\calP)\circ \delta\\
&=(f\otimes g\otimes h)\circ m_{1,3,24}\circ (\delta \otimes \delta)\circ \Delta\\
&=(f\otimes h\otimes g\otimes h) \circ (\delta \otimes \delta)\circ \Delta\\
&=(f\star h)*(g\star h).
\end{align*}
The associativity of $\star$ comes from the coassociativity of $\delta$.
\end{proof}

\begin{prop} \label{prop26}
Let $\calP=(\calP,m,\Delta,\delta)$ be a double twisted bialgebra. 
Then $M_\calP$ acts on the space $\mor(\calP,\calQ)$ od species morphisms from $\calP$ to $\calQ$ by:
\[\left\{\begin{array}{rcl}
\mor(\calP,\calQ)\times M_\calP&\longrightarrow&\mor(\calP,\calQ)\\
(\phi,f)&\longrightarrow&\phi\leftarrow f=(\phi\otimes f)\circ \delta.
\end{array}\right.\]
Moreover:
\begin{enumerate}
\item Let $\calQ$ be a twisted algebra. We denote by $\mor_A(\calP,\calQ)$ the set of algebra morphisms from $\calP$ to $\calQ$.
Then $\mor_A(\calP,\calQ)$ is a $M_\calP$-submodule of $\mor(\calP,\calQ)$.
\item Let $\calQ$ be a twisted coalgebra. We denote by $\mor_C(\calP,\calQ)$ the set of coalgebra morphisms from $\calP$ to $\calQ$.
Then $\mor_C(\calP,\calQ)$ is a $M_\calP$-submodule of $\mor(\calP,\calQ)$.
\item Let $\calQ$ be a twisted bialgebra. We denote by $\mor_B(\calP,\calQ)$ 
the set of bialgebra morphisms from $\calP$ to $\calQ$.
Then $\mor_B(\calP,\calQ)$ is a $M_\calP$-submodule of $\mor(\calP,\calQ)$.
\end{enumerate}
\end{prop}

\begin{proof}
Let $\phi\in \mor(\calP,\calQ)$, $f,g\in M_\calP$.
\[(\phi\leftarrow f)\leftarrow g=(\phi\otimes f\otimes g)\circ (\delta \otimes \id_\calP)\circ \delta=
(\phi\otimes f\otimes g)\circ (\id_\calP\otimes \delta)\circ \delta=\phi\leftarrow (f\star g).\]
So $\leftarrow$ is indeed an action.\\

1. Let $\phi\in \mor_A(\calP,\calQ)$ and $f\in M_\calP$. As $\phi$, $f$ and $\delta$ are algebra morphisms, $\phi\leftarrow f$ 
is an algebra morphism, so belong to $\mor_A(\calP,\calQ)$.

2. Let $\phi \in \mor_C(\calP,\calQ)$ and $f\in M_\calP$.
\begin{align*}
\Delta \circ (\phi\leftarrow f)&=\Delta\circ (\phi\otimes f)\circ \delta\\
&=(\phi\otimes \phi\otimes f)\circ (\Delta \otimes \id_\calP)\circ \delta\\
&=(\phi\otimes \phi \otimes f)\circ m_{1,3,24}\circ (\delta \otimes \delta)\circ \Delta\\
&=(\phi \otimes f\otimes \phi\otimes f)\circ (\delta \otimes \delta)\circ \Delta\\
&=(\phi\leftarrow f\otimes \phi\leftarrow f)\circ \Delta.
\end{align*}
So $\phi\leftarrow f\in \mor_C(\calP,\calQ)$. The fact that $\leftarrow$ is an action is proved as in the first point.\\

3. This comes from $\mor_B(\calP,\calQ)=\mor_A(\calP,\calQ)\cap \mor_C(\calP,\calQ)$. \end{proof}

\begin{lemma}\label{lemme27}
Let $\calP$ be a double twisted bialgebra, $\calQ$ and $\calR$ 
be twisted bialgebras. For any species morphisms $\phi:\calP\longrightarrow \calQ$ 
and $\psi:\calQ\longrightarrow \calR$, for any $f\in M_\calP$:
\[(\psi \circ \phi)\leftarrow f=\psi \circ (\phi\leftarrow f).\]
\end{lemma}

\begin{proof}
For any finite set $A$, for any $x\in \calP[A]$:
\begin{align*}
(\psi \circ \phi)\leftarrow f(x)&=(\psi \circ \phi\otimes f)\circ \delta(x)\\
&=\psi \circ \phi(x')f(x'')\\
&=\psi(\phi(x')f(x''))\\
&=\psi(\phi\leftarrow f(x))\\
&=\psi\circ (\phi\leftarrow f)(x),
\end{align*}
where we put $\delta_A(x)=x'\otimes x''$.
\end{proof}

\begin{prop}
Let $\calP$ be a double twisted bialgebra.The following map is an injective morphism of monoids:
\[\chi_\calP:\left\{\begin{array}{rcl}
(M_\calP,\star)&\longrightarrow&(\mor_B(\calP,\calP),\circ)\\
f&\longrightarrow&\id_\calP\leftarrow f.
\end{array}\right.\]
\end{prop}

\begin{proof} Let $f,g\in M_\calP$. By Lemma \ref{lemme27}:
\[(\id_\calP \leftarrow f)\circ (\id_\calP \leftarrow g)=((\id_\calP\leftarrow f)\circ \id_\calP)\leftarrow g
=(\id_\calP\leftarrow f)\leftarrow g=\id\leftarrow (f\star g).\]
So $\chi_\calP(f)\circ \chi_\calP(g)=\chi_\calP(f\star g)$. Let us consider the map:
\[\chi'_\calP:\left\{\begin{array}{rcl}
\mor_B(\calP,\calP)&\longrightarrow&M_\calP\\
\phi&\longrightarrow&\varepsilon'\circ \phi.
\end{array}\right.\]
For any $f\in M_\calP$:
\[\chi'\circ \chi_\calP(f)=\varepsilon'\circ (\id_\calP \otimes f)\circ \delta
=(\varepsilon'\otimes f)\circ \delta=\varepsilon'\star f=f,\]
so $\chi'_\calP\circ \chi_\calP=\id_{M_\calP}$. Consequently, $\chi_\calP$ is injective (and $\chi'_\calP$ is surjective). \end{proof}

\subsection{The terminal property of $\calcomp$}

\begin{theo}\label{theo29}
Let $\calP=(\calP,m,\Delta,\delta)$ be a connected double twisted bialgebra. 
There exists a unique morphism $\phi$ of double bialgebras from $\calP$ to $\calcomp$. Moreover, the following maps
are bijections, inverse one from the other:
\begin{align*}
\varsigma&:\left\{\begin{array}{rcl}
\chara(\calP)&\longrightarrow&\mor_B(\calP,\calcomp)\\
f&\longrightarrow&\phi\leftarrow f,
\end{array}\right.\\
\varsigma'&:\left\{\begin{array}{rcl}
\mor_B(\calP,\calcomp)&\longrightarrow&\chara(\calP)\\
\phi&\longrightarrow&\varepsilon'\circ \phi.
\end{array}\right.
\end{align*}
\end{theo}

\begin{proof}
\textit{Unicity of $\phi$}. If $\phi$ is such a morphism, then $\varepsilon'\circ \phi=\varepsilon'_\calP$. 
By Theorem \ref{theo17}, $\phi$ is unique.\\

\textit{Existence of $\phi$}. Let $\phi$ be the unique morphism from $(\calP,m,\Delta)$ to $(\calcomp,m,\Delta)$, such that 
$\varepsilon'\circ \phi=\varepsilon'_\calP$.
Let us prove that it is compatible with $\delta$. Let $A$ be a finite set and $x\in \calP[A]$. We proceed by induction on $\sharp A$.
If $A=\emptyset$, we can assume that $x=1_\calP$. Then:
\[\delta_\emptyset\circ \phi(1_\calP)=\delta_\emptyset(1_{\calcomp})=1_{\calcomp}\otimes 1_{\calcomp}
=(\phi\otimes \phi)\circ \delta(1_\calP).\]
Let us now assume that $A\neq \emptyset$. Let $\emptyset \subsetneq I\subsetneq A$ and $J=A\setminus I$. Then:
\begin{align*}
(\Delta_{I,J}\otimes \id_{\calcomp(A)})\circ \delta_A \circ \phi(x)
&=m_{1,3,24}\circ (\delta_I\otimes \delta_J)\circ \Delta_{I,J}\circ \phi(x)\\
&=m_{1,3,24}\circ (\delta_I \otimes \delta_J)\circ (\phi\otimes \phi)\circ \Delta_{I,J}(x)\\
&=m_{1,3,24}\circ (\phi\otimes \phi\otimes \phi\otimes \phi)\circ  (\delta_I \otimes \delta_J)
\circ (\phi\otimes \phi)\circ \Delta_{I,J}(x)\\
&=(\phi\otimes \phi\otimes \phi)\circ m_{1,3,24}\circ (\delta_I\otimes \delta_J)\circ \Delta_{I,J}(x)\\
&=(\phi\otimes \phi\otimes \phi)\circ (\Delta_{I,J}\otimes \id_{\calP[A]})\circ \delta_A(x)\\
&=(\Delta_{I,J}\otimes \id_{\calcomp(A)})\circ (\phi\otimes \phi)\circ \delta_A(x).
\end{align*}
We put $y=(\phi\otimes \phi)\circ \delta_A(x)-\delta_A\circ \phi(x)$. Then:
\[y\in \bigcap_{(I,J)\in \comp(A)} \ker(\Delta_{I,J})\otimes \calcomp[A]=(A)\otimes \calcomp[A].\]
We put $y=(A)\otimes z$. Then:
\begin{align*}
z&=(\varepsilon'\otimes \id_{\calcomp(A)})(y)\\
&=(\varepsilon'\circ \phi \otimes \phi)\circ \delta_A(x)-(\varepsilon'\otimes \id_{\calcomp(A)})\circ \delta_A \circ \phi(x)\\
&=(\id \circ \phi)\circ (\varepsilon'\otimes \id)\circ \delta_A(x)-\phi(x)\\
&=\phi(x)-\phi(x)\\
&=0.
\end{align*}
Therefore, $y=0$, so $(\phi\otimes \phi)\circ \delta_A(x)=\delta_A\circ \phi(x)$. \\

\textit{Bijectivity of $\varsigma$ and $\varsigma'$}. Let $f\in \chara(\calP)$. 
\[\varepsilon'\circ (\phi\leftarrow f)=\varepsilon'\circ (\phi \otimes f)\circ \delta
=(\varepsilon'\circ \phi)\otimes f\circ \delta=(\varepsilon'_\calP \otimes f)\circ \delta=f.\]
Hence, $\Upsilon'\circ \Upsilon=\id_{\chara(\calP)}$. By Theorem \ref{theo10}, $\Upsilon'$ is bijective,
so its inverse is $\Upsilon$, which is bijective too.
\end{proof}

\section{The example of graphs}

\subsection{Double twisted bialgebra of graphs}

Let us now give the species $\calgr$ a structure of double bialgebra.

\begin{notation}
Let $G$ be a graph in $\calgr[A]$.
\begin{enumerate}
\item The relation $\sim_G$ is the equivalence whose classes are the vertices of $G$.
\item Let $I\subseteq A$. We assume that $I$ is a union of vertices of $A$.
The graph $G_{\mid I} \in \calgr[I]$ is defined by:
\begin{itemize}
\item $V(G_{\mid I})=\{x\in V(G),\: x\subseteq I\}$.
\item $E(G_{\mid I})=\{\{x,y\}\in E(G), x\cup y\subseteq I\}$.
\end{itemize}
\item Let $\sim$ be an equivalence on $A$. We shall write $\sim\triangleleft G$ if the following hold:
\begin{itemize}
\item For any $x,y\in A$, if $x\sim_G y$, then $x\sim y$.
\item For any class $I$ of $\sim$, $G_{\mid I}$ is a connected graph.
\end{itemize}
If so, we define two graphs $G\mid\sim$ and $G/\sim$:
\begin{enumerate}
\item If $C_1,\ldots,C_k$ are the classes of $\sim$, then $G|\sim=G_{\mid C_1}\ldots G_{\mid C_k}$.
\item $V(G/\sim)$ is the set of classes of $\sim$.
\item Two classes $C$ and $C'$ of $\sim$ are related by an edge in $G/\sim$ if there exist vertices $x$, $x'$ of $G$, 
related by an edge in $G$, such that $x\subseteq C$ and $x'\subseteq C'$. 
\end{enumerate}\end{enumerate}\end{notation}

\begin{prop}
The species $\calgr$ is a double twisted bialgebra, with the following coproducts:
\begin{enumerate}
\item If $G$ is a graph in $\calgr[I\sqcup J]$:
\[\Delta_{I,J}(G)=\begin{cases}
G_{\mid I}\otimes G_{\mid J}\mbox{ if $I$ is a union of vertices of $G$},\\
0\mbox{ otherwise}.
\end{cases}\]
\item If $G$ is a graph in $\calgr[A]$:
\[\delta_A(G)=\sum_{\sim\triangleleft G}G/\sim\otimes G|\sim.\]
The counit of $\delta$ is given by:
\[\varepsilon'_A(G)=\begin{cases}
1\mbox{ if $G$ has no edge},\\
0\mbox{ otherwise}.
\end{cases}\]
\end{enumerate}
\end{prop}

\begin{proof}
Let $G\in \calgr[I\sqcup J\sqcup K]$. Then:
\begin{align*}
&(\Delta_{I,J}\otimes \id_{\calgr(K)})\circ \Delta_{I\sqcup J,K}(G)\\
&=\begin{cases}
G_{\mid I}\otimes G_{\mid J}\otimes G_{\mid K}\mbox{ if $I$, $J$, $K$ are union of vertices of $G$},\\
0\mbox{ otherwise}.
\end{cases}\\
&=(\id_{\calgr[I]}\otimes \Delta_{J,K})\circ \Delta_{I,J\sqcup K}(G)
\end{align*}
Hence, $\Delta$ is coassocative.\\

Let $G$, $H$ be graphs in respectively $\calgr[A]$ and $\calgr[B]$. If $A\sqcup B=I\sqcup J$, then:
\begin{align*}
\Delta_{I,J}(GH)&=\begin{cases}
(GH)_{\mid I}\otimes (GH)_{\mid J} \mbox{ if $I$ is a union of vertices of $G$ and $H$},\\
0\mbox{ otherwise}
\end{cases}\\
&=\begin{cases}
G_{\mid A\cap I}H_{\mid B\cap I}\otimes G_{\mid A\cap J}H_{\mid B\cap J} 
\mbox{ if $I$ is a union of vertices of $G$ and $H$},\\
0\mbox{ otherwise}
\end{cases}\\
&=\Delta_{A\cap I,A\cap J}(G)\Delta_{A\cap J,B\cap J}(H).
\end{align*}
Hence, $(\calgr,m,\Delta)$ is a twisted bialgebra.\\

Let $A$ be a graph in $\calgr[A]$. We consider the set $X$ of pairs $(\sim,\sim')$ of equivalences on $A$ such that:
\begin{enumerate}
\item For any $x,y\in  A$, if $x\sim_G y$, then $x\sim y$; if $x\sim y$, then $x\sim' y$.
\item The equivalence classes of $\sim$ and $\sim'$ are connected graphs. 
\end{enumerate}
Then:
\begin{align*}
(\delta_A\otimes \id_{\calgr[A]})\circ \delta_A(G)&=\sum_{\substack{\sim \triangleleft G,\\ \sim'\triangleleft G/\sim}} 
(G/\sim)/\sim'\otimes (G/\sim)|\sim'\otimes G|\sim\\
&=\sum_{(\sim,\sim')\in X}  (G/\sim)/\sim'\otimes (G/\sim)|\sim'\otimes G|\sim\\
&=\sum_{(\sim,\sim')\in X} G/\sim'\otimes (G|\sim')/\sim\otimes (G|\sim')|\sim\\
&=\sum_{\substack{\sim' \triangleleft G,\\ \sim\triangleleft G|\sim'}}G/\sim'\otimes (G|\sim')/\sim\otimes (G|\sim')|\sim\\
&=(\id_{\calgr[A]}\otimes \delta_A)\circ \delta_A(G).
\end{align*}
So $\delta$ is coassociative. Let $\sim'_G$ be the equivalence whose classes are the connected classes of $G$. 
Then $\sim_G$, $\sim'_G\triangleleft G$ and:
\begin{align*}
(\varepsilon'\otimes \id_{\calgr[A]})\circ \delta_A(G)&=\varepsilon'(G/\sim_G')G|\sim_G'+0=G,\\
(\id_{\calgr[A]}\otimes \varepsilon')\circ \delta_A(G)&=\varepsilon'(G|\sim_G)G/\sim_G+0=G.
\end{align*}
So $\varepsilon'$ is indeed the counit of $\delta$.\\

Let $G$, $H$ be graphs in respectively $\calgr[A]$ and $\calgr[B]$.
 As the connected components of $GH$ are the connected components of $G$ and the connected components of 
$H$, for any equivalence $\sim$ on $A\sqcup B$,  $\sim \triangleleft GH$, if, and only if, $\sim=\sim'\sqcup \sim''$, 
with $\sim'\triangleleft G$ and $\sim''\triangleleft H$. Consequently:
\begin{align*}
\delta_{A\sqcup B}(GH)&=\sum_{\substack{\sim'\triangleleft G,\\\sim'\triangleleft H}}(GH)/\sim'\sqcup 
\sim''\otimes (GH)|\sim'\sqcup \sim''\\
&=\sum_{\substack{\sim'\triangleleft G,\\\sim'\triangleleft H}} (G/\sim')(H/\sim'')\otimes (G|\sim)'(H|\sim'')\\
&=\delta_{A\sqcup B}(G)\delta_{A\sqcup B}(H).
\end{align*}

Let $G$ be a graph in $\calgr[A\sqcup B]$. If $A$ is not a union of vertices of $G$, then:
\[(\Delta_{A,B}\otimes \id_{\calgr[A]})\circ \delta_{A\sqcup B}(G)
=m_{1,3,24}\circ (\delta_A\otimes \delta_B)\circ \Delta_{A,B}(G)=0.\]
Otherwise:
\begin{align*}
(\Delta_{A,B}\otimes \id_{Gr(A)})\circ \delta_{A\sqcup B}(G)&=\sum_{\substack{\sim\triangleleft G,\\
\mbox{\scriptsize $A$ and $B$ are union of classes of $\sim$}}}
(G/\sim)_{\mid A}\otimes (G/\sim)_{\mid B}\otimes G/\sim\\
&=\sum_{\substack{\sim'\triangleleft G_{\mid A},\\\sim'\triangleleft G_{\mid B}}}(G_{\mid A})/\sim' \otimes (G_{\mid B})
/\sim''\otimes (G_{\mid A})|\sim'(G_{\mid B})|\sim''\\
&=m_{1,3,24}\circ (\delta_A\otimes \delta_B)\circ \Delta_{A,B}(G).
\end{align*}
Consequently, $\calgr$ is a double twisted bialgebra. \end{proof}

\begin{example} If $A$, $B$, $C$ are nonempty finite sets:
\begin{align*}
\Delta(\tdun{$A$})&=\tdun{$A$}\otimes 1+1\otimes \tdun{$A$},\\
\Delta(\tddeux{$A$}{$B$})&=\tddeux{$A$}{$B$}\otimes 1+\tdun{$A$}\otimes \tdun{$B$}+
\tdun{$B$}\otimes \tdun{$A$}+1\otimes \tddeux{$A$}{$B$},\\
\Delta(\tdtroisun{$A$}{$C$}{$B$})&=\tdtroisun{$A$}{$C$}{$B$}\otimes 1
+\tddeux{$A$}{$B$}\otimes \tdun{$C$}+\tddeux{$A$}{$C$}\otimes \tdun{$B$}+
\tdun{$B$}\tdun{$C$}\otimes \tdun{$A$}\\
&+\tdun{$C$}\otimes \tddeux{$A$}{$B$}+\tdun{$B$}\otimes \tddeux{$A$}{$C$}+\tdun{$A$}\otimes \tdun{$B$}\tdun{$C$}+
1\otimes \tdtroisun{$A$}{$C$}{$B$};\\ \\
\delta(\tdun{$A$})&=\tdun{$A$}\otimes \tdun{$A$},\\
\delta(\tddeux{$A$}{$B$})&=\tddeux{$A$}{$B$}\otimes \tdun{$A$}\tdun{$B$}+
\tdun{$A\sqcup B$} \hspace{6mm} \otimes \tddeux{$A$}{$B$},\\
\delta(\tdtroisun{$A$}{$C$}{$B$})&=\tdtroisun{$A$}{$C$}{$B$}\otimes \tdun{$A$}\tdun{$B$}\tdun{$C$}
+\tddeux{$A\sqcup B$}{$C$}\hspace{6mm}\otimes \tddeux{$A$}{$B$}\tdun{$C$}+
\tddeux{$A\sqcup C$}{$B$}\hspace{6mm}\otimes \tddeux{$A$}{$C$}\tdun{$B$}+
\tdun{$A\sqcup B\sqcup C$}\hspace{12mm}\otimes \tdtroisun{$A$}{$C$}{$B$}.
\end{align*}
\end{example}

\subsection{Morphism to $\calcomp$}

We now look for the unique morphism $\phi:\calgr\longrightarrow \calcomp$ of Theorem \ref{theo29}.

\begin{defi}
Let $G$ be a graph in $\calgr[A]$. A valid packed coloration of $G$ is a surjective map 
$c:A\longrightarrow [\max(c)]$, such that:
\begin{enumerate}
\item For any $x,y \in A$, if $x\sim_G y$, then $c(x)=c(y)$.
\item For any $x,y \in A$, if there is an edge in $G$ between 
the $\sim_G$-equivalence classes of $x$ and $y$, then $c(x)\neq c(y)$.
\end{enumerate}
The set of valid colorations of $G$ is denoted by $VC(G)$.
\end{defi}

\begin{prop} \label{prop32}
The unique morphism of double bialgebras from $\calgr$ to $\calcomp$ of Theorem \ref{theo17} 
is given on any graph $G$ by:
\begin{align}
\label{eq3}
\phi(G)=\sum_{c\in VC(G)} (c^{-1}(1),\ldots,c^{-1}(\max(c))).
\end{align}
\end{prop}

\begin{proof} By Theorem \ref{theo17}:
\begin{align*}
\phi(G)&=\sum_{(A_1,\ldots,A_k)} \varepsilon'^{\otimes k}\circ \Delta_{A_1,\ldots,A_k}(G)(A_1,\ldots,A_k)\\
&=\sum_{(A_1,\ldots,A_k)} \varepsilon'(G_{\mid A_1})\ldots \varepsilon'(G_{\mid A_k})(A_1,\ldots,A_k).
\end{align*}
By definition of $\varepsilon'$, $\varepsilon'(G_{\mid A_1})\ldots \varepsilon'(G_{\mid A_k})=1$ 
if the map $c$ sending any $a\in A_i$ to $i$ for any $i$ is a valid coloration of $G$,
and $0$ otherwise, which implies (\ref{eq3}).
\end{proof}

\begin{example} If $A$, $B$ and $C$ are finite sets:
\begin{align*}
\phi(\tdun{$A$})&=(A),\\
\phi(\tddeux{$A$}{$B$})&=(A,B)+(B,A),\\
\phi(\tdtroisun{$A$}{$C$}{$B$})&=(A,B,C)+(A,C,B)+(B,A,C)+(B,C,A)+(C,A,B)+(C,B,A)\\
&+(A,B\sqcup C)+(B\sqcup C,A).
\end{align*}
\end{example}

\section{The example of finite topologies}

\subsection{Double twisted bialgebra of finite topologies}

\begin{defi}
Let $T=(A,\leq)$ be a finite topology and let $\sim$ be an equivalence on $A$.
\begin{enumerate}
\item We define a second quasi-order $\leq_{T|\sim}$ on $A$ by the relation:
\begin{align*}
&\forall x,y\in A,&x\leq_{T|\sim} y&\mbox{ if }(x\leq y\mbox{ and }x\sim y).&
\end{align*}
\item We define a third quasi-order $\leq_{T/\sim}$ on $A$ as the transitive closure of the relation $R$ defined by:
\begin{align*}
&\forall x,y\in A,&x R y&\mbox{ if } (x\leq y \mbox{ or }x \sim y).
\end{align*}
In other words, $x \leq_{T/\sim} y$ if there exist $x_1,y_1,\ldots,x_n,y_n\in A$, such that:
\[x=x_1\sim y_1\leq_T \ldots \leq_T x_k\sim y_k=y.\]
\item We shall say that $\sim$ is $T$-compatible and we shall write $\sim \triangleleft T$ if the two following conditions are satisfied:
\begin{itemize}
\item The restriction of $T$ to any equivalence class of $\sim$ is connected.
\item The equivalences $\sim_{T/\sim}$ and $\sim$ are equal. In other words:
\begin{align*}
&\forall x,y\in A,& (x\leq_{T/\sim} y\mbox{ and }y\leq_{T/\sim} x) &\Longrightarrow x\sim y;
\end{align*}
note that the converse assertion trivially holds.
\end{itemize}
The set of $T$-compatible equivalences is denoted by $\CE(T)$. \end{enumerate}\end{defi}

\begin{lemma}
Let $T$ be a finite topology on a set $A$ and let $\sim$ be an equivalence on $A$.
\begin{enumerate}
\item The open sets of $T/\sim$ are the $\sim$-saturated open sets of $T$.
\item The open sets of $T|\sim$ are the sets of the form
\[(O_1\cap X_1)\cup \ldots \cup (O_n\cap X_n),\]
where $O_1,\ldots,O_n$ are open sets of $T$ and $X_1,\ldots,X_n$ are equivalence classes of $\sim$.
\item $\sim \in \CE(T)$ if, and only if, the following hold:
\begin{enumerate}
\item For any equivalence class $X$ of $\sim$, $T_{\mid X}$ is connected.
\item For any equivalence class $X$ of $\sim$, there exist an $\sim$-saturated open set $O$ and 
a $\sim$-saturated closed set $C$ such that $X=O\cap C$. 
\end{enumerate}
\end{enumerate}
\end{lemma}

\begin{proof}
1. Let $O$ be an open set of $T/\sim$. If $x\in O$ and $y\in A$, such that $x\leq_T y$ or $x\sim y$, then $x\leq_{T/\sim} y$, so $y\in O$:
$O$ is an $\sim$-saturated open set of $T$. Conversely, let $O$ be a $\sim$-saturated open set of $T$.
Let $x\in O$ and $x\leq_{T/\sim} y$. There exist $x_1,y_1,\ldots,x_n,y_n\in A$, such that:
\[x=x_1\sim y_1\leq_T \ldots \leq_T x_k\sim y_k=y.\]
As $O$ is open and saturated, for any $i$, $x_i,y_i \in O$, so $y\in O$: $O$ is an open set of $T/\sim$. \\

2. Let $O$ be an open set of $T|\sim$. For any equivalence class $X$ of $\sim$, we put:
\[O_X=\{y\in A,\exists x\in O\cap X, x\leq_T y\}.\]
Then $O_X$ is obviously an open set of $T$ containing $O\cap X$. If $y\in O_X\cap X$, there exists $x\in O_X$, such that $x\leq_T y$.
As $O$ is open, $y\in O$, so $y\in O\cap X$. We finally obtain that $O_X\cap X=O\cap X$, and:
\[O=\bigcup_{\mbox{\scriptsize $X$ class of $\sim$}} (O_X\cap X).\]
Conversely, let us assume that $O$ has the form 
\[O=(O_1\cap X_1)\cup \ldots \cup (O_n\cap X_n).\]
Let $x\in O$ and $x\leq_{T|\sim} y$. Then $x\leq_T y$ and $x\sim y$.
There exists $i$ such that $x\in O_i\cap X_i$. As $O_i$ is open, $y\in O_i$. As $X_i$ is an equivalence class of $\sim$, $y\in X_i$,
 so $y\in O_i\cap X_i\subseteq O$.\\

3. For any $X\subseteq A$, we put:
\begin{align*}
O(X)&=\{y\in A,\exists x\in X, x\leq_{T/\sim} y\},&
C(X)&=\{y\in A,\exists x\in X, y\leq_{T/\sim} x\}.
\end{align*}
Then $O(X)$ is an open set of $T/\sim$, so is a $\sim$-saturated open set of $T$; $C(X)$ is a closed set of $T/\sim$, 
so is a $\sim$-saturated closed set of $T$. Moreover, $X\subseteq O(X)\cap C(X)$.\\

$\Longleftarrow$. Let us assume that $\sim \in \CE(T)$ and let $X$ be an equivalence class of $\sim$.
By hypothesis, $T_{\mid X}$ is connected. Let $y\in O(X)\cap C(X)$. There exists $x,x'\in X$, 
such that $x\leq_{T/\sim} y\leq_{T/\sim} x'$. 
As $x\sim x'$, $x\sim_{T/\sim} x'$ so $x\sim_{T/\sim} y\sim_{T/\sim} x'$. Moreover, $\sim\in \CE(T)$, so $x\sim y\sim x'$, 
and $y\in X$. We proved that $X=O(X)\cap C(X)$.\\

$\Longrightarrow$. Let $O'$ be a $\sim$-saturated open set and  $C'$ be a $\sim$-saturated open set such that $X=O'\cap C'$.
Let $y\in O(X)$. There exists $x\in X$, such that $x\leq_{T/\sim} y$. As $O'$ is a open set of $T/\sim$ by the first point, 
$y\in O'$, so $O(X)\subset O'$. Similarly, $C(X)\subset C'$. Hence:
\[X\subseteq O(X)\cap C(X)\cap O'\cap C'=X.\]
We proved that $X=O(X)\cap C(X)$. 

Let $x,y\in A$, such that $x\sim_{T/\sim} y$. Let us denote by $X$ the equivalence class of $x$.
As $x\leq_{T/\sim} y$, $y\in O(X)$; as $y\leq_{T/\sim} x$, $y\in C(X)$. So $y\in X$ and $x\sim y$. 
\end{proof}

\begin{lemma} \label{lemme35}
Let $T$ be a finite topology on a set $A$ and $\sim,\sim'$ be equivalences on $A$ such that:
\begin{align*}
&\forall x,y\in A,&x\sim y\Longrightarrow x\sim' y.
\end{align*}
Then:
\begin{align*}
(T|\sim')|\sim&=T|\sim,&(T/\sim)/\sim'&=T/\sim'.
\end{align*}
If moreover $\sim'\in \CE(T)$, then:
\begin{align*}
(T|\sim')/\sim&=(T/\sim)|\sim'.
\end{align*}
\end{lemma}

\begin{proof}
Let $x,y\in A$.
\begin{align*}
x\leq_{(T\sim')|\sim} y&\Longleftrightarrow x\leq_{T|\sim'} y\mbox{ and } x\sim y\\
&\Longleftrightarrow x\leq_T y\mbox{ and }x\sim' y\mbox{ and } x\sim y\\
&\Longleftrightarrow x\leq_T y\mbox{ and } x\sim y\\
&\Longleftrightarrow x\leq_{T|\sim} y.
\end{align*}
So $(T|\sim')|\sim=T|\sim$. 

By definition, $\leq_{(T/\sim)/\sim'}$ is the transitive closure of the relation given by:
\[ x\leq_T y\mbox{ or }x\sim y\mbox{ or }x\sim' y,\]
or equivalently:
\[ x\leq_T y\mbox{ or }x\sim' y.\]
Therefore, $(T/\sim)/\sim'=T/\sim'$.

Let $x,y\in A$. If $x\leq_{(T|\sim')/\sim} y$, there exist $x_1,y_1,\ldots,x_k,y_k\in A$ such that:
\[x=x_1\sim y_1 \leq_{T|\sim'} x_2\sim y_2\leq_{T|\sim'}\ldots \leq_{T|\sim'} x_k\sim y_k=y.\]
Consequently:
\[x=x_1\sim y_1 \leq_T x_2\sim y_2\leq_T\ldots \leq_T x_k\sim y_k=y,\]
so $x\leq_{T/\sim} y$. Moreover,
\[x=x_1\sim' y_1\sim'\ldots\sim' x_k\sim' y_k=y,\]
so $x\sim' y$ and $x\leq_{(T/\sim)|\sim'} y$.
Let us assume that $\sim' \in \CE(T)$. If $x\leq_{(T/\sim)|\sim'} y$, then $x\leq_{T/\sim} y$ and $x\sim' y$.
There exist $x_1,y_1,\ldots,x_k,y_k\in A$, such that:
\[x=x_1\sim y_1 \leq_T x_2\sim y_2\leq_T\ldots \leq_T x_k\sim y_k=y.\]
Consequently:
\[x=x_1\leq_{T/\sim'} y_1\leq_{T/\sim'}\ldots \leq_{T/\sim'} x_k\leq_{T/\sim'} y_k\leq_{T/\sim'} y \leq_{T/\sim'} x,\]
so all the elements of $A$ appearing here are $\sim_{T/\sim'}$-equivalent. As $\sim'\in \CE(T)$, they are $\sim'$-equivalent. Hence:
 \[x=x_1\sim y_1 \leq_{T|\sim'} x_2\sim y_2\leq_{T|\sim'}\ldots \leq_{T|\sim'} x_k\sim y_k=y.\]
So $x\leq_{(T|\sim')/\sim} y$. Finally, $(T|\sim')/\sim=(T/\sim)|\sim'$. \end{proof}

\begin{prop}
Let $T$ be a finite topology on a set $A$ and $\sim \in \CE(T)$. 
\begin{enumerate}
\item The Hasse graph of $T|\sim$ is obtained from the Hasse graph of $T$ by deleting the edges 
of the graph of $T$ between non $\sim$-equivalent vertices.
\item The Hasse graph of $T/\sim$ is obtained from the Hasse graph of $T$ by:
\begin{itemize}
\item Contracting any equivalence class of $\sim$ to a single vertex.
\item Deleting the superfluous edges created in this process.
\end{itemize}
\end{enumerate}
\end{prop}

\begin{proof}
First, if $x\sim_T y$, then $x\sim_{T/\sim} y$, so $x\sim y$ as $\sim\in \CE(T)$. 
Hence, the sentence describing the construction of the Hasse graph of $T|\sim$ makes sense.

1. Obviously, $\sim_{T|\sim}=\sim_T$, so the vertices of the Hasse graph of $T|\sim$ are the vertices of the Hasse graph of $T$,
that is to say the classes of $\sim_T$. For any $x\in A$, we denote by $cl_T(x)$ its equivalence class for $\sim_T$.

Let us prove that the edges of the Hasse graph of $T|\sim$ are of the form $(cl_T(x),cl_T(y))$, with $x\sim y$.
Firstly, if $(cl_T(x),cl_T(y))$ is an edge of the Hasse graph of $T|\sim$, then $x\leq_{T|\sim} y$, 
so $x\leq_T y$ and $x\sim y$.
If $x\leq_T z\leq_T y$, then $x\leq_{T/\sim} z\leq_{T/\sim} y \leq_{T/\sim} x$, so $x\sim_{T/\sim} z\sim_{T/\sim} y$.
As $\sim\in \CE(T)$, $x\sim z\sim y$, so $x\leq_{T|\sim} z\leq_{T|\sim} y$. Consequently, $x\sim_T z$ or $y\sim_T z$.
Secondly, if $(cl_T(x),cl_T(y))$ is an edge of the Hasse graph of $T$ and $x\sim y$, then $x\leq_{T|\sim} y$.
If $x\leq_{T\sim} z\leq_{T|\sim} y$, then $x\leq_T z\leq_T y$, so $x\sim_T y$ or $y\sim_T z$. \\

2. As $\sim\in \CE(T)$, the vertices of the Hasse graph of $T/\sim$ are the classes of $\sim$.
Let us prove that for any edge $(cl_\sim(x),cl_\sim(y))$ of the Hasse graph of $T/\sim$, there exist $x'\sim x$ and $y'\sim y$
such that $(cl_T(x'),cl_T(y'))$ is an edge of the Hasse graph of $T$: this will prove the second point.
As $x\leq_T y$, there exist $x_1,y_1,\ldots,x_k,y_k \in A$ such that:
\[x=x_1\sim y_1 \leq_T \ldots \leq_T x_k\sim y_k=y.\]
Consequently:
\[x=x_1\leq_{T/\sim} y_1\leq_{T/\sim} \ldots \leq_{T/\sim} x_k \leq_{T/\sim} y_k=y.\]
As $(cl_\sim(x),cl_\sim(y))$ is an edge, there exists $i$ such that:
\[x=x_1\sim y_1\sim \ldots \sim y_i \leq_T x_{i+1}\sim \ldots \sim x_k\sim y_k=y.\]
We consider the set:
\[X=\{x'\in A,x'\sim x,\exists y'\in A, y'\sim y\mbox{ and }x'\leq_T y'\}.\]
We proved that $X$ is nonempty: let $x'\in X$, maximal for $\leq_T$. We now consider:
\[Y=\{y'\in A, y'\sim y\mbox{ and }x'\leq_T y'\}.\]
By definition of $x'$, this is nonempty: let $y\in Y'$, minimal for $\leq_T$. Then $x'\leq_T y'$, $x'\sim x$ and $y'\sim y$.
Let us assume that $x'\leq_T z\leq_T y$. Then $x'\leq_{T/\sim} z\leq_{T/\sim} y$, so $x'\sim z$ or $z\sim y'$.
\begin{itemize}
\item If $x'\sim z$, by maximality of $x'$, $x'\sim_T z$.
\item If $z\sim y'$, by minimality of $y'$, $z\sim_T y'$.
\end{itemize}  
So $(cl_T(x'),cl_T(y'))$ is an edge of the Hasse graph of $T$. \end{proof}

\begin{cor} \label{cor37}
Let $T$ be a finite topology on a set $A$ and let $\sim \in \CE(T)$.
\begin{enumerate}
\item Let $X \subseteq A$, $\sim$-saturated. Then $T_{\mid X}$ is connected if, and only if, $(T/\sim)_{\mid X}$ is connected.
\item Let $X\subseteq A$, included in a class of $\sim$. Then $T_{\mid X}$ is connected if, and only if, $(T|\sim)_{\mid X}$ is connected.
\end{enumerate}
\end{cor}

\begin{proof}
1. $\Longleftarrow$. Let $x$, $y \in X$. There exists a path from $cl_T(x)$ to $cl_T(y)$ in the Hasse graph of $T$, 
with all its vertices in $X$. By construction of the Hasse graph of $T/\sim$, there exists a path from $cl_\sim(x)$ 
to $cl_\sim(y)$ in the Hasse graph of $T/\sim$, with all its vertices in $X$. So $X$ is $T/\sim$-connected.

$\Longrightarrow$. Let $x$, $y\in X$. There exists a path from $cl_\sim(x)$ to $cl_\sim(y)$ in the Hasse graph of $T/\sim$, 
with all its internal vertices $x=x_0,x_1,\ldots,x_k=y$ be classes of elements of $x$. 
By construction of the Hasse graph of $T/\sim$,  for any $i$ there exist $x'_i,x''_i$ with an edge between 
$cl_T(x'_i)$ and $cl_T(x''_i)$ in the Hasse graph of $T$, $x'_i\sim x_i$ and $x''_i\sim x_{i+1}$.
As the classes of $\sim$ are connected, there exists a path in the Hasse graph of $T$ between $cl_T(x'_i)$ and $cl_T(x_i)$, 
with all its vertices being equivalent to $x'_i$. As $X$ is $\sim$-saturated, all these vertices belong to $X$. 
Similarly, there exists a path in the Hasse graph of $T$ between $cl_T(x''_i)$ and 
$cl_T(x_{i+1})$, with all its vertices in $X$. We finally obtain a path in the Hasse graph of $T$ between $cl_T(x)$ and $cl_T(y)$,
so $X$ is $T$-connected.\\

2. $\Longleftarrow$. There exists a path from $cl_T(x)$ to $cl_T(y)$ in the Hasse graph of $T$, with all its vertices in $X$.
As $X$ is included in a single class of $\sim$, all the edges between them belong to the Hasse graph of $T|\sim$, 
so $X$ is $T|\sim$-connected.

$\Longrightarrow$. Immediate, as the Hasse graph of $(T|\sim)_{\mid X}$ 
is a subgraph of the Hasse graph of $T_{\mid X}$.
\end{proof}

\begin{theo}
The species $\caltop$ is a double algebra with the following coproducts:
\begin{enumerate}
\item For any quasi-poset $T\in \caltop[A\sqcup B]$,
\[\Delta_{A,B}(T)=\begin{cases}
T_{\mid A}\otimes T_{\mid B}\mbox{ if $B$ is an open set of $T$},\\
0\mbox{ otherwise}.
\end{cases}\]
\item For any quasi-poset $T\in \caltop[A]$,
\[\delta_A(T)=\sum_{\sim \in \CE(T)} T/\sim \otimes T|\sim.\]
The counit $\varepsilon'$ of $\delta$ is given by:
\[\varepsilon'_A(T)=\begin{cases}
1\mbox{ if $\leq_T$ is an equivalence},\\
0\mbox{ otherwise}.
\end{cases}\]
\end{enumerate} \end{theo}

\begin{proof}
Let $T$ be a quasi-poset on a set $E$, and let $E=A\sqcup B\sqcup C$.  
\begin{align*}
&(\Delta_{A,B}\otimes \id)\circ \Delta_{A\sqcup B,C}(T)\\
&=\begin{cases}
T_{\mid A}\otimes T_{\mid B}\otimes T_{\mid C}\mbox{ if $C$ is an open set of $T$ and $B$ is an open set of $T_{\mid A\sqcup B}$},\\
0\mbox{ otherwise},
\end{cases}\\
&=\begin{cases}
T_{\mid A}\otimes T_{\mid B}\otimes T_{\mid C}\mbox{ if $B \sqcup C$ and $C$ are open sets of $T$},\\
0\mbox{ otherwise},
\end{cases}\\
&=\begin{cases}
T_{\mid A}\otimes T_{\mid B}\otimes T_{\mid C}
\mbox{ if $B\sqcup C$ is an open set of $T$ and $C$ is an open set of $T_{\mid B\sqcup C}$},\\
0\mbox{ otherwise},
\end{cases}\\
&=(\id \otimes \Delta_{B,C})\circ \Delta_{A,B\sqcup C}(T).
\end{align*}
So $\Delta$ is coassociative. 

Let $T$ and $T'$ be a finite topology on $A$ and $A'$, respectively. Let $B$ and $B'$ be finite sets such that $A\sqcup A'=B\sqcup B'$. 
\begin{align*}
\Delta_{B,B'}(TT')&=\begin{cases}
(TT')_{\mid B}\otimes (TT')_{\mid B'}\mbox{ if $B'$ is an open set of $TT'$},\\
0\mbox{ otherwise},
\end{cases}\\
&=\begin{cases}
T_{\mid A\cap B}T'_{\mid A'\cap B}\otimes T_{\mid A\cap B'}T'_{\mid A'\cap B'}\mbox{ if $A\cap B'$
and $A'\cap B'$ are open sets of $T$ and $T'$},\\
0\mbox{ otherwise},\\
\end{cases}\\
&=\Delta_{A\cap B,A\cap B'}(T)\Delta_{A'\cap B,A'\cap B'}(T').
\end{align*}
So $(\caltop,m,\Delta)$ is a twisted bialgebra.

Let $T\in \caltop[A]$ and $T'\in \caltop[A']$. If $\sim \in \CE(TT')$, as its classes are connected, $\sim$ can be uniquely written as 
$\sim=\sim'\sqcup \sim''$, with $\sim'\in \CE(T)$ and $\sim''\in \CE(T'')$. 
Conversely, if  $\sim'\in \CE(T)$ and $\sim''\in \CE(T'')$, then $\sim=\sim'\sqcup \sim'' \in \CE(TT')$. Hence:
\begin{align*}
\delta_{A\sqcup A'}(TT')&=\sum_{\sim'\in \CE(T),\:\sim''\in \CE(T')} (TT')/(\sim' \sqcup \sim'') \otimes (TT')|(\sim'\sqcup \sim'')\\
&=\sum_{\sim'\in \CE(T),\:\sim''\in \CE(T')}  (T/\sim')(T'/\sim'')\otimes (T|\sim')(T'|\sim'')\\
&=\delta_A(T)\delta_{A'}(T').
\end{align*}

Let $T\in \caltop[A]$. Let $\sim_0=\sim_T$ and $\sim_1$ be the equivalence which classes are the connected components of $T$.
Then both belong to $\CE(T)$. For any $\sim \in \CE(T)$:
\begin{align*}
\leq_{T/\sim} \mbox{ is an equivalence}&\Longleftrightarrow \sim=\sim_1,\\
\leq_{T|\sim} \mbox{ is an equivalence}&\Longleftrightarrow \sim=\sim_0.
\end{align*}
Consequently:
\begin{align*}
(\varepsilon'\otimes \id)\circ \delta(A)&=T|\sim_1+0=T,&
(\id \otimes \varepsilon')\circ \delta_A(T)&=T/\sim_0=T.
\end{align*}
So $\varepsilon'$ is the counit of $\delta$.\\

Let $T\in \caltop[A]$. 
\begin{align*}
(\delta_A\otimes \id)\circ \delta(T)&=\sum_{\sim\in \CE(T), \sim'\in \CE(T/\sim)} 
(T/\sim)/\sim' \otimes (T/\sim)|\sim'\otimes T|\sim,\\
(\id \otimes \delta_A)\circ \delta_A(T)&=\sum_{\sim'\in \CE(T),\sim\in \CE(T|\sim')} 
T/\sim'\otimes (T|\sim')/\sim \otimes (T|\sim')|\sim.
\end{align*}
We consider the two sets:
\begin{align*}
\mathcal{X}&=\{(\sim,\sim'), \sim\in \CE(T), \sim'\in \CE(T/\sim)\},&
\mathcal{Y}&=\{(\sim,\sim'),\sim'\in \CE(T),\sim\in \CE(T|\sim')\}.
\end{align*}
Let $(\sim,\sim') \in \mathcal{X}$. If $x\sim y$, $x\sim_{T/\sim} y$ and $x\sim_{(T/\sim)/\sim'} y$. As $\sim' \in \CE(T/\sim)$, $x\sim' y$.
Let $(\sim,\sim') \in \mathcal{Y}$. If $x\sim y$, then $x$ and $y$ are in the same connected component of $T|\sim'$ 
as the classes of $\sim$ are $T|\sim'$-connected. Hence, $x\sim' y$. By Lemma \ref{lemme35}:
\begin{align*}
(\delta_A\otimes \id)\circ \delta(T)&=\sum_{(\sim,\sim')\in \mathcal{X}} T/\sim' \otimes (T/\sim)|\sim'\otimes T|\sim,\\
(\id \otimes \delta_A)\circ \delta_A(T)&=\sum_{(\sim,\sim')\in \mathcal{Y}} T/\sim'\otimes  (T/\sim)|\sim'\otimes T|\sim.
\end{align*}
In order to prove the coassociativity of $\delta$, it is enough to prove that $\mathcal{X}=\mathcal{Y}$. 

Let $(\sim,\sim')\in \mathcal{X}$. Let $X'$ be an equivalence class of $\sim'$. Then it is $\sim'$-saturated, so also $\sim$-saturated.
By Corollary \ref{cor37}, as $X'$ is $T/\sim$-connected, it is also $T$-connected. Let $X$ be an equivalence class of $\sim$.
Then it is included in a single class of $\sim'$. By Corollary \ref{cor37}, 
as $X$ is $T$-connected, it is also $T|\sim'$-connected.

By Lemma \ref{lemme35}, $\sim_{T/\sim'}=\sim_{(T/\sim)/\sim'}=\sim'$ as $\sim'\in \CE(T/\sim)$. So $\sim'\in \CE(T)$.
Consequently, $\sim_{(T|\sim')/\sim}=\sim_{(T/\sim)|\sim'}$. For any $x,y\in A$, as $\sim \in \CE(T)$:
\begin{align*}
x\sim_{(T/\sim)|\sim'} y&\Longleftrightarrow x\sim' y\mbox{ and } x\sim_{T/\sim} y\\
&\Longleftrightarrow x\sim' y\mbox{ and }x\sim y\\
&\Longleftrightarrow x\sim y.
\end{align*}
 We obtain that $\sim \in \CE(T|\sim')$. Therefore, $\mathcal{X}\subseteq \mathcal{Y}$.\\
 
 Let $(\sim,\sim')\in \mathcal{Y}$. Let $X$ be a class of $\sim$. Then $X$ is $T|\sim'$-connected. 
 By Corollary \ref{cor37}, it is $T$-connected.
 Let $X'$ be a class of $\sim'$. Then it is $\sim$-saturated and $T$-connected. By Corollary \ref{cor37}, 
 it is $T/\sim$-connected.\\
 
 By Lemma \ref{lemme35}, $\sim_{(T/\sim)/\sim'}=\sim_{T/\sim'}=\sim'$ as $\sim' \in \CE(T)$, so $\sim'\in \CE(T/\sim)$.
 Let $x, y \in A$. If $x\sim_{T/\sim} y$, as $\sim \in \CE(T|\sim')$, $x$ and $y$ are 
 in the same connected component of $T|\sim'$,
 that is to say in the same class of $\sim'$ as $\sim'\in \CE(T)$, so $x\sim' y$. Consequently, $x\sim_{(T/\sim)|\sim'} y$, 
 so $x\sim_{(T|\sim')/\sim} y$ and $x\sim y$, as $\sim \in \CE(T|\sim')$.
 Hence, $\sim\in \CE(T)$. We proved that $\mathcal{Y}\subseteq \mathcal{X}$.\\
 
Let $T\in \caltop[A]$ and $A=I\sqcup J$. 
\begin{align*}
(\Delta_{I,J} \otimes \id)\circ \delta_A(T)&=\sum_{\substack{\sim \in \CE(T),\\ J\in O(T/\sim)}}
 (T/\sim)_{\mid I} \otimes (T/\sim)_{\mid J} \otimes T|\sim\\
&=\sum_{\substack{\sim \in \CE(T),\\ J\in O(T), \:\mbox{\scriptsize $\sim$-saturated}}} 
(T/\sim)_{\mid I} \otimes (T/\sim)_{\mid J} \otimes T|\sim\\
&=\begin{cases}
\displaystyle \sum_{\substack{\sim' \in \CE(T_{\mid I}),\\ \sim''\in \CE(T_{\mid J})}} 
(T_{\mid I})/\sim'\otimes (T_{\mid J})/\sim''\otimes T|(\sim' \sqcup \sim'') \mbox{ if $J\in O(T)$},\\
0\mbox{ otherwise},
\end{cases}\\
&=\begin{cases}
\displaystyle \sum_{\substack{\sim' \in \CE(T_{\mid I}),\\ \sim''\in \CE(T_{\mid J})}}
(T_{\mid I})/\sim'\otimes (T_{\mid J})/\sim''\otimes (T_{\mid I})|\sim' (T_{\mid J})|\sim''  \mbox{ if $J\in O(T)$},\\
0\mbox{ otherwise},
\end{cases}\\
&=m_{1,3,24}\circ (\delta_I \otimes \delta_J)\circ \Delta_{I,J}(T).
\end{align*} 
So $T$ is indeed a double twisted bialgebra. \end{proof}

\begin{example} If $A$, $B$, $C$ are finite sets:
\begin{align*}
\Delta(\tdun{$A$})&=\tdun{$A$}\otimes 1+1\otimes \tdun{$A$},\\
\Delta(\tddeux{$A$}{$B$})&=\tddeux{$A$}{$B$}\otimes 1+\tdun{$A$}\otimes \tdun{$B$}+1\otimes \tddeux{$A$}{$B$},\\
\Delta(\tdtroisun{$A$}{$C$}{$B$})&=\tdtroisun{$A$}{$C$}{$B$}\otimes 1
+\tddeux{$A$}{$B$}\otimes \tdun{$C$}+\tddeux{$A$}{$C$}\otimes \tdun{$B$}+
\tdun{$A$}\otimes \tdun{$B$}\tdun{$C$}+1\otimes \tdtroisun{$A$}{$C$}{$B$},\\
\Delta(\tdtroisdeux{$A$}{$B$}{$C$})&=\tdtroisdeux{$A$}{$B$}{$C$}\otimes 1+
\tddeux{$A$}{$B$}\otimes \tdun{$C$}+\tdun{$A$}\otimes \tddeux{$B$}{$C$}+
1\otimes \tdtroisdeux{$A$}{$B$}{$C$};\\ \\
\delta(\tdun{$A$})&=\tdun{$A$}\otimes \tdun{$A$},\\
\delta(\tddeux{$A$}{$B$})&=\tddeux{$A$}{$B$}\otimes \tdun{$A$}\tdun{$B$}+
\tdun{$A\sqcup B$} \hspace{6mm} \otimes \tddeux{$A$}{$B$},\\
\delta(\tdtroisun{$A$}{$C$}{$B$})&=\tdtroisun{$A$}{$C$}{$B$}\otimes \tdun{$A$}\tdun{$B$}\tdun{$C$}
+\tddeux{$A\sqcup B$}{$C$}\hspace{6mm}\otimes \tddeux{$A$}{$B$}\tdun{$C$}+
\tddeux{$A\sqcup C$}{$B$}\hspace{6mm}\otimes \tddeux{$A$}{$C$}\tdun{$B$}+
\tdun{$A\sqcup B\sqcup C$}\hspace{12mm}\otimes \tdtroisun{$A$}{$C$}{$B$},\\
\delta(\tdtroisdeux{$A$}{$B$}{$C$})&=\tdtroisdeux{$A$}{$B$}{$C$}\otimes \tdun{$A$}\tdun{$B$}\tdun{$C$}
+\tddeux{$A\sqcup B$}{$C$}\hspace{6mm}\otimes \tddeux{$A$}{$B$}\tdun{$C$}+
\tddeux{$A$}{$B\sqcup C$}\hspace{6mm}\otimes \tdun{$A$}\tddeux{$B$}{$C$}+
\tdun{$A\sqcup B\sqcup C$}\hspace{12mm}\otimes \tdtroisdeux{$A$}{$B$}{$C$}.
\end{align*}
\end{example}

\subsection{Morphism to $\calcomp$}

Let us describe the unique morphism $\phi:\caltop\longrightarrow \calcomp$ of double twisted bialgebra.

\begin{defi}
Let $T=(A,\leq_T)$ be a quasi-poset. 
\begin{enumerate}
\item We denote by $L'(T)$ the set of surjections $f:A\longrightarrow \underline{\max(f)}$ such that:
\begin{itemize}
\item For any $a,b\in A$, if $a\leq_T b$, then $f(a)\leq f(b)$.
\end{itemize}
\item We denote by $L(T)$ the set of surjections $f:A\longrightarrow \underline{\max(f)}$ such that:
\begin{itemize}
\item For any $a,b\in A$, if $a\leq_T b$, then $f(a)\leq f(b)$.
\item For any $a,b\in A$, if $a\leq_T b$ and $f(a)=f(b)$, then $a\sim_T b$.
\end{itemize}\end{enumerate}\end{defi}

\begin{prop} \label{prop40}
Let $\lambda \in \chara(\caltop)$. The unique twisted bialgebra morphism $\psi:\caltop\longrightarrow \calcomp$
such that $\varepsilon'\circ \psi=\lambda$ (Theorem \ref{theo17}) is given by the following:
for any quasi-poset $T=(A,\leq_T)$,
\begin{align*}
\psi(T)&=\sum_{f\in L'(T)}\lambda(T_{\mid f^{-1}(1)})\ldots\lambda(T_{\mid f^{-1}(\max(f))})
(f^{-1}(1),\ldots,f^{-1}(\max(f)).
\end{align*}
\end{prop}

\begin{proof}
We denote by $S(T)$ the set of sequences $(A_1,\ldots,A_k)$ of subsets of $A$ such that:
\begin{enumerate}
\item For any $i\in \underline{k}$, $A_i$ is nonempty and $A_1\sqcup \ldots \sqcup A_k=A$.
\item For any $i\in \underline{k}$, $A_i\sqcup \ldots \sqcup A_k$ is an open set of $T$.
\end{enumerate}
By Theorem \ref{theo17}:
\[\psi(T)=\sum_{(A_1,\ldots,A_k)\in S(T)}\lambda(T_{\mid A_1}\ldots T_{\mid A_k})
(A_1,\ldots,A_k).\]
The following map is a bijection:
\[\left\{\begin{array}{rcl}
L'(T)&\longrightarrow&S(T)\\
f&\longrightarrow&(f^{-1}(1),\ldots,f^{-1}(\max(f))),
\end{array}\right.\]
which implies the result.
\end{proof}

In order to obtain the morphism $\phi$ of Theorem \ref{theo29}, we  consider $\lambda=\varepsilon'$. 
Firstly, observe that $L(T)\subseteq L'(T)$. Moreover, for any $f\in L'(T)$, the following assertions are equivalent:
\begin{enumerate}
\item $f\in L(T)$.
\item $\forall i \in \underline{\max(f)}$, $\leq_{T_{\mid f^{-1}(i)}}$ is an equivalence.
\item $\forall i \in \underline{\max(f)}$, $\varepsilon'(T_{\mid f^{-1}(i)})=1$.
\item $\forall i \in \underline{\max(f)}$, $\varepsilon'(T_{\mid f^{-1}(i)})\neq 0$.
\end{enumerate}
We obtain:

\begin{theo}\label{theo41}
The unique double bialgebra morphism from $\caltop$ to $\calcomp$ of Theorem \ref{theo29} is given 
by the following: for any quasi-poset $T$,
\[\phi(T)=\sum_{f\in L(T)} (f^{-1}(1),\ldots,f^{-1}(\max(f))). \]
\end{theo}

\begin{example} If $A$, $B$ and $C$ are finite sets:
\begin{align*}
\phi(\tdun{$A$})&=(A),\\
\phi(\tddeux{$A$}{$B$})&=(A,B),\\
\phi(\tdtroisun{$A$}{$C$}{$B$})&=(A,B,C)+(A,C,B)+(A,B\sqcup C),\\
\phi(\tdtroisdeux{$A$}{$B$}{$C$})&=(A,B,C).
\end{align*}
\end{example}

\subsection{Structure of $\caltop$}

\begin{theo}\label{theo42}
The bialgebra $\caltop$ is isomorphic to a shuffle twisted bialgebra.
\end{theo}

\begin{proof} By theorem \ref{theo18}, it is enough to prove that $\caltop$ is cofree.
We shall apply Theorem \ref{theo11}. Let us consider the product $\oast$ (joint product)
defined on $\caltop$ by the following:
if $T=(A,\leq_T)$ and $S=(B,\leq_S)$ are quasi-posets, then, for any $x,y\in A\sqcup B$,
\begin{align*}
x\leq_{S\oast T} y&\mbox{ if }(x,y)\in A^2\mbox{ and }x\leq_S y\\
&\mbox{ or }(x,y)\in B^2\mbox{ and }x\leq_T y\\
&\mbox{ or }(x,y)\in A\times B.
\end{align*}
The open sets of $S\oast T$ are:
\begin{itemize}
\item The open sets of $T$;
\item The sets $O\sqcup B$, where $O$ is an open set of $S$.
\end{itemize}
Hence, if $A\sqcup B=A'\sqcup B'$:
\begin{enumerate}
\item If $B'$ is an open set of $T$, then $B'\subseteq B$ and:
\[\Delta_{A',B'}(S\oast T)=S\oast T_{\mid B\setminus B'}\otimes T_{\mid B'}
=(S\otimes 1)\oast\Delta_{B\setminus B',B'}(T).\]
\item If $A'$ is a closed set of $S$, then $A'\subseteq A$ and:
\[\Delta_{A',B'}(S\oast T)=S_{\mid A'}\otimes S_{\mid A\setminus A'}\oast T
=\Delta_{A',A\setminus A'}(S)\oast(1\otimes T).\]
\item Otherwise, $\Delta_{A',B'}(S\oast T)=0$; if $A'\subseteq A$, $\Delta_{A',A\setminus A}(S)=0$;
if $B'\subseteq B$, $\Delta_{B\setminus B',B'}(T)=0$.
\end{enumerate}
Therefore, the product $\oast$ satisfies the properties of Theorem \ref{theo11}. \end{proof}

\section{Graduation and homogeneous morphisms}

\subsection{Definition}

\begin{defi} \label{defi43}
Let $\calP$ be a twisted bialgebra. We shall say that $\calP$ is graded if there exists a  family of subspecies
$(\calP_n)_{n\geq 0}$ such that
\[\calP=\bigoplus_{n\geq 0} \calP_n,\]
with:
\begin{align*}
&\forall k,l\in \N,&m(\calP_k\otimes \calP_l)&\subseteq \calP_{k+l},\\
&\forall n\in \N,&\Delta(\calP_n)&\subseteq \bigoplus_{k+l=n}\calP_k\otimes \calP_l.
\end{align*}
We also  assume that for any finite set $A$:
\[\calP_0[A]=\begin{cases}
\calP[A]\mbox{ if }A=\emptyset,\\
(0)\mbox{ otherwise}.
\end{cases}\]
If $\calP$ and $\calQ$ are two graded twisted bialgebras and if $\phi:\calP\longrightarrow \calQ$ is a morphism,
we shall say that $\phi$ is homogeneous if for any $n\geq 0$, $\phi(\calP_n)\subseteq \calQ_n$.
\end{defi}

\begin{example}
\begin{enumerate}
\item $\calgr$ is graded: for any $n\in \N$, for any finite set $A$, $\calgr_n[A]$ is the vector space of 
graphs in $\calgr[A]$ with $n$ vertices. For example, if $A$ is a finite set:
\begin{align*}
\calgr_1[A]&=\Vect(\tdun{$A$}),\\
\calgr_2[A]&=\bigoplus_{(B,C)\in \comp[A]}\Vect(\tddeux{$B$}{$C$},\tdun{$B$}\tdun{$C$}).
\end{align*}
\item $\caltop$ is graded: for any $n\in \N$, for any finite set $A$, $\caltop_n[A]$ is the vector space of 
finite topologies on $A$ with $n$ equivalence classes. For example, if $A$ is a finite set:
\begin{align*}
\caltop_1[A]&=\Vect(\tdun{$A$}),\\
\caltop_2[A]&=\bigoplus_{(B,C)\in \comp[A]}\Vect(\tddeux{$B$}{$C$},\tddeux{$C$}{$B$},\tdun{$B$}\tdun{$C$}).
\end{align*}
\item Let $\calP$ be a species with $\calP[\emptyset]=(0)$. The shuffle twisted bialgebra $(coT(\calP),\shuffle,\Delta)$ 
is graded by the length: for any $n\in \N$, $coT(\calP)_n=\calP^{\otimes n}$.
\end{enumerate}
\end{example}

If $\calP$ is graded, for any $q\in \K$, we obtain a twisted bialgebra endomorphism:
\[\iota_q:\left\{\begin{array}{rcl}
\calP&\longrightarrow&\calP\\
x\in \calP_n[A]&\longrightarrow&q^n x.
\end{array}\right.\]

\begin{remark}
For any $q,q'\in \K$, $\iota_q\circ \iota_{q'}=\iota_{qq'}$, and $\iota_1=\id_{\calP}$.
Moreover, if $\phi:\calP\longrightarrow\calQ$ is a morphism between two graded twisted bialgebra,
it is homogeneous if, and only if, for any $q\in K$, $\iota_q\circ\phi=\phi\circ \iota_q$.
Conversely, if $\calP$ is a twisted bialgebra and $(\iota_q)_{q\in \K}$ is a family of twisted bialgebra endomorphims such that:
\begin{itemize}
\item For any $q,q'\in \K$, $\iota_q\circ \iota_{q'}=\iota_{qq'}$.
\item $\iota_1=\id_{\calP}$.
\end{itemize}
For any $n\in \N$, any finite set $A$, we put:
\[\calP_n[A]=\{x\in A,\:\forall q\in \K,\: \iota_q(x)=q^n x\}.\]
Then the direct sum $\calP'$ of the subspecies $\calP_n$ is a graded twisted bialgebra. 
\end{remark}

\subsection{Graduation of quasi-shuffles bialgebras}

\begin{prop}\label{prop44}
Let $\calQ$ be a commutative twisted bialgebra of the second type and let $\calP=(coT(\calQ),\squplus,\Delta,\delta)$ 
be the associated double twisted bialgebra. We denote by $\epsilon':\calP\longrightarrow \com$ the counit of $\delta$.
\begin{enumerate}
\item For any $q\in \K$, we put $\lambda_q=\epsilon'^q$. Then:
\[\lambda_q:\left\{\begin{array}{rcl}
\calP&\longrightarrow&\com\\
x_1\ldots x_n&\longrightarrow&H_n(q)\epsilon'(x_1)\ldots \epsilon'(x_n).
\end{array}\right.\]
\item For any $q\in \K$, we put $\theta_q=\id_{\calP}\leftarrow \lambda_q$. 
Then $\theta_q$ is a twisted bialgebra endomorphism
and, for any $q,q'\in \K$, $\theta_q\circ \theta_{q'}=\theta_{qq'}$.
\item For any $n\in \N$, for any finite set $A$, we put:
\[\calP_n[A]=\{x\in \calP[A],\:\forall q\in \K,\: \theta_q(x)=q^n x\}.\]
This gives $(\calP,\squplus,\Delta)$ a graduation.
\end{enumerate} \end{prop}

\begin{proof} First, note that for any $q,q'\in \K$, $\lambda_q*\lambda_{q'}=\lambda_{q+q'}$
and $\lambda_q\star \lambda_{q'}=\lambda_{qq'}$. \\

1. Let $x_1\ldots x_n\in \calP[A]$. Then:
\begin{align*}
\lambda_q(x_1\ldots x_n)&=\sum_{k=1}^\infty H_k(q)\sum_{1\leq i_1<\ldots<i_k<n}
\epsilon'^{\otimes k}(x_1 \ldots x_{i_1}\otimes \ldots \otimes x_{i_k+1}\ldots x_n)\\
&=H_n(q)\epsilon'^{\otimes n}(x_1\otimes \ldots \otimes x_n)+0\\
&=H_n(q)\epsilon'(x_1)\ldots \epsilon'(x_n).
\end{align*}

2. For any $q,q'\in \K$:
\[\theta_q\circ \theta_{q'}=(\id\leftarrow \lambda_q)\circ (\id\leftarrow \lambda_q)
=\id\leftarrow (\lambda_q\star \lambda_{q'})=\id\leftarrow \lambda_{qq'}=\lambda_{qq'}.\]

3. For any $x_1\ldots x_n\in \calP[A]$:
\begin{align*}
\theta_q(x_1\ldots x_n)&=x'_1\ldots x'_n\lambda_q(x''_1\squplus \ldots \squplus x''_n)
+\mbox{ a span of tensors of length $<n$}\\
&=x'_1\ldots x'_n\lambda_q(x''_1) \ldots \lambda_q( x''_n)+\mbox{ a span of tensors of length $<n$}\\
&=x'_1\ldots x'_nq\epsilon'(x''_1) \ldots q\epsilon'( x''_n)+\mbox{ a span of tensors of length $<n$}\\
&=q^nx_1\ldots x_n+\mbox{ a span of tensors of length $<n$}.
\end{align*}
Hence, for any $q\in \K$, $\theta_q[A]$ is diagonalisable and its eigenvalues are powers of $q$. 
For any integer $n$, we put:
\[\calP_n[A]=\{x\in \calP[A], \theta_2(x)=2^nx\}.\]
Then:
\[\calP=\bigoplus_{n\in \N} \calP_n.\]
Moreover, for any $n$, the restriction of the canonical projection $\pi_n$ on $\calP^{\otimes n}$ to $\calP_n$
is a bijection $\varpi_n:\calP_n\longrightarrow \calP^{\otimes n}$.
For any $q\in \K$, $\theta_q\circ \theta_2=\theta_2\circ \theta_q=\theta_{2q}$,
so for any $n$, $\theta_q(\calP_n)\subseteq \calP_n$. Moreover, if $x\in \calP_n[A]$,
\[\varpi_n \circ \theta_q(x)=q^n \varpi_n(x)=\varpi_n(q^nx).\]
As $\varpi_n$ is a bijection, $\theta_q(x)=q^nx$. We obtain:
\[\calP_n[A]=\{x\in \calP[A],\:\forall q\in \K,\: \theta_q(x)=q^n x\}.\]
If $x\in \calP_k[A]$ and $y\in \calP_l[B]$:
\[\theta_2(x\squplus y)=\theta_2(x)\squplus \theta_2(y)=(2^k x)\squplus (2^l y)=2^{k+l}(x\squplus y),\]
so $\squplus (\calP_k\otimes \calP_l)\subseteq \calP_{k+l}$.
If $x\in \calP_n[A]$ and $A=A_1\sqcup A_2$:
\[(\theta_2\otimes \theta_2)\circ \Delta_{A_1,A_2}(x)=\Delta_{A_1,A_2}\circ \theta_2(x)=2^n \Delta_{A_1,A_2}(x),\]
so:
\[\Delta_{A_1,A_2}(x)\in \bigoplus_{k+l=n} \calP_k[A_1]\otimes \calP_l[A_2].\]
Hence, $\calP$ is a graded twisted bialgebra.  \end{proof}

\begin{example}We obtain:
\begin{align*}
\theta_q(x_1)&=qx_1,\\
\theta_q(x_1x_2)&=q^2 x_1x_2+\frac{q(q-1)}{2}(x_1\cdot x_2),\\
\theta_q(x_1x_2x_3)&=q^3 x_1x_2x_3+\frac{q^2(q-1)}{2}(x_1\cdot x_2)x_3\\
&+\frac{q^2(q-1)}{2}x_1(x_2\cdot x_3)+\frac{q(q-1)(q-2)}{6}(x_1\cdot x_2\cdot x_3),\\
\theta_q(x_1x_2x_3x_4)&=q^4x_1x_2x_3x_4\\
&+\frac{q^2(q-1)}{2}(x_1\cdot x_2)x_3x_4+\frac{q^2(q-1)}{2}x_1(x_2\cdot x_3)x_4\\
&+\frac{q^2(q-1)}{2}x_1x_2(x_3\cdot x_4)+\frac{q^2(q-1)^2}{4}(x_1\cdot x_2)(x_3\cdot x_4)\\
&+\frac{q^2(q-1)(q-2)}{6}(x_1\cdot x_2\cdot x_3)x_4+\frac{q^2(q-1)(q-2)}{6}x_1(x_2\cdot x_3\cdot x_4)\\
&+\frac{q(q-1)(q-2)(q-3)}{24}(x_1\cdot x_2\cdot x_3\cdot x_4).
\end{align*}\end{example}

We shall give in Proposition \ref{prop46} another description of the graduation of $\calP$.

\subsection{Homogeneous morphisms to quasi-shuffle bialgebras}

\begin{prop} 
\label{prop46}
Let $\calQ$ be a commutative twisted bialgebra of the second type and let $\calP=coT(\calQ)$.
Let $\varrho:(\calP,\shuffle,\Delta)\longrightarrow(\calP,\squplus,\Delta)$ defined by:
\[\varrho((x_1\ldots x_n))=\sum_{1\leq i_1<\ldots<i_k<n}
\frac{1}{i_1!(i_2-i_1)!\ldots (n-i_k)!} (x_1\cdot \ldots \cdot x_{i_1})\ldots(x_{i_k+1}\cdot\ldots \cdot x_n).\]
This is a homogeneous twisted bialgebra morphism.
\end{prop}

\begin{proof}
Let $\lambda:\calP \longrightarrow \calQ$, defined by
\[\lambda(x_1\ldots x_n)=\frac{x_1\cdot \ldots \cdot x_n}{n!}.\]
This is a twisted algebra  morphism from $(\calP,\shuffle)$ to $(\calQ,\cdot)$: for any $x_1\ldots x_k$
 and $x_{k+1}\ldots x_{k+l}\in \cal P$, as $\calQ$ is commutative:
\begin{align*}
\lambda (x_1\ldots x_k\shuffle x_{k+1}\ldots x_{k+l})&=\sum_{\sigma \in Sh(k,l)}
\frac{1}{(k+l)!} x_{\sigma^{-1}(1)}\cdot \ldots \cdot x_{\sigma^{-1}(k+l)}\\
&=\frac{1}{(k+l)!} \sum_{\sigma \in Sh(k,l)} x_1\cdot \ldots \cdot x_{k+l}\\
&=\frac{1}{k!}{l!}  x_1\cdot \ldots \cdot x_{k+l}\\
&=\lambda(x_1\ldots x_k)\lambda(x_{k+1}\ldots x_{k+l}).
\end{align*}
Moreover, $\varrho$ is the unique twisted bialgebra morphism such that $\epsilon'\circ \varrho=\lambda$
of Theorem \ref{theo10}. \\

Let $x_1\ldots x_n \in \calQ$. For any $q\in \K$:
\[\pi \circ \varrho \circ \iota_q(x_1\ldots x_n)=\frac{q^n}{n!}x_1\cdot \ldots \cdot x_n.\]
Moreover, if $q\in \N$:
\begin{align*}
&\pi \circ \theta_q \circ \rho(x_1\ldots x_n)\\
&=\sum_{1\leqslant i_1<\ldots<i_k\leqslant n} \frac{1}{i_1!(i_2-i_1)!\ldots (n-i_k)!} 
\pi \circ \theta_q((x_1\cdot \ldots \cdot x_{i_1})\ldots (x_{i_k+1}\cdot \ldots \cdot x_n))\\
&=\sum_{1\leqslant i_1<\ldots<i_k\leqslant n} \frac{1}{i_1!(i_2-i_1)!\ldots (n-i_k)!} 
H_{k+1}(q)x_1\cdot \ldots \cdot \cdot x_n\\
&=\sum_{l=1}^n \sum_{a_1+\ldots+a_l=n}\frac{1}{a_1!\ldots a_l!} H_l(q)x_1\cdot \ldots \cdot \cdot x_n\\
&=\frac{1}{n!} \sum_{k=1}^n \sum_{a_1+\ldots+a_k=n}
\frac{(a_1+\ldots+a_k)!}{a_1!\ldots a_k!}\binom{q}{k}x_1\cdot \ldots \cdot \cdot x_n\\
&=\frac{1}{n!}\sum_{k=1}^{\min(n,q)}|\{f:\underline{n}\longrightarrow \underline{q},\: |f(\underline{n})|=k\}|
x_1\cdot \ldots \cdot \cdot x_n\\
&=\frac{1}{n!}|\{f:\underline{n}\longrightarrow \underline{q}\}|x_1\cdot \ldots \cdot \cdot x_n\\
&=\frac{q^n}{n!}x_1\cdot \ldots \cdot \cdot x_n\\
&=\pi \circ \rho \circ \iota_q(x_1\ldots x_n).
\end{align*}
As $\theta_q\circ \rho$ and $\rho \circ \iota_q$ are both coalgebra morphisms, by unicity in the universal property
of $coT(\calP)$, $\theta_q\circ \rho=\rho \circ \iota_q$ for any $q\in \N$. 
Hence, for any $n\geq 2$,  for any finite set $A$, we have:
\begin{align*}
coT(\calP)_n[A]&=\{x\in \calP[A], \theta_2(x)=2^nx\},\\
\calP^{\otimes n}[A]&=\{x\in \calP[A],\iota_2(x)=2^nx\},
\end{align*}
we obtain that $\rho(\calP^{\otimes n})\subseteq coT(\calP)_n$. So $\rho$ is homogeneous. \end{proof}

Consequently, in the particular case of $\calcomp$, for any $n\in \N$, for any finite set $A$:
\[\calcomp_n[A]=\Vect(\varrho((A_1,\ldots,A_n)),\: (A_1,\ldots,A_n)\in \comp[A]).\]
For example:
\begin{align*}
\calcomp_1[A]&=\Vect((A)),\\
\calcomp_2[A]&=Vect\left((A_1,A_2)+\frac{1}{2}(A_1\sqcup A_2),\: (A_1,A_2)\in \comp[A]\right),\\
\calcomp_3[A]&=Vect\left(\begin{array}{c}
(A_1,A_2,A_3)+\frac{1}{2}(A_1\sqcup A_2,A_3)+
\frac{1}{2}(A_1, A_2\sqcup A_3)+\frac{1}{6}(A_1\sqcup A_2\sqcup A_3),\\
 (A_1,A_2,A_3)\in \comp[A])
 \end{array}\right).
\end{align*}

\begin{cor}\label{cor47}
Let $\calP$ be a graded, connected twisted bialgebra. We denote by $\mor_B^0(\calP,\calcomp)$ the
set of homogeneous twisted bialgebra morphisms from $\calP$ to $\calcomp$
and by $L(\calP_1,\com)$ the set of species morphisms from $\calP_1$ to $\com$. The following map is a bijection:
\[\left\{\begin{array}{rcl}
\mor_B^0(\calP,\calcomp)&\longrightarrow&L(\calP_1,\com)\\
\phi&\longrightarrow&\varepsilon'\circ \phi_{\mid \calP_1}.
\end{array}\right.\]
\end{cor}

\begin{proof}  \textit{Injectivity}. 
 and let $\phi,\phi':\calP\longrightarrow \calcomp$
be two homogeneous twisted bialgebra morphism, such that 
$\varepsilon'\circ \phi_{\mid \calP_1}=\varepsilon'\circ \phi'_{\mid \calP_1}$.
We put $\lambda=\varepsilon'\circ \phi$ and $\lambda'=\varepsilon'\circ \phi'$.
Let us prove that $\phi=\phi'$. By Theorem \ref{theo29}, it is enough to prove that $\lambda=\lambda'$:
let us prove that  $\lambda_{\mid \calP_n}=\lambda'_{\mid \calP_n}$ by induction on $n$.
This is obvious if $n=0$ or $1$. If $n\geq 2$, then for any $q\in \K$, if $x\in \calP_n[A]$, by the induction hypothesis:
\begin{align*}
q^n \lambda(x)&=\lambda \circ \iota_q(x)\\
&=\lambda^q(x)\\
&=H_1(q)\lambda(x)+\sum_{k=2}^\infty \sum_{(A_1,\ldots,A_k)\in \comp[A]}
(\lambda-\varepsilon)^{\otimes k}\circ \Delta_{A_1,\ldots,A_k}(x).
\end{align*}
Hence, by the induction hypothesis:
\begin{align*}
(q^n-q)\lambda(x)&=\sum_{k=2}^\infty \sum_{(A_1,\ldots,A_k)\in \comp[A]}
(\lambda-\varepsilon)^{\otimes k}\circ \Delta_{A_1,\ldots,A_k}(x)\\
&=\sum_{k=2}^\infty \sum_{(A_1,\ldots,A_k)\in \comp[A]}
(\lambda'-\varepsilon)^{\otimes k}\circ \Delta_{A_1,\ldots,A_k}(x)\\
&=(q^n-q)\lambda'(x).
\end{align*}
Choosing $q=2$, as $n\geq 2$, we obtain $\lambda(x)=\lambda'(x)$. \\

\textit{Surjectivity}. Let $\lambda_1:\calP_1\longrightarrow \com$ be a species morphism.
We define $\lambda_0:\calP_0\longrightarrow \com$ by $\lambda_0(1_\calP)=1$.
Let us define $\lambda_n:\calP_n\longrightarrow \com$ for $n\geq 2$ by induction on $n$. If $x\in \calP_n[A]$, we put:
\begin{align*}
\lambda_n(x)&=\frac{1}{2^n-2}\sum_{k=2}^n \sum_{(A_1,\ldots,A_k)\in \comp[A]}
H_k(2)(\lambda_{\sharp A_1}\otimes \ldots \otimes \lambda_{\sharp A_k})\circ \Delta_{A_1,\ldots,A_k}(x)\\
&=\frac{1}{2^n-2}\sum_{(A_1,A_2)\in \comp[A]} (\lambda_{\sharp A_1} \otimes \lambda_{\sharp A_2})\circ \Delta_{A_1,A_2}(x),
\end{align*}
which we shortly write as:
\begin{align*}
\lambda_n(x)&=\frac{1}{2^n-2}\lambda(x^{(1)})\lambda(x^{(2)}).
\end{align*}
We define in this way a map $\lambda:\calP\longrightarrow \calcomp$, such that $\lambda^2=\lambda \circ \iota_2$.
Let us prove that $\lambda$ is a character: let $x\in \calP_k[A]$, $y\in\calP_l[B]$, and let us prove
that $\lambda(xy)=\lambda(x)\lambda(y)$ by induction on $k+l=n$.
If $k=0$ or $l=0$, we can assume that $x=1_\calP$ or $y=1_\calP$ and the result is obvious.
We now assume that $k,l\geq 1$. There is nothing to prove if $n=0$ or $1$. Otherwise, by the induction hypothesis:
\begin{align*}
\lambda(xy)&=\frac{1}{2^{k+l}-2}\left(\begin{array}{c}
2\lambda(x)\lambda(y)+\lambda(x^{(1)}y)\lambda(x^{(2)})\\
+\lambda(x^{(1)})\lambda(x^{(2)}y)+\lambda(xy^{(1)})\lambda(y^{(2)})\\
+\lambda(y^{(1)})\lambda(xy^{(2)})+\lambda(x^{(1)}y^{(1)})\lambda(x^{(2)}y^{(2)})
 \end{array}\right)\\
 &=\frac{1}{2^{k+l}-2}\left(\begin{array}{c}
2\lambda(x)\lambda(y)+2\lambda(x^{(1)})\lambda(x^{(2)})\lambda(y)\\
+2\lambda(x)\lambda(y^{(1)})\lambda(y^{(2)})+\lambda(x^{(1)})\lambda(y^{(1)})\lambda(x^{(2)})\lambda(y^{(2)})
 \end{array}\right)\\
 &=\frac{\lambda(x)\lambda(y)}{2^{k+l}-2} (2+2(2^k-2)+2(2^l-2)+(2^k-2)(2^l-2))\\
 &=\lambda(x)\lambda(y).
\end{align*}
So $\lambda$ is a character, $\lambda_{\mid \calP_1}=\lambda_1$ and $\lambda^2=\lambda \circ \iota_2$.
Let $\phi:\calP\longrightarrow\calcomp$ be the unique twisted bialgebra morphism such that $\varepsilon'\circ \phi=\lambda$. \\

 Let $q\in \K$. For any $x\in \calP[A]$:
\begin{align*}
\lambda_q\circ \phi(x)&=\sum_{k\geq 0}\sum_{(A_1,\ldots,A_k)\in \comp[A]}
\lambda^{\otimes k} \circ \Delta_{A_1,\ldots,A_k}(x) \lambda_q(A_1,\ldots,A_k)\\
&=\sum_{k\geq 0}H_k(q)\sum_{(A_1,\ldots,A_k)\in \comp[A]} \lambda^{\otimes k} \circ \Delta_{A_1,\ldots,A_k}(x)\\
&=\sum_{k\geq 0} H_k(q)(\lambda-\varepsilon)^{*k}(x)\\
&=\lambda^q(x).
\end{align*}
So $\lambda_q \circ \phi=\lambda^q$. For any finite set $A$:
\begin{align*}
\calcomp_n[A]&=\{x\in \calcomp[A],\: \theta_2(x)=2^nx\},\\
\calP_n[A]&=\{x\in \calP_n[A],\: \iota_2(x)=2^nx\}.
\end{align*}
As the characters $\lambda \circ \iota_2=\varepsilon'\circ \phi\circ \iota_2$ 
and $\lambda^2=\lambda_2\circ \phi=\varepsilon'\circ \theta_2\circ \phi$ are equal, by Theorem \ref{theo29}
$\phi\circ \iota_2=\theta_2\circ \phi$. Consequently, if $x\in \calP_n[A]$:
\[\theta_2\circ\phi(x)=\phi\circ \iota_2(x)=2^n \phi(x),\]
so $\phi(x)\in \calcomp_n[A]$: $\phi$ is homogeneous. Moreover, $\varepsilon'\circ \phi_{\mid \calP_1}=\lambda_1$. \end{proof}

\subsection{The example of graphs}

\begin{prop} \label{prop48}
\begin{enumerate}
\item Let $u=(u_k)_{k\geq 1}$ be a sequence of scalars. The following map is a homogeneous morphism of twisted bialgebras:
\[\phi_u:\left\{\begin{array}{rcl}
\calgr&\longrightarrow&\calcomp\\
G\in \calgr[A]&\longrightarrow&\displaystyle \left(\prod_{I\in V(G)}u_{\sharp I} \right)\prod^{\squplus}_{I\in V(G)} (I).
\end{array}\right.\]
All homogeneous morphisms from $\calgr$ to $\calcomp$ are obtained in this way.
\item Let $\lambda_0:\calgr\longrightarrow \comp$ be defined by $\lambda_0(G)=1$ for any graph $G$.
Then $\lambda_0$ is a character of $\calgr$ and for any $q\in \K$:
\[\lambda_0^q(G)=q^{\deg(G)} G.\]
\end{enumerate} 
\end{prop}

\begin{proof} For any finite set $A$, $\calgr_1[A]$ is one-dimensional, generated by $\tdun{$A$}$.
Hence, the species morphisms from $\calgr_1$ to $\com$ are given by:
\[\mu_u(\tdun{$A$})=u_{\sharp A},\]
where $u$ is a sequence of scalars. By Corollary \ref{cor47}, there exists a  unique homogeneous morphism 
of twisted bialgebra $\phi_u:\calgr\longrightarrow\calcomp$ such that $(\varepsilon'\circ \phi_u)_{\mid \calgr_1}=\mu_u$.
Let us consider the map:
\[\iota_u:\left\{\begin{array}{rcl}
\calgr&\longrightarrow&\calgr\\
G&\longrightarrow&\displaystyle \left(\prod_{I\in V(G)}u_{\sharp I} \right)G.
\end{array}\right.\]
This is obviously a homogeneous endomorphism of twisted bialgebras, and for any sequences $u$, $v$,
then $\phi_u\circ \iota_v$ is a homogeneous morphism
and for any finite set $A$ of cardinal $n$:
\[\varepsilon'\circ \phi_u\circ \iota_v(\tdun{$A$})=u_nv_n,\]
so, putting $uv=(u_nv_n)_{n\geq 1}$, $\phi_{uv}=\phi_u\circ \iota_v$. It is now enough to describe
$\phi_\mathbf{1}$, with $\mathbf{1}_n=1$ for any $n$. Let us put $\lambda_0=\varepsilon'\circ \phi_\mathbf{1}$,
and let us prove that $\lambda_0(G)=1$ for any graph $G$.  We proceed by induction on $\deg(G)=n$.
If $n=1$, then $G=\tdun{$A$}$ and $\lambda_0(G)=\mathbf{1}_n=1$. Otherwise,
as $\lambda_0^2=\lambda_0 \circ \iota_2$:
\begin{align*}
\lambda_0(G)&=\frac{1}{2^n-2}\sum_{(A,B)\in \comp[V(G)]} \lambda(G_{\mid A})\lambda(G_{\mid B})\\
&=\frac{1}{2^n-2}\sum_{(A,B)\in \comp[V(G)]}1\\
&=\frac{1}{2^n-2}\sum_{k=1}^{n-1} \binom{n}{k}\\
&=1.
\end{align*}
So $\lambda_0(G)=1$ for any $G$. 
As $\phi_\mathbf{1}$ is homogeneous, $\lambda_0^q=\lambda_0\circ \iota_q$ for any $q\in \K$. \end{proof}

\begin{example} If $A$, $B$ and $C$ are finite sets:
\begin{align*}
\phi_{\mathbf{1}}(\tdun{$A$})&=(A),\\
\phi_{\mathbf{1}}(\tddeux{$A$}{$B$})&=(A)\squplus (B)\\
&=(A,B)+(B,A)+(A\sqcup B),\\
\phi_{\mathbf{1}}(\tdtroisun{$A$}{$C$}{$B$})&=(A)\squplus (B)\squplus (C)\\
&=(A,B,C)+(A,C,B)+(B,A,C)+(B,C,A)+(C,A,B)+(C,B,A)\\
&+(A\sqcup B,C)+(A\sqcup C,B)+(B\sqcup C,A)\\
&+(A,B\sqcup C)+(B,A\sqcup C)+(C,A\sqcup B)+(A\sqcup B\sqcup C).
\end{align*}\end{example}

\subsection{Application: acyclic orientations}

Let $q\in \K$. We shall consider the homogeneous morphism defined in Proposition \ref{prop48}, with
$u_n=q$ for any $n$:
\[\phi_q:\left\{\begin{array}{rcl}
\calgr&\longrightarrow&\calcomp\\
G\in \calgr[A]&\longrightarrow&\displaystyle q^{\deg(G)}\prod^{\squplus}_{I\in V(G)} (I).
\end{array}\right.\]
Then:
\[\phi_q=\phi\leftarrow\lambda_q,\]
where $\lambda_q=\varepsilon'\circ \phi_q$ is the character given by:
\[\lambda_q(G)=q^{\deg(G)}.\]

If $q\neq 0$, we put:
\[\phi_{chr_q}=\theta_{q^{-1}}\circ \phi\circ \iota_q.\]
This is a twisted bialgebra morphism from $\calgr$ to $\calcomp$. In particular, $\phi_{chr_1}=\phi$.

\begin{example} If $A$, $B$ and $C$ are finite sets: 
\begin{align*}
\phi_{chr_q}(\tdun{$A$})&=(A),\\
\phi_{chr_q}(\tddeux{$A$}{$B$})&=(A,B)+(B,A)+(1-q)(A\sqcup B),\\
\phi_{chr_q}(\tdtroisun{$A$}{$C$}{$B$})&=(A,B,C)+(A,C,B)+(B,A,C)+(B,C,A)+(C,A,B)+(C,B,A)\\
&+(1-q)(A\sqcup B,C)+(1-q)(A\sqcup C,B)+(B\sqcup C,A)\\
&+(1-q)(C,A\sqcup B)+(1-q)(B,A\sqcup C)+(A,B\sqcup C)+2(1-q)(A\sqcup B\sqcup C).
\end{align*}\end{example}

\begin{lemma} \label{lemme49}
Let $\lambda \in \chara(\calgr)$. Then $\lambda$ is invertible for $\star$ if, and only if, for any finite set $A$,
$\lambda(\tdun{$A$})\neq 0$.
\end{lemma}

\begin{proof}
$\Longrightarrow$. Let $\mu$ be the inverse of $\lambda$. For any finite set $A$,
$\lambda\star \mu(\tdun{$A$})=\lambda(\tdun{$A$})\mu(\tdun{$A$})=\varepsilon(\tdun{$A$})=1$,
so $\lambda(\tdun{$A$})\neq 0$. \\

$\Longleftarrow$. We define inductively scalars $\mu_1(G)$ and $\mu_2(G)$ for any graph $G$ by induction on the 
number $k$ of edges of $G$. If $k=0$, then $G=\tdun{$A_1$}\hspace{2mm}\ldots \tdun{$A_n$}\hspace{2mm}$, and we put:
\[\mu_1(G)=\mu_2(G)=\frac{1}{\lambda(\tdun{$A_1$}\hspace{2mm})\ldots \lambda(\tdun{$A_n$}\hspace{2mm})}.\]
Otherwise, let $C_1,\ldots,C_l$ be the connected components of $G$, and let 
$A_1,\ldots,A_k$ be the decorations of its vertices. Then:
\[\delta(G)=G\otimes  \tdun{$A_1$}\hspace{2mm}\ldots \tdun{$A_k$}\hspace{2mm}
+ \tdun{$C_1$}\hspace{2mm}\ldots \tdun{$C_l$}\hspace{2mm}\otimes G
+\mbox{sum of terms $G'\otimes G''$, with $\sharp E(G'),\: \sharp E(G'')<\sharp  E(G)$}.\]
We then put:
\begin{align*}
\mu_1(G)&=-\frac{1}{\lambda(\tdun{$C_1$}\hspace{2mm})\ldots \lambda(\tdun{$C_l$}\hspace{2mm})}
(\mu_1(\tdun{$A_1$}\hspace{2mm}\ldots \tdun{$A_k$}\hspace{2mm})\lambda(G)+\mu_1(G')\lambda(G'')),\\
\mu_2(G)&=-\frac{1}{\lambda(\tdun{$A_1$}\hspace{2mm})\ldots \lambda(\tdun{$A_k$}\hspace{2mm})}
(\lambda(G)\mu_2(\tdun{$C_1$}\hspace{2mm}\ldots \tdun{$C_l$}\hspace{2mm})+\lambda(G')\mu_2(G'')).
\end{align*} 
By construction, $\mu_1\star\lambda=\lambda\star \mu_2=\varepsilon'$. Hence, $\lambda$ is invertible,
and its inverse is $\mu_1=\mu_2$. \end{proof}

Consequently, if $q\neq 0$, $\lambda_q$ is invertible. 
We denote by $\lambda_{chr}$ the inverse of $\lambda_1$ and for any $q\neq 0$, 
$\lambda_{chr_q}=\lambda_{chr}\star (\varepsilon'\circ \phi_{chr_q})$. Then:
\[\phi_{chr_q}=\phi\leftarrow  (\varepsilon'\circ \phi_{chr_q})
=(\phi_1\leftarrow \lambda_{chr})\leftarrow  (\varepsilon'\circ \phi_{chr_q})
=\phi_1\leftarrow \lambda_{chr}.\]
In particular, $\lambda_{chr_q}=\lambda_{chr}$.
Consequently:

\begin{cor}
For any graph $G\in \gr'[A]$:
\begin{align*}
\phi_{chr_q}(G)	=\sum_{\sim\triangleleft G} \lambda_{chr_q}(G\mid \sim)
 C_\sim^{(1)}\squplus\ldots \squplus C_\sim^{(\cl(\sim))},
\end{align*}
where $C_\sim^{(1)},\ldots,C_\sim^{(\cl(\sim))}$ are the equivalence classes of $\sim$.
\end{cor}

\begin{prop}[Extraction-contraction principle]
Let $G$ be a graph and $e$ be an edge of $G$. We denote by $G\setminus e$ the graph obtained in deleting the edge $e$
and by $G/e$ the graph obtained in  contracting the edge $e$, the two extremities of $e$ being identified with their union.
Then, for any graph $G$, for any edge $e$ of $G$:
\begin{align*}
\phi_{chr_1}(G)&=\phi_{chr_1}(G\setminus e)-\phi_{chr_1}(G/e).
\end{align*}
For any $q\neq 0$:
\begin{align*}
\phi_{chr_q}(G)&=\phi_{chr_q}(G\setminus e)-q\phi_{chr_q}(G/e).
\end{align*}
In particular, $\lambda_{ao}(G)=\varepsilon'\circ \phi_{chr_{-1}}(G)$ is the number of acyclic orientations of $G$.
\end{prop}

\begin{proof} We denote by $X$ and $Y$ the 
set of valid colorations of $G\setminus e$ is the set of colorations $c$ of $G$
such that the conditions for $c$ to be valid is satisfied, except maybe for pairs of elements $(x,y)\in X\times Y$.
The set of valid colorations of $G/e$ is the set of colorations $c$ of $G$
such that the conditions for $c$ to be valid is satisfied, except for any pair of elements $(x,y)\in X\times Y$.
Hence, $VC(G\setminus e)=VC(G)\sqcup VC(G/e)$, which gives
$\phi_{chr_1}(G\setminus e)=\phi_{chr_1}(G)-\phi_{chr_1}(G/e)$.\\

If $q\neq 0$:
\begin{align*}
\phi_{chr_q}(G)&=q^{\deg(G)} \theta_{q^{-1}}\circ\phi_{chr_1}(G)\\
&=q^{\deg(G)}\theta_{q^{-1}}\left(\phi_{chr_1}(G\setminus e)-\phi_{chr_1}(G/e)\right)\\
&=\theta_{q^{-1}}\left(q^{\deg(G\setminus e)}\phi_{chr_1}(G\setminus e)-q^{\deg(G/e)+1}\phi_{chr_1}(G/e)\right)\\
&=\phi_{chr_q}(G\setminus e)-q \phi_{chr_q}(G/e).
\end{align*}

We denote by $\lambda_{ao}(G)$ the number of acyclic orientations of $G$
and we put $\mu(G)=\varepsilon'\circ \phi_{chr_{-1}}(G)$.  
Let us show that $\mu(G)=\lambda_{ao}(G)$  by induction on the number $k$ of edges of $G$. 
If $k=0$, then $\mu(G)=1$ and the result is obvious. If $k\geq 1$, let $e$ be an edge of $G$. 
Let $O$ be an orientation of $G$ and let $O'$ be the orientation of $G$ obtained by changing the orientation of $e$ in $G$.
Both $O$ and $O'$ give the same orientation of $G\setminus e$, which is acyclic if, and only if, $O$ or $O'$ is acyclic.
Both $O$ and $O'$ induce an orientation of $G/e$, which is acyclic if, and only if, both $O$ and $O'$ are acyclic.
Consequently, $\lambda_{ao}(G)=\lambda_{ao}(G\setminus e)+\lambda_{ao}(G/e)
=\mu(G\setminus e)+\mu(G/e)=\mu(G)$.  \end{proof}

\begin{cor} Let $q\neq 0$. Moreover for any graph $G$:
\[\lambda_{chr_q}(G)=q^{\deg(G)-\cc(G)}\lambda_{chr_1}(G).\]
Moreover, the character $\lambda_{chr_q}$ is invertible and for any graph $G$,
$\lambda^{\star-1}_{chr_q}(G)=q^{\deg(G)-\cc(G)}$. 
\end{cor}

\begin{proof}
As $\phi_{chr_q}=\phi_1\leftarrow \lambda_{chr_q}=\theta_{q^{-1}}\circ \phi_{chr_1}\circ \iota_q$:
\begin{align*}
\phi_{chr_1}\circ \iota_q&=\theta_q\circ (\phi_1\leftarrow \lambda_{chr_q})\\
&=(\theta_q\circ\phi_1)\leftarrow \lambda_{chr_q},\\
\varepsilon'\circ \iota_q&=\varepsilon'\circ \phi_{chr_1}\circ \iota_q\\
&=(\varepsilon'\circ \theta_q\circ \phi_1)\star \lambda_{chr_q}.
\end{align*}
For any graph $G$, if $V(G)=\{A_1,\ldots,A_k\}$:
\begin{align*}
\varepsilon'\circ \theta_q\circ \phi_1(G)&=\varepsilon'\circ \theta_q((A_1)\squplus\ldots \squplus (A_k))\\
&=\varepsilon'\circ \theta_q((A_1))\ldots \varepsilon'\circ \theta_q((A_k))\\
&=q^k\\
&=q^{\deg(G)}.
\end{align*}
Therefore:
\begin{align*}
\varepsilon'\circ\iota_q(G)&=q^{\deg(G)}\varepsilon'(G)\\
&=\sum_{\sim\triangleleft G} q^{\deg(G/\sim)} \lambda_{chr_q}(G\mid \sim)\\
&=\sum_{\sim\triangleleft G} q^{\cc(G\mid \sim)}\lambda_{chr_q}(G\mid \sim).
\end{align*}
So:
\[\varepsilon'(G)=\sum_{\sim\triangleleft G}  q^{\cc(G\mid \sim)-\deg(G\mid \sim)} \lambda_{chr_q}(G\mid \sim).\]
Hence, the inverse of the character defined by $\lambda(G)=1$ for any graph $G$ is given by:
\[\lambda^{\star-1}(G)=q^{\cc(G)-\deg(G)}\lambda_{chr_q}(G).\]
If $q=1$, $\lambda^{\star-1}(G)=\lambda_{chr_1}(G)$ and, if $q\neq 0$:
\[\lambda_{chr_q}(G)=q^{\deg(G)-\cc(G)} \lambda_{chr_1}(G).\]
Moreover:
\begin{align*}
\varepsilon'(G)&=\sum_{\sim\triangleleft G}  q^{\cc(G\mid \sim)-\deg(G\mid \sim)} \lambda_{chr_q}(G\mid \sim)\\
&=q^{\cc(G)-\deg(G)}\sum_{\sim\triangleleft G}  q^{\deg(G/\sim)-\cc(G/\sim)} \lambda_{chr_q}(G\mid \sim).
\end{align*}
If $\varepsilon'(G)\neq 0$, then $\cc(G)=\deg(G)$,so:
\[\sum_{\sim\triangleleft G}  q^{\deg(G/\sim)-\cc(G/\sim)} \lambda_{chr_q}(G\mid \sim)=\varepsilon'(G).\]
Therefore, $\lambda_{chr_q}^{\star-1}(G)=q^{\deg(G)-\cc(G)}$. \end{proof}

\begin{cor}\label{cor53}
\begin{enumerate}
\item The character $\lambda_{ao}$ is invertible and, for any graph $G$:
\[\lambda_{ao}^{\star-1}(G)=(-1)^{\deg(G)+\cc(G)} \lambda_{ao}(G).\]
\item The following maps are twisted bialgebra automorphisms, inverse one from the other:
\begin{align*}
\Gamma&:\left\{\begin{array}{rcl}
\calgr&\longrightarrow&\calgr\\
G\in \gr'[A]&\longrightarrow&\displaystyle\sum_{\sim\triangleleft G} \lambda_{ao}(G|\sim) G/\sim,
\end{array}\right.\\
\Gamma'&:\left\{\begin{array}{rcl}
\calgr&\longrightarrow&\calgr\\
G\in \gr'[A]&\longrightarrow&\displaystyle\sum_{\sim\triangleleft G} (-1)^{\cl(\sim)+\deg(G)}\lambda_{ao}(G|\sim) G/\sim.
\end{array}\right.
\end{align*}
Moreover, $\phi_{chr_{-1}}=\phi_{chr_1}\circ \Gamma$ and $\phi_{chr_1}=\phi_{chr_{-1}}\circ \Gamma'$.
\end{enumerate}
\end{cor}

\begin{proof}
1. By definition of $\lambda_{ao}$:
\[\lambda_{ao}=\varepsilon'\circ \phi_{chr_{-1}}=\varepsilon'\circ (\phi_1\leftarrow \lambda_{chr_{-1}})
=\lambda_{chr_1}^{\star-1}\star \lambda_{chr_{-1}}.\]
So $\lambda_{ao}^{\star-1}=\lambda_{chr_{-1}}^{\star-1}\star \lambda_{chr_1}$, and for  any graph $G$:
\begin{align*}
\lambda_{ao}^{\star-1}(G)&=\sum_{\sim \triangleleft G} (-1)^{\cl(\sim)-\cc(G)}\lambda_{chr_1}(G\mid \sim)\\
&=(-1)^{\deg(G)+\cc(G)}\sum_{\sim\triangleleft G}(-1)^{\deg(G\mid \sim)+\cc(G\mid \sim)} \lambda_{chr_1}(G\mid \sim)\\
&=(-1)^{\deg(G)+\cc(G)}\sum_{\sim\triangleleft G} \lambda_{chr_{-1}}(G\mid \sim)\\
&=(-1)^{\deg(G)+\cc(G)}\lambda_{chr_1}^{\star-1}\star \lambda_{chr_{-1}}(G)\\
&=(-1)^{\deg(G)+\cc(G)}\lambda_{ao}(G).
\end{align*}
2.  We obtain:
\[\phi_{chr_1} \circ \Gamma=\phi_{chr_1} \circ (\id \leftarrow \lambda_{ao})
=\phi_{chr_1}\leftarrow \lambda_{ao}=\phi_{chr_1}\leftarrow (\lambda_{chr_1}^{\star-1}\star \lambda_{chr_{-1}})
=\phi_1\leftarrow \lambda_{chr_{-1}}=\phi_{chr_{-1}}.\]
Note that $\Gamma=\id \leftarrow \lambda_{ao}$ and $\Gamma'=\id \leftarrow \lambda_{ao}^{\star-1}$,
so they are indeed twisted bialgebra morphisms, inverse one from each other. 
\end{proof}

To summarize, we obtain  twisted bialgebra morphisms $\phi_q$ and $\phi_{chr_q}$ from
$\calgr$ to $\calcomp$, such that the following diagrams are commutative:
\begin{align*}
&\xymatrix{\calgr\ar[r]^{\iota_q}\ar[d]_{\phi_{chr_{q'}}}&\calgr\ar[d]^{\phi_{chr_{\frac{q'}{q}}}}\\
\calcomp\ar[r]_{\theta_q}&\calcomp}&
&\xymatrix{\calgr\ar[r]^{\iota_q}\ar[d]_{\phi_{q'}}&\calgr\ar[d]^{\phi_{q'}}\\
\calcomp\ar[r]_{\theta_q}&\calcomp}
\end{align*}
We also have a twisted bialgebra automorphism $\Gamma$ of $\calgr$ making the following diagram commuting:
\[\xymatrix{\calgr\ar[rr]^{\Gamma} \ar[rd]_{\phi_{chr_{-1}}}&&\calgr\ar[ld]^{\phi_{chr_1}}\\
&\calcomp&}\]
They are related to several characters of $\calgr$:
\begin{align*}
\phi_{chr_q}&=\phi_1\leftarrow \lambda_{chr_q},&\Gamma&=\id\leftarrow \lambda_{ao}.
\end{align*}
For any graph $G$:
\begin{align*}
\lambda_{chr_q}(G)&=q^{\deg(G)-\cc(G)}\lambda_{chr_1}(G),&
\lambda_{chr_q}^{\star-1}(G)&=q^{\deg(G)-\cc(G)},&
\lambda_{ao}^{\star-1}(G)&=(-1)^{\deg(G)+\cc(G)}\lambda_{ao}(G).
\end{align*}

\begin{remark}
Taking the limit when $q$ goes to $0$, one gives a meaning to $\phi_{chr_0}$ and to $\lambda_{chr_0}$.
For any graph $G$:
\begin{align*}
\lambda_{chr_0}(G)&=\begin{cases}
1\mbox{ if }\deg(G)=\cc(G),\\
0\mbox{ otherwise}
\end{cases}\\
&=\varepsilon'(G).
\end{align*}
So $\phi_{chr_0}=\phi_1\leftarrow \lambda_{chr_0}=\phi_1\leftarrow \varepsilon'=\phi_1$.
\end{remark}

\subsection{The example of finite topologies}

\begin{defi}
For any quasi-poset $T=(A,\leq_T)$, we denote by $\HO(T)$ the set of heap-orders on $T$, that is to
say surjections $f:A\longrightarrow \underline{\max(f)}$ such that:
\begin{enumerate}
\item For any $a,b\in A$, if $a\leq_T b$, then $f(a)\leq f(b)$.
\item For any $a,b\in A$, $f(a)=f(b)$ if, and only if, $a\sim_T b$.
\end{enumerate}
Note that for any $f\in \HO(T)$, $\max(f)=\cl(T)$. We denote by $\ho(T)$ the cardinality of $\HO(T)$.
\end{defi}

\begin{lemma}
Let $T$ be a finite topology on a finite set $A$. Then:
\[\sum_{ \substack{\mbox{\scriptsize  $O$ non trivial}\\\mbox{\scriptsize  open set of $T$}}}
\frac{\cl(T)!}{\cl(T_{\mid O})!\cl(T_{\mid A\setminus O})!}\ho(T_{\mid A\setminus O})
\ho(T_{\mid O})=(2^{\cl(T)}-2)\ho(T).\]
\end{lemma}

\begin{proof}
We consider the two following sets:
\begin{enumerate}
\item $X$ is the set of pairs $(h,k)$, where $h\in \HO(T) $ and $1\leq k\leq \cl(T)$.
\item $Y$ is the set of triples $(O,h',h'')$, where $O$ is a non trivial open set of $T$,
$h'\in \HO(T_{\mid A\setminus O})$ and $h''\in \HO(T_{\mid O})$.
\end{enumerate}
For any $(h,k)$ in $X$, we put:
\begin{itemize}
\item $O=h^{-1}(\{k+1,\ldots,\cl(T)\})$. As $h\in \HO(T)$, this is an open set of $T$,
and as $1\leq k<\cl(T)$, it is non trivial.
\item $h'=h_{\mid A\setminus O}$ and $h''=h_{\mid O}-k$. As $h\in \HO(T)$,
$(O,h',h'')\in B$.
\end{itemize}
We define in this way a map $\vartheta:X\longrightarrow Y$.\\

If $(O,h,h')\in Y$, we put:
\begin{enumerate}
\item $k=\cl(A\setminus O)=\max(h')$.
\item $h:A\longrightarrow \underline{\cl(T)}$ defined by:
\[h(x)=\begin{cases}
h'(x)\mbox{ if }x\notin O,\\
h''(x)+k\mbox{ if }x\in O.
\end{cases}\]
It is not difficult to show that $h\in \HO(T)$.
\end{enumerate}
We define in this way a map $\vartheta':Y\longrightarrow X$, and it is immediate that $\vartheta\circ \vartheta'=\id_Y$
and $\vartheta'\circ \vartheta=\id_X$. So $\vartheta$ and $\vartheta'$ are bijections. We obtain:
\begin{align*}
&\sum_{ \substack{\mbox{\scriptsize  $O$ non trivial}\\\mbox{\scriptsize  open set of $T$}}}
 \frac{\cl(T)!}{\cl(T_{\mid O})!\cl(T_{\mid A\setminus O})!}\ho(T_{\mid A\setminus O})\ho(T_{\mid O})\\
&=\sum_{(O,h',h'')\in Y}\frac{\cl(T)!}{\cl(T_{\mid O})!\cl(T_{\mid A\setminus O})!}\\
&=\sum_{(h,k)\in X} \frac{n!}{k!(n-k)!}\\
&=\sum_{h\in \HO(T)} \sum_{k=1}^{n-1} \frac{n!}{k!(n-k)!}\\
&=\sum_{h\in \HO(T)}(2^n-2)\\
&=(2^n-2)\ho(T),
\end{align*}
where $n=\cl(T)$. \end{proof}

\begin{prop}\label{prop56}
\begin{enumerate}
\item Let $u=(u_k)_{k\geq 1}$ be a sequence of scalars. The following map is a homogeneous morphism of twisted bialgebras:
\[\phi_u:\left\{\begin{array}{rcl}
\caltop&\longrightarrow&\calcomp\\
T\in \caltop[A]&\longrightarrow&\displaystyle \left(\prod_{I\in CL(T)}u_{\sharp I} \right)\phi\leftarrow \lambda_{ho}(T).
\end{array}\right.\]
All homogeneous morphisms from $\caltop$ to $\calcomp$ are obtained in this way. 
\item We put:
\begin{align*}
\lambda_{ho}(T)=\frac{\ho(T)}{\cl(T)!}.
\end{align*}
Then $\lambda_{ho}$ is a character of $\caltop$. Moreover, for any $q\in \K$,
$\lambda_{ho}\circ \iota_q=\lambda_{ho}^q$.
\end{enumerate}\end{prop}

\begin{proof}
For any finite set $A$, $\caltop_1[A]$ is one-dimensional, generated by $\tdun{$A$}$.
Hence, the species morphisms from $\caltop_1$ to $\com$ are given by:
\[\mu_u(\tdun{$A$})=u_{\sharp A},\]
where $u$ is a sequence of scalars. Corollary \ref{cor47} implies that there exists a unique homogeneous
morphism of twisted bialgebra $\phi_u:\caltop\longrightarrow\calcomp$ such that 
$(\varepsilon'\circ \phi_u)_{\mid \calgr_1}=\mu_u$.

Let us consider the map:
\[\iota_u:\left\{\begin{array}{rcl}
\caltop&\longrightarrow&\caltop\\
T&\longrightarrow&\displaystyle \left(\prod_{I\in CL(T)}u_{\sharp I} \right)T.
\end{array}\right.\]
This is obviously a homogeneous endomorphism of twisted bialgebras, and for any sequences $u$, $v$,
then $\phi_u\circ \iota_v$ is a homogeneous morphism
and for any finite set $A$ of cardinal $n$:
\[\varepsilon'\circ \phi_u\circ \iota_v(\tdun{$A$})=u_nv_n.\]
It is now enough to describe
$\phi_\mathbf{1}$, with $\mathbf{1}_n=1$ for any $n$. Let us put $\lambda_{ho}=\varepsilon'\circ \phi_\mathbf{1}$,
and let us prove that for any finite topology $T$ on a finite set $A$:
\[\lambda_{ho}(T)=\frac{\ho(T)}{\cl(T)!}.\]
We proceed by induction on $\cl(T)=n$.
If $n=1$, then $T=\tdun{$A$}$ and $\lambda_{ho}(T)=\mathbf{1}_n=1$. Otherwise,
as $\lambda_{ho}^2=\lambda_{ho} \circ \iota_2$, by the preceding Lemma:
\begin{align*}
\lambda_{ho}(T)&=\frac{1}{2^n-2}\sum_{\mbox{\scriptsize  $O$ non trivial open set of $T$}} 
\lambda_{ho}(T_{\mid V(T)\setminus O})\lambda_{ho}(T_{\mid O})\\
&=\frac{1}{2^n-2}\sum_{\mbox{\scriptsize  $O$ non trivial open set of $T$}} 
\frac{\cl(T)!}{\cl(T_{\mid O})!\cl(T_{\mid A\setminus O})!}\ho(T_{\mid O})\ho(T_{\mid A\setminus O})\\
&=\frac{\ho(T)}{n!}.
\end{align*}

As $\phi_\mathbf{1}$ is homogeneous, $\lambda_{ho}^q=\lambda_{ho}\circ \iota_q$ for any $q\in \K$. 
\end{proof}

\begin{example} If $A$, $B$ and $C$ are finite sets:
\begin{align*}
\phi_\mathbf{1}(\tdun{$A$})&=(A),\\
\phi_\mathbf{1}(\tddeux{$A$}{$B$})&=(A,B)+\frac{1}{2}(A\sqcup B),\\
\phi_\mathbf{1}(\tdtroisun{$A$}{$C$}{$B$})&=(A,B,C)+(A,C,B)+(A,B\sqcup C)\\
&+\frac{1}{2}(A\sqcup B,C)+\frac{1}{2}(A\sqcup C,B)+\frac{1}{3}
(A\sqcup B\sqcup C),\\
\phi_\mathbf{1}(\tdtroisdeux{$A$}{$B$}{$C$})&=(A,B,C)+\frac{1}{2}(A,B\sqcup C)+\frac{1}{2}(A\sqcup B,C)
+\frac{1}{6}(A\sqcup B\sqcup C).
\end{align*}
\end{example}

\subsection{Application: duality principle}

\begin{lemma}
Let $\lambda$ be a character on $\caltop$. It is inversible for $\star$ if, and only if, for any finite set $A$,
$\lambda(\tdun{$A$})\neq 0$. 
\end{lemma}

\begin{proof}
Similar to the proof of Lemma \ref{lemme49}.
\end{proof}

Let $q\in\K$. We denote by $\phi_q$ the homogeneous morphism associated to the sequence defined by $u_n=q$
for any $n$ in Proposition \ref{prop56}:
\[\phi_q:\left\{\begin{array}{rcl}
\caltop&\longrightarrow&\calcomp\\
T\in \caltop[A]&\longrightarrow&\displaystyle q^{\cl(T)}\phi\leftarrow \lambda_{ho}(T).
\end{array}\right.\]
The associated character $\lambda_q=\varepsilon'\circ \phi_q$ is defined by:
\[\lambda_q(T)=q^{\cl(T)}\frac{\ho(T)}{\cl(T)!}.\]
If $q\neq 0$, for any finite set $A$, $\lambda_q(\tdun{$A$})=q\neq 0$, so $\lambda_q$ is invertible.\\

For any $q\in \K$, nonzero, we put $\phi_{ehr_q}=\theta_{q^{-1}}\circ \phi\circ \iota_q:\caltop\longrightarrow \calcomp$.
These are twisted bialgebra morphisms, and $\phi_{ehr_1}=\phi$. We put:
\[\lambda_{ehr_q}=\lambda_1^{\star-1}\star (\varepsilon'\circ \phi_{ehr_q}),\]
which implies that for any $q\in \K$:
\[\phi_{ehr_q}=\phi_1\leftarrow \lambda_{ehr_q}.\]

\begin{example} If $A$, $B$ and $C$ are finite sets:
\begin{align*}
\phi_{ehr_q}(\tdun{$A$})&=(A),\\
\phi_{ehr_q}(\tddeux{$A$}{$B$})&=(A,B)+\frac{1-q}{2}(A\sqcup B),\\
\phi_{ehr_q}(\tdtroisun{$A$}{$C$}{$B$})&=(A,B,C)+(A,C,B)+(A,B\sqcup C)\\
&+\frac{1-q}{2}(A\sqcup B,C)+\frac{1-q}{2}(A\sqcup C,B)+\frac{(1-q)(2-q)}{6}(A\sqcup B\sqcup C),\\
\phi_{ehr_q}(\tdtroisdeux{$A$}{$B$}{$C$})&=(A,B,C)+\frac{1-q}{2}(A,B\sqcup C)+\frac{1-q}{2}(A\sqcup B,C)+
\frac{(1-q)(2-q)}{6}(A\sqcup B\sqcup C).
\end{align*}\end{example}

\begin{prop} \label{prop58}
For any finite topology $T$, $\varepsilon'\circ \phi_{ehr_{-1}}(T)=1$ and:
\[\phi_{ehr_{-1}}(T)=\sum_{f\in L'(T)} (\sharp f^{-1}(1),\ldots,\sharp f^{-1}(\max(f))).\]
\end{prop}

\begin{proof}
For any finite topology  $T$:
\begin{align*}
\varepsilon'\circ \phi_{ehr_{-1}}(T)&=(-1)^{\cl(T)}\varepsilon'\circ \theta_{-1} \circ \phi_{ehr_1}(T)\\
&=\sum_{f\in L(T)} (-1)^{\cl(T)+\max(f)}.
\end{align*}
Let us prove that this is equal to $1$ for any quasi-poset $T$. We proceed by induction on $n=\cl(T)$.
This is obvious if $n=0$. If $n\geq 1$, let us denote by $\min(T)$ the set of minimal classes (for $\leq_T$) 
of $T$. For any linear extension $f$ of $T$, $f^{-1}(T)$ is a subset of $\min(T)$; we obtain a bijection:
\begin{align*}
&\left\{\begin{array}{rcl}
\displaystyle \bigsqcup_{\emptyset \subsetneq I\subseteq \min(T)}
L(T_{\mid V(T)\setminus I})&\longrightarrow&L(T)\\
f \in L(T_{\mid V(T)\setminus I})&\longrightarrow&\tilde{f}:\left\{\begin{array}{rcl}
V(T)&\longrightarrow&\mathbb{N}\\
x&\longrightarrow&\begin{cases}
1\mbox{ if }x\in I,\\
f(x)+1\mbox{ otherwise}.
\end{cases}
\end{array}\right.
\end{array}\right.
\end{align*}
Hence, using the induction hypothesis on the quasi-posets $T_{\mid V(T)\setminus I}$:
\begin{align*}
\sum_{f\in L(T)} (-1)^{\cl(T)+\max(f)}&=\sum_{\emptyset \subsetneq I\subseteq \min(T)}
\sum_{f\in L(T_{\mid V(T)\setminus I})} (-1)^{\sharp I+\cl(T_{\mid V(T)\setminus I})+1+\max(f)}\\
&=-\sum_{\emptyset \subsetneq I\subseteq \min(T)}(-1)^{\sharp I}\\
&=-\sum_{I\subseteq \min(T)}(-1)^{\sharp I}+1\\
&=1.
\end{align*}
Hence, $\varepsilon'\circ \phi_{ehr_{-1}}(T)=1$ for any $T$. The formula for $\phi_{ehr_{-1}}(T)$ comes from 
Proposition \ref{prop40}. \end{proof}

\begin{prop}
For any quasi-poset $T$, for any $q\in \K$, nonzero:
\begin{align*}
\lambda_{ehr_q}(T)&=q^{\cl(T)-\cc(T)}\lambda_{ehr_1}(T),&
\lambda_{ehr_q}^{\star-1}(T)&=q^{\cl(T)-\cc(T)}\frac{\ho(T)}{\cl(T)!}.
\end{align*}\end{prop}

\begin{proof}
As $\phi_{ehr_q}=\phi_1\leftarrow\lambda_{ehr_q}=\theta_{q^{-1}}\circ \phi_{ehr_1}\circ \iota_q$:
\begin{align*}
\theta_q\circ \phi_{ehr_q}&=\phi_{ehr_1}\circ \iota_q&\varepsilon'\circ \iota_q&=\varepsilon'\circ \phi_{ehr_1}\circ \iota_q\\
&=(\theta_q\circ \phi_1)\leftarrow \lambda_{ehr_q},&&=\varepsilon'\circ \theta_q\circ \phi_{ehr_q}\\
&&&=(\varepsilon'\circ \theta_q\circ \phi_1)\star \lambda_{ehr_q}.
\end{align*}
Let $T$ be a finite topology. Then:
\begin{align*}
\varepsilon'\circ \theta_q\circ \phi_1(T)&=\frac{q^{\cl(T)}\ho(T)}{\cl(T)!} .
\end{align*}
Hence:
\begin{align*}
q^{\cl(T)}\varepsilon'(T)&=\sum_{\sim\triangleleft T}\frac{q^{\cl(\sim)}\ho(T/\sim)}{\cl(\sim)}\lambda_{ehr_q}(T\mid \sim),\\
\varepsilon'(T)&=\sum_{\sim\triangleleft T} \frac{\ho(T/\sim)}{\cl(T/\sim)!}
q^{\cc(T\mid\sim)-\cl(T\mid\sim)}\lambda_{ehr_q}(T\mid \sim)\\
&=\sum_{\sim\triangleleft T} \lambda_{ho}(T/\sim)q^{\cc(T\mid\sim)-\cl(T\mid\sim)}\lambda_{ehr_q}(T\mid \sim).
\end{align*}
Hence, the inverse of the character $\lambda_{ho}$ is given by:
\[\lambda_{ho}^{\star-1}(T)=q^{\cc(T)-\cl(T)} \lambda_{ehr_q}(T).\]
Consequently, for $q=1$, $\lambda_{ho}^{\star-1}=\lambda_{ehr_1}$ and, for any $q\neq 0$:
\[\lambda_{ehr_q}(T)=q^{\cl(T)-\cc(T)}\lambda_{ehr_1}(T).\]
Moreover:
\begin{align*}
\varepsilon'(T)&=q^{\cc(T)-\cl(T)}\sum_{\sim\triangleleft T} \lambda_{ho}(T/\sim)q^{\cl(T/\sim)-\cc(T/\sim)}
\lambda_{ehr_q}(T\mid \sim).
\end{align*}
If $\varepsilon'(T)\neq 0$, then $\leq_T$ is an equivalence and $\cl(T)=\cc(T)$, so:
\begin{align*}
\varepsilon'(T)&=\sum_{\sim\triangleleft T} \lambda_{ho}(T/\sim)q^{\cl(T/\sim)-\cc(T/\sim)}
\lambda_{ehr_q}(T\mid\sim).
\end{align*}
Hence, the inverse of $\lambda_{ehr_q}$ is given by $\lambda_{ehr_q}^{\star-1}(T)=q^{\cl(T)-\cc(T)}\lambda_{ho}(T)$. \end{proof}

\begin{cor}
The following maps are twisted bialgebra automorphisms, inverse one from each other:
\begin{align*}
\Gamma:&\left\{\begin{array}{rcl}
\caltop&\longrightarrow&\caltop\\
T&\longrightarrow&\displaystyle \sum_{\sim\triangleleft T} T/\sim,
\end{array}\right.\\
\Gamma':&\left\{\begin{array}{rcl}
\caltop&\longrightarrow&\caltop\\
T&\longrightarrow&\displaystyle \sum_{\sim\triangleleft T}(-1)^{\cl(\sim)+\cl(T)} T/\sim.
\end{array}\right.
\end{align*}
Moreover, $\phi_{ehr_{-1}}=\phi_{ehr_1} \circ \Gamma$ and $\phi_{ehr_1}=\phi_{ehr_{-1}}\circ \Gamma'$.
\end{cor}

\begin{proof} Let $\lambda$ be the character defined by $\lambda(T)=1$ for any finite topology $T$.
It is an invertible character and, as $\Gamma=\id\leftarrow \lambda$, $\Gamma$ is a twisted bialgebra automorphism.
By Proposition \ref{prop58}, $\lambda=\varepsilon'\circ \phi_{ehr_{-1}}$, so:
\[\lambda=\varepsilon'\circ (\phi_1\leftarrow \lambda_{ehr_{-1}})
=\varepsilon'\circ (\phi_{ehr_1}\leftarrow (\lambda_{ehr_1}^{\star-1}\star \lambda_{ehr_{-1}}))
=\lambda_{ehr_1}^{\star-1}\star \lambda_{ehr_{-1}}.\]
Hence:
\[\phi_{ehr_1}\circ \Gamma=\phi_{ehr_1}\leftarrow \lambda
=\phi_1\leftarrow (\lambda_{ehr_1}\star \lambda)=\phi_1\leftarrow (\lambda_{ehr_{-1}})=\phi_{ehr_{-1}}.\]
Moreover, $\Gamma^{-1}=\id\leftarrow \lambda^{\star-1}$. It remains to compute
$\lambda^{\star-1}=\lambda_{ehr_{-1}}^{\star-1}\star \lambda_{ehr_1}$. Let $T$ be a finite topology:
\begin{align*}
\lambda_{ehr_{-1}}^{\star-1}\star \lambda_{ehr_1}(T)&=\sum_{\sim\triangleleft T}
(-1)^{\cl(T/\sim)+\cc(T/\sim)} \lambda_{ehr_1}^{\star-1}(T/\sim) \lambda_{ehr_1}(T\mid \sim)\\
&=\sum_{\sim\triangleleft T}(-1)^{\cc(T\mid\sim)+\cc(T)} \lambda_{ehr_1}^{\star-1}(T/\sim) \lambda_{ehr_1}(T\mid \sim)\\
&=(-1)^{\cl(T)+\cc(T)}\sum_{\sim\triangleleft T}
(-1)^{\cc(T\mid \sim)+\cl(T\mid \sim)}\lambda_{ehr_1}^{\star-1}(T/\sim) \lambda_{ehr_1}(T\mid \sim)\\
&=(-1)^{\cl(T)+\cc(T)}\sum_{\sim\triangleleft T} \lambda_{ehr_1}^{\star-1}(T/\sim) \lambda_{ehr_{-1}}(T\mid \sim)\\
&=(-1)^{\cl(T)+\cc(T)} \lambda_{ehr_1}^{\star-1}\star \lambda_{ehr_{-1}}(T)\\
&=(-1)^{\cl(T)+\cc(T)} \lambda(T)\\
&=(-1)^{\cl(T)+\cc(T)}.
\end{align*}
As $\Gamma'=\id \leftarrow \lambda^{\star-1}$, $\Gamma'=\Gamma^{-1}$. \end{proof}

To summarize, we obtain  twisted bialgebra morphisms $\phi_q$ and $\phi_{ehr_q}$ from
$\caltop$ to $\calcomp$, such that the following diagrams are commutative:
\begin{align*}
&\xymatrix{\caltop\ar[r]^{\iota_q}\ar[d]_{\phi_{ehr_{q'}}}&\caltop\ar[d]^{\phi_{ehr_{\frac{q'}{q}}}}\\
\calcomp\ar[r]_{\theta_q}&\calcomp}&
&\xymatrix{\caltop\ar[r]^{\iota_q}\ar[d]_{\phi_{q'}}&\caltop\ar[d]^{\phi_{q'}}\\
\calcomp\ar[r]_{\theta_q}&\calcomp}
\end{align*}
We also have a twisted bialgebra automorphism $\Gamma$ of $\caltop$ making the following diagram commuting:
\[\xymatrix{\caltop\ar[rr]^{\Gamma} \ar[rd]_{\phi_{ehr_{-1}}}&&\caltop\ar[ld]^{\phi_{ehr_1}}\\
&\calcomp&}\]
They are related to several characters of $\caltop$:
\begin{align*}
\phi_{ehr_q}&=\phi_1\leftarrow \lambda_{ehr_q},&\Gamma&=\id\leftarrow \lambda.
\end{align*}
For any finite toplogy $T$:
\begin{align*}
\lambda_{ehr_q}(G)&=q^{\cl(T)-\cc(T)}\lambda_{ehr_1}(T),&
\lambda_{ehr_q}^{\star-1}(T)&=q^{\cl(T)-\cc(T)}\frac{\ho(T)}{\cl(T)!},\\
\lambda(T)&=1,&\lambda^{\star-1}(T)&=(-1)^{\cl(T)+\cc(T)}.
\end{align*}

\begin{remark}
Taking the limit when $q$ goes to $0$, one gives a meaning to $\phi_{ehr_0}$ and to $\lambda_{ehr_0}$.
For any finite topology $T$:
\begin{align*}
\lambda_{ehr_0}(T)&=\begin{cases}
1\mbox{ if }\cl(T)=\cc(T),\\
0\mbox{ otherwise}
\end{cases}\\
&=\varepsilon'(T).
\end{align*}
So $\phi_{ehr_0}=\phi_1\leftarrow \lambda_{ehr_0}=\phi_1\leftarrow \varepsilon'=\phi_1$.
\end{remark}

\section{Fock functors}

\subsection{Definition}

Let us now use the Fock functors of \cite{Aguiar}.

\begin{notation}
Let $V$ be a left $\mathfrak{S}_n$-module. The space of coinvariants of $V$ is:
\[Coinv(V)=\frac{V}{\Vect(x-\sigma.x,x\in V,\sigma \in \mathfrak{S}_n)}.\]
\end{notation}

In particular, if $\calP$ is a species, for any $n\geq 0$, $\calP[\underline{n}]$ is a left $\mathfrak{S}_n$-module.

\begin{defi}
\begin{enumerate}
\item (Full Fock functor). Let $\calP$ be a species. We put:
\[\fun(\calP)=\bigoplus_{n=0}^\infty \calP[\underline{n}].\]
If $\phi:\calP\longrightarrow \calQ$ is a species morphism, we put:
\[\fun(\phi):\left\{\begin{array}{rcl}
\fun(\calP)&\longrightarrow&\fun(\calQ)\\
x\in \calP[\underline{n}]&\longrightarrow& \phi[\underline{n}](x) \in \calQ[\underline{n}].
\end{array}\right.\]
This defines a functor from the category of species to the category of graded vector spaces.
\item (Bosonic Fock functor). Let $\calP$ be a species. We put:
\[\fdeux(\calP)=\bigoplus_{n=0}^\infty Coinv(\calP[\underline{n}]).\]
If $\phi:\calP\longrightarrow \calQ$ is a morphism of species, we put:
\[\fdeux(\phi):\left\{\begin{array}{rcl}
\fdeux(\calP)&\longrightarrow&\fdeux(\calQ)\\
\overline{x}\in Coinv(\calP[\underline{n}])&\longrightarrow& \overline{\phi[\underline{n}](x)} \in Coinv(\calQ[\underline{n}]).
\end{array}\right.\]
This defines a functor from the category of species to the category of graded vector spaces.
\end{enumerate}
\end{defi}

\begin{remark}
\begin{enumerate}
\item $\fdeux(\phi)$ is well-defined: if $x\in \calP[\underline{n}]$ and $\sigma \in \mathfrak{S}_n$, 
$\fun(\phi)(x-\sigma.x)=\fun(\phi)(x)-\sigma.\fun(\phi)(x)$. We have a commutative diagram:
\[\xymatrix{\fun(\calP)\ar@{->>}[d]\ar[r]^{\fun(\phi)} &\fun(\calQ)\ar@{->>}[d]\\
\fdeux(\calP)\ar[r]_{\fdeux(\phi)}&\fdeux(\calQ)}\]
where the vertical arrows are the canonical surjections.
\item For all $n\geq 0$, the $n$-th homogeneous component of $\fun(\calP)$ 
inherits a structure of left $\mathfrak{S}_n$-module.
This induces a trivial left action of $\mathfrak{S}_n$ on the $n$-th homogeneous component of $\fdeux(\calP)$.
\end{enumerate}
\end{remark}

\subsection{Fock functors applied to twisted (double) bialgebras}

\begin{notation}
Let $m,n\in \N$. We denote by $\sigma_{m,n}:\underline{m}\sqcup \underline{n}\longrightarrow \underline{m+n}$ 
the bijection defined by:
\begin{align*}
&\forall i\in \underline{m},&\sigma_{m,n}(i)&=i,\\
&\forall j\in \underline{n},&\sigma_{m,n}(j)&=j+m.
\end{align*}
For any $I\subseteq \underline{n}$, we denote by $\sigma_I$ the unique increasing bijection from $I$ to 
$\underline{\sharp I}$.
\end{notation}

\begin{theo} \label{theo62}
\begin{enumerate}
\item Let $\calP$ be a twisted algebra. 
Then $\fun(\calP)$ is a graded algebra, with the product defined by:
\begin{align*}
&\forall x\in \calP[\underline{m}],\: \forall y\in \calP[\underline{n}],
&x\cdot y&=\calP[\sigma_{m,n}]\circ m_\calP(x\otimes y).
\end{align*}
The unit is $1_\calP\in \calP[\emptyset]$.
\item Let $\calP$ be a twisted coalgebra. Then $\fun(\calP)$ is a graded coalgebra, with the coproduct defined by:
\begin{align*}
&\forall x\in \calP[\underline{n}],&\Delta(x)&=\sum_{I\subseteq \underline{n}} (\calP[\sigma_I] \otimes
 \calP[\sigma_{\underline{n}\setminus I}])\circ \Delta_{I,\underline{n}\setminus I}(x).
\end{align*}
We shall put $\Delta^{(I,\underline{n}\setminus I)}= (\sigma_I \otimes \sigma_{\underline{n}\setminus I})
\circ \Delta_{I,\underline{n}\setminus I}$ for any $I\subseteq \underline{n}$.
\item If $\calP$ is a twisted bialgebra, then $\fun(\calP)$ is a graded bialgebra.
\item If $\calP$ is a double twisted bialgebra, then $\fun(\calP)$ 
inherits a second coproduct $\delta$, defined by:
\begin{align*}
&\forall x\in \calP[\underline{n}],&\delta(x)&=\delta_{\underline{n}}(x).
\end{align*}
The triple $(\fun(\calP),m,\delta)$ is a bialgebra. Moreover, for any $I\subseteq \underline{n}$:
\[(\Delta^{(I,\underline{n}\setminus I)} \otimes \id_{\calP[\underline{n}]})\circ \delta
=(\id \otimes \id \otimes \calP[\alpha_{I,\underline{n}\setminus I}])\circ m_{1,3,24}
\circ (\delta \otimes \delta)\circ \Delta^{(I,\underline{n}\setminus I)},\]
where $\alpha_{i,\underline{n}\setminus I}$ is an element of $\mathfrak{S}_n$.
\end{enumerate}
\end{theo}

\begin{proof}
1. Let $x\in \calP[\underline{m}]$, $y\in \calP[\underline{n}]$ and $z\in \calP[\underline{p}]$. 
We consider the bijection $\sigma_{m,n,p}:\underline{m}\sqcup \underline{n}\sqcup \underline{p}\longrightarrow 
\underline{m+n+p}$ defined by:
\begin{align*}
\sigma_{m,n,p}(i)&=\begin{cases}
i\mbox{ if }i\in \underline{m},\\
i+m\mbox{ if }i\in \underline{n},\\
i+m+n \mbox{ if }i\in \underline{p}.
\end{cases}
\end{align*}
Then:
\begin{align*}
(x\cdot y)\cdot z&=\calP[\sigma_{m+n,p}]\circ (\calP[\sigma_{m,n}]\otimes \id)\circ m
\circ (m\otimes \id)(x\otimes y\otimes z)\\
&=\calP[\sigma_{m,n,p}]\circ m\circ (m\otimes \id)(x\otimes y\otimes z)\\
&=\calP[\sigma_{m,n,p}]\circ m\circ (\id\otimes m)(x\otimes y\otimes z)\\
&=\calP[\sigma_{m,n+p}]\circ(\id \otimes \calP[\sigma_{n,p}])\circ m\circ (\id\otimes m)(x\otimes y\otimes z)\\
&=x\cdot(y\cdot z).
\end{align*}

2. Let $x\in \calP[\underline{n}]$. 
\begin{align*}
&(\Delta \otimes \id)\circ \Delta(x)\\
&=\sum_{I\sqcup J\sqcup K=\underline{n}} (\calP[\sigma_I] \otimes \calP[\sigma_J]\otimes \id)
\circ (\Delta_{I,J}\otimes \id)
\circ (\id\otimes \calP[\sigma_K])
\circ \Delta_{I\sqcup J,K}(x)\\
&=\sum_{I\sqcup J\sqcup K=\underline{n}}(\calP[\sigma_I]\otimes \calP[\sigma_J]\otimes \calP[\sigma_K])
\circ \Delta_{I,J,K}(x)\\
&=\sum_{I\sqcup J\sqcup K=\underline{n}}(\id \otimes \calP[\sigma_J] \otimes 
\calP[\sigma_K])
\circ (\id \otimes \Delta_{J,K})\circ(\calP[\sigma_I]\otimes \id)
\circ \Delta_{I,J\sqcup K}(x)\\
&=(\id \otimes \Delta)\circ\Delta(x). 
\end{align*}

3. Let $x\in \calP[\underline{m}]$ and $y\in \calP[\underline{n}]$. 
\begin{align*}
\Delta(x\cdot y)&=(\calP\otimes \calP)[\sigma_{m,n}]\circ \Delta \circ m(x\otimes y)\\
&=(\calP\otimes \calP)[\sigma_{m,n}]\circ (m\otimes m)\circ 
(\id \otimes c_{\calP,\calP}\otimes \id)\circ (\Delta \otimes \Delta)(x\otimes y)\\
&=\sum_{\substack{I\sqcup J=\underline{m},\\I'\sqcup J'=\underline{n}}}
(\calP\otimes \calP)[\sigma_{m,n}]\circ (m\otimes m)\circ (\id \otimes c_{\calP,\calP}\otimes \id)\\
&\hspace{1cm} \circ (\calP[\sigma_I]\otimes \calP[\sigma_J]\otimes \calP[\sigma_{I'}]\otimes \calP[\sigma_{J'})]
\circ(\Delta_{I,J} \otimes \Delta_{I',J'})(x\otimes y)\\
&=\sum_{\substack{I\sqcup J=\underline{m},\\I'\sqcup J'=\underline{n}}}
(\calP\otimes \calP)[\sigma_{m,n}]\circ (m\otimes m)\circ (\calP[\sigma_I]\otimes \calP[\sigma_{I'}]
\otimes \calP[\sigma_J]\otimes \calP[\sigma_{J'}])\\
&\hspace{1cm} \circ (\id \otimes c_{\calP,\calP}\otimes \id)\circ(\Delta_{I,J} \otimes \Delta_{I',J'})(x\otimes y)\\
&=\sum_{\substack{I\sqcup J=\underline{m},\\I'\sqcup J'=\underline{n}}} \Delta^{(I,J)}(x)\cdot \Delta^{(I',J')}(y)\\
&=\Delta(x)\cdot \Delta(y).
\end{align*}

4. Let $x\in \calP[\underline{m}]$ and $y\in \calP[\underline{n}]$. 
\begin{align*}
\delta(x\cdot y)&=\delta\circ \calP[\sigma_{m,n}]\circ m(x\otimes y)\\
&=(\calP[\sigma_{m,n}]\otimes \calP[\sigma_{m,n}])\circ (m\otimes m) \circ 
(\id \otimes c_{\calP,\calP}\otimes \id)\circ (\delta \otimes \delta)(x\otimes y)\\
&=\delta(x)\cdot \delta(y).
\end{align*}
If $x\in \calP[\underline{m}]$ and $\underline{m}=I\sqcup J$:
\begin{align*}
(\Delta^{(I,J)}\otimes \id)\circ \delta(x)&=(\calP[\sigma_I]\otimes \calP[\sigma_J]\otimes \id)\circ m_{1,3,24}
\circ (\delta\otimes \delta)\circ \Delta_{I,J}(x),\\
m_{1,3,24} \circ (\delta \otimes \delta)\circ \Delta^{(I,\underline{n}\setminus I)}
&=m_{1,3,24}\circ(\calP[\sigma_I]\otimes \calP[\sigma_I]\otimes \calP[\sigma_J]\otimes \calP[\sigma_J])
\circ (\delta\otimes \delta)\circ \Delta_{I,J}(x)\\
&=(\calP[\sigma_I]\otimes \calP[\sigma_J]\otimes \calP[\sigma_I\sqcup \sigma_J]
\circ m_{1,3,24}\circ (\delta\otimes \delta)\circ \Delta_{I,J}(x)\\
&=(\id \otimes \id \otimes \calP[\sigma_{I,J}])\circ (\Delta^{(I,J)}\otimes \id)\circ \delta(x),
\end{align*}
with $\sigma_{I,J}=\sigma_I\sqcup \sigma_J$. This permutation is nondecreasing on $I$ and $J$ 
and sends $I$ to $\underline{\sharp I}$, so is the inverse of a shuffle. \end{proof}

\begin{cor}
Let $\calP$ be a twisted algebra [respectively coalgebra, bialgebra]. Then $\fdeux(\calP)$ 
is an algebra [respectively a coalgebra, a bialgebra], quotient of $\fun(\calP)$.
If $\calP$ is a double twisted bialgebra, then $\fdeux(\calP)$ is a bialgebra in cointeraction 
in the sense of \cite{Foissychrom,FoissyEhrhart}:
\[(\Delta \otimes \id)\circ \delta=m_{1,3,24}\circ (\delta \otimes \delta)\circ \Delta.\]
\end{cor}

\begin{proof}
Let $I$ be the subspace of $\fun(\calP)$ generated by the elements $x-\calP[\sigma](x)$, 
where $x\in \calP[\underline{n}]$, $\sigma\in \mathfrak{S}_n$, $n\geq 0$. Then $\fdeux(\calP)=\fun(\calP)/I$.\\

If $\calP$ is an algebra, let us prove that $I$ is an ideal of $\fun(\calP)$. For any $x\in \calP[\underline{m}]$, 
$\sigma \in \mathfrak{S}_m$, $y\in \calP[\underline{n}]$:
\begin{align*}
(x-\calP[\sigma](x))\cdot y&=\calP[\sigma_{m,n}]\circ m(x\otimes y-\calP[\sigma](x)\times y)\\
&=\calP[\sigma_{m,n}]\circ m(x\otimes y)-\calP[\sigma_{m,n}]\circ \calP[\id_m\sqcup \sigma]\circ m(x\otimes y)\\
&=X-\calP[\tau](X),
\end{align*}
with:
\begin{align*}
X&=\calP[\sigma_{m,n}]\circ m(x\otimes y),&\tau&=\sigma_{m,n}\circ (\id_m\sqcup \sigma)\circ \sigma_{m,n}^{-1}.
\end{align*}
So $(x-\calP[\sigma](x))\cdot y\in I$. Similarly, $y\cdot (x-\calP[\sigma](x))\in I$. \\

If $\calP$ is a coalgebra, let us prove that $I$ is a coideal of $\fun(\calP)$. 
For any $x\in \calP[\underline{m}]$, $\tau \in \mathfrak{S}_m$:
\begin{align*}
&\Delta(\calP[\tau](x))\\
&=\sum_{I\sqcup J=\underline{m}} (\calP[\sigma_{\tau(I)}]\otimes \calP[\sigma_{\tau(J)}])
\circ (\calP\otimes \calP[\tau])\circ \Delta_{I,J}(x)\\
&=\sum_{I\sqcup J=\underline{m}} (\calP[\sigma_{\tau(I)}]\otimes \calP[\sigma_{\tau(J)}])
\circ (\calP\otimes \calP[\tau]) \circ (\calP[\sigma_I^{-1}]\otimes \calP[\sigma_J^{-1}])
\circ (\calP[\sigma_I]\otimes \calP[\sigma_J])\circ \Delta_{I,J}(x)\\
&=\sum_{I\sqcup J=\underline{m}} \calP\otimes\calP[\tau_{I,J}]\circ (\calP[\sigma_I]\otimes \calP[\sigma_J])\circ \Delta_{I,J}(x),
\end{align*}
where $\tau_{I,J}$ is the permutation defined by:
\begin{itemize}
\item If $i\leq \sharp I$, $\tau_{I,J}(i)=\sigma_{\tau(I)}\circ \tau \circ \sigma_I^{-1}(i)$.
\item If $i> \sharp I$, $\tau_{I,J}(i)=\sigma_{\tau(J)}\circ \tau \circ \sigma_J^{-1}(i)$.
\end{itemize}
Hence, putting $k=\sharp I$, $\tau_{I,J}(\underline{k})=\underline{k}$. 
There exists $\tau'_{I,J} \in \mathfrak{S}_k$ and $\tau''_{I,J} \in \mathfrak{S}_{m-k}$ such that
$\tau_{I,J}=\tau'_{I,J}\sqcup \tau''_{I,J}$. 
Hence, putting $(\calP[\sigma_I]\otimes \calP[\sigma_J])\circ \Delta_{I,J}(x)=x'_{I,J}\otimes x''_{I,J}$:
\begin{align*}
\Delta(x-\calP[\tau](x))&=\sum_{I\sqcup J=\underline{m}} x'_{I,J}\otimes x''_{I,J}
-(\calP[\tau'_{I,J}]\otimes \calP[\tau''_{I,J}])(x'_{I,J}\otimes x''_{I,J})\\
&=\sum_{I\sqcup J=\underline{m}} (x'_{I,J}-\calP[\tau'_{I,J}](x'_{I,J}))\otimes 
x''_{I,J}+\calP[\tau'_{I,J}](x'_{I,J})\otimes (x''_{I,J}-\calP[\tau''_{I,J}](x''_{I,J}))\\
&\in I\otimes \fun(\calP)+\fun(\calP)\otimes I.
\end{align*}
So $I$ is a coideal for $\Delta$.\\

Let us assume that $\calP$ is a double bialgebra. By the preceding points, $(\fdeux(\calP),m,\Delta)$ is a bialgebra. 
 For any $x\in \calP[\underline{m}]$, $\tau \in \mathfrak{S}_m$, putting $\delta(x)=x'\otimes x''$:
\begin{align*}
\delta(x-\calP[\tau](x))&=\delta(x)-(\calP[\tau]\otimes \calP[\tau])\circ \delta(x)\\
&=(x'-\calP[\tau](x'))\otimes x''+\calP[\tau](x')\otimes (x''-\calP[\tau](x''))\\
&\in I\otimes \fun(\calP)+\fun(\calP)\otimes I.
\end{align*}
So $I$ is a coideal for $\delta$, and $\fdeux(\calP)$ inherits a second coproduct $\delta$. 
By Theorem \ref{theo62}, as the action of the symmetric groups of $\fdeux(\calP)$ are trivial:
\[(\Delta \otimes \id)\circ \delta=m_{1,3,24}\circ (\delta \otimes \delta)\circ \Delta.\]
So $(\fdeux(\calP),m,\Delta)$ is a bialgebra in the category of right $(\fdeux(\calP),m,\Delta)$-comodules. 
\end{proof}

\begin{remark}
Let $\calP$ be a double twisted bialgebra. Denoting by $\pi$ the canonical projection from $\fun(\calP)$ to $\fdeux(\calP)$,
we can define a right coaction of $\fdeux(\calP)$ on $\fun(\calP)$ by:
\[\rho=(\id \otimes \pi)\circ \delta:\fun(\calP)\longrightarrow \fun(\calP)\otimes \fdeux(\calP).\]
Then $(\fun(\calP),m,\Delta)$ is a bialgebra in the category or right $(\fdeux(\calP),m,\delta)$-comodules:
\[(\Delta \otimes \id)\circ \rho=m_{1,3,24}\circ (\rho \otimes \rho)\circ \Delta.\]
\end{remark}

\begin{example} Let us apply the Fock functors to $\com$. 
As the action of $\mathfrak{S}_n$ over $\com[\underline{n}]$ is trivial, $\fun(\com)=\fdeux(\com)$.
If we denote by $e_n$ the unit of $\K=\com[\underline{n}]$ for all $n\geq 0$, then a basis of $\fun(\com)$
is $(e_n)_{n\geq 0}$. For any $k$, $l\geq 0$, $e_k.e_l=e_{k+l}$, so putting $X=e_1$, for any $n\geq 0$,
$e_n=X^n$ and $\fun(\com)=\K[X]$. For any $n\geq 0$:
\begin{align*}
\Delta(X^n)&=\sum_{k=0}^n \binom{n}{k}X^k\otimes X^{n-k},&\delta(X^n)&=X^n\otimes X^n.
\end{align*}
So $\fun(\com)$ is the algebra $\K[X]$ with its usual coproducts $\Delta$ and $\delta$, as considered in
\cite{Foissychrom,FoissyEhrhart}.
\end{example}

\begin{example} Let us apply the Fock functors to $\fqsym$. For any total order $\preceq$ on $\underline{n}$, 
there exists a unique $\sigma \in \mathfrak{S}_n$, such that:
\begin{align*}
&\forall i,j\in \underline{n},&i\preceq j\Longleftrightarrow \sigma(i)\leq \sigma(j).
\end{align*}
Hence, $\fun(\fqsym)$ has a basis indexed by permutations. For any $\sigma \in \mathfrak{S}_n$:
\[\Delta(\sigma)=\sum_{i=0}^n \sigma^{(1)}_i\otimes \sigma^{(2)}_i,\]
where, if $\sigma$ is represented by the word $\sigma(1)\ldots \sigma(n)$:
\begin{itemize}
\item $\sigma^{(1)}_i$ is obtained by keeping only the letters $1,\ldots,i$, in the same order.
\item $\sigma^{(2)}_i$ is obtained by keeping only the letters $i+1,\ldots,n$, in the same order,
and by subtracting $i$ to any of these letters.
\end{itemize} 
If $\sigma \in \mathfrak{S}_k$ and $\tau\in \mathfrak{S}_l$:
\[\sigma \squplus \tau=\sum_{\alpha \in Sh(k,l)} \alpha \circ (\sigma \otimes \tau).\]
In other words, $\fun(\fqsym)$ is the Hopf algebra of permutations $\FQSym$ of \cite{DHT,MR3}.
Taking the coinvariants, $\fdeux(\fqsym)$ has a basis $(e_n)_{n\geq 0}$, where for all $n$
$e_n$ is the canonical projection of any permutation $\sigma \in \mathfrak{S}_n$.
For any $k,l,n\geq 0$:
\begin{align*}
\Delta(e_n)&=\sum_{i+j=n}e_i\otimes e_j,&
e_k \squplus e_l&=\binom{k+l}{k} e_{k+l}.
\end{align*}
Therefore, $\fdeux(\fqsym)$ is isomorphic to the Hopf algebra $\K[X]$, through the morphism:
\[\left\{\begin{array}{rcl}
\fdeux(\fqsym)&\longrightarrow&\K[X]\\
e_n&\longrightarrow&n!X^n.
\end{array}\right.\]
\end{example}

\begin{example} Let us apply the Fock functors to $\calcomp$.
\begin{enumerate}
\item The bialgebra $\fun(\calcomp)$ has for basis the set of set compositions of $\underline{n}$, with $n\geq 0$,
or, equivalently, the set of packed words, that is to say surjective maps from $\underline{n}$ to $\underline{k}$, for any $k\geq 1$. 
Its product is given by quasishuffles: if $u$, $v$ are packed words of respective length $k$ and $l$,
\begin{align*}
u\squplus v&=\sum_{\sigma \in \qsh(k,l)} \sigma \circ (u\sqcup  v[\max(u)]),
\end{align*}
where $v[\max(u)]=(v_1+\max(u))\ldots+(v_l+\max(u))$ if $v=v_1\ldots v_l$. For example:
\begin{align*}
(11)\squplus (11)&=(1122)+(2211)+(1111),\\
(12)\squplus (11)&=(1233)+(1322)+(2311)+(1211)+(1222).
\end{align*}
The coproduct $\Delta$ is given by extractions of letters: if $u$ is a packed word of length $n$,
\[\Delta(u)=\sum_{k=0}^n u_{\underline{k}} \otimes u_{\underline{n}\setminus \underline{k}},\]
where for any $I\subseteq \underline{n}$, $u_I$ is obtained by the following steps:
\begin{enumerate}
\item Keep only the letters of $u$ which belong to $I$: one obtains a word $u'_1\ldots u'_p$.
\item Let $\sigma:I\longrightarrow \underline{k}$ be the unique increasing bijection between $I$ and set $\underline{k}$,
where $k=|I|$. Then $u_I=f(u'_1)\ldots f(u'_p)$.
\end{enumerate}
For example:
\begin{align*}
\Delta((1123))&=1\otimes (1123)+(11)\otimes (12)+(112)\otimes (1)+(1123)\otimes 1.
\end{align*}
The second coproduct $\delta$ is given by:
\[\delta(u)=\sum_{(\sigma,\tau)\in \cont_{\max(u)}} \sigma \circ u\otimes \tau \circ u.\]
For example:
\begin{align*}
\delta((11))&=(11)\otimes (11),\\
\delta((12))&=(12)\otimes ((12)+(21)+(11))+(11)\otimes (12).
\end{align*}
In other words, $\fun(\calcomp)$ is the bialgebra $\WQSym$ \cite{NovelliThibon,NovelliThibon2,NovelliThibon3}.
\item The bialgebra $\fdeux(\calcomp)$ has for basis the set of compositions, that is to say finite sequence of
positive integers. The product is given by quasishuffles, the coproduct $\Delta$ by deconcatenation
and the coproduct $\delta$ by the action of the elements of $\cont_k$. For example, if $a,b,c,d\geq 1$:
\begin{align*}
(a)\squplus (b)&=(a,b)+(b,a)+(a+b),\\
(a,b)\squplus (c)&=(a,b,c)+(a,c,b)+(c,a,b)+(a+b,c)+(a,b+c),\\
(a,b)\squplus (c,d)&=(a,b,c,d)+(a,c,b,d)+(c,a,b,d)\\
&+(a,c,d,b)+(c,a,d,b)+(c,d,a,b)\\
&+(a,b+c,d)+(a+c,b,d)+(a+c,d,b)\\
&+(a,c,b+d)+(c,a,b+d)+(c,a+d,b)\\
&+(a+c,b+d);\\ \\
\Delta(a,b,c,d)&=1\otimes (a,b,c,d)+(a)\otimes (b,c,d)+(a,b)\otimes (c,d)\\
&+(a,b,c)\otimes (d)+(a,b,c,d)\otimes 1;\\ \\
\delta(a)&=(a)\otimes (a),\\
\delta(a,b)&=(a,b)\otimes ((a)\squplus (b))+(a+b)\otimes (a,b),\\
\delta(a,b,c)&=(a,b,c)\otimes ((a)\squplus (b)\squplus (c))+(a,b+c)\otimes ((a)\squplus (b,c))\\
&+(a+b,c)\otimes ((a,b)\squplus (c))+(a+b+c)\otimes (a,b,c).
\end{align*}
This is the Hopf algebra of quasisymmetric functions $\QSym$; see for example \cite{Hazewinkel} for more details.
\end{enumerate}
\end{example}

\subsection{The terminal property of $\K[X]$}

We observed that $\K[X]$, with its two coproducts $\Delta$ and $\delta$, is a bialgebra in cointeraction.

\begin{prop} \label{prop64}
Let $(A,\Delta)$ be a graded and connected coalgebra. We denote by $1_A$ the unique element of $A_0$
such that $\varepsilon(1_A)=1$. Let $\lambda:A\longrightarrow \K$ be any linear form such that $\lambda(1_A)=1$.
\begin{enumerate}
\item There exists a unique coalgebra morphism $\psi:(A,\Delta)\longrightarrow (\K[X],\Delta)$, 
such that $\varepsilon'\circ \psi=\lambda$.
\item Let us assume that $(A,m,\Delta)$ is a bialgebra. Then $\psi:(A,m,\Delta)\longrightarrow (\K[X],m,\Delta)$
is a bialgebra morphism if, and only if, $\lambda$ is a character of $A$.
\item Let us assume that $(A,m,\Delta,\delta)$ is a bialgebra in cointeraction. Then 
$\psi:(A,m,\Delta,\delta)\longrightarrow (\K[X],m,\Delta,\delta)$
is a  morphism of bialgebras in cointeraction  if, and only if, $\lambda$ is the counit $\varepsilon'$ of $(A,m,\delta)$.
\end{enumerate}
\end{prop}

\begin{proof}
1. We define $\psi(a)$ for any $a\in A_n$ by induction on $n$. If $n=0$, it is defined by $\psi(1_A)=1$. Otherwise,
we put $\Delta(a)=a\otimes 1_A+1_A\otimes 1+\tdelta(a)$, and by the connectivity condition:
\[\tdelta(a)\in \sum_{i=1}^{n-1}A_i\otimes A_{n-i}.\]
By the induction hypothesis, $X=(\psi\otimes \psi)\circ \tdelta(a)$ is well-defined. Moreover:
\begin{align*}
(\tdelta\otimes \id)(X)&=((\tdelta \circ \psi)\otimes \psi)\circ \tdelta(a)\\
&=(\psi\otimes \psi\otimes \psi)\circ (\tdelta \otimes \id)\circ \tdelta(a)\\
&=(\psi\otimes \psi\otimes \psi)\circ (\\id \otimes \tdelta)\circ \tdelta(a)\\
&=(\id \otimes \tdelta)\circ (\psi\otimes \psi)(a)\\
&=(\id \otimes \tdelta)(X).
\end{align*}
So $X\in \ker(\tdelta \otimes \id-\id \otimes \tdelta)$. It is not difficult to prove that in $\K[X]$:
\[ \ker(\tdelta \otimes \id-\id \otimes \tdelta)=Im(\tdelta).\]
Therefore, there exists a linear map $\psi':A_n\longrightarrow\K[X]$
such that for any $a\in A_n$, $(\psi\otimes \psi)\circ\tdelta(a)=\tdelta \circ \psi'(a)$. We then put:
\[\psi(a)=\psi'(a)-\varepsilon'\circ \psi'(a)X+\lambda(a)X.\]
As $X\in \ker(\tdelta)$, for any $a\in A_n$, $(\psi\otimes \psi)\circ\tdelta(a)=\tdelta \circ \psi(a)$ and
$\varepsilon'\circ \psi(a)=\lambda(a)$. \\

\textit{Unicity}. Let $\psi$ and $\psi'$ be two such morphisms. Let us prove that $\psi_{\mid A_n}=\psi'_{\mid A_n}$
by induction on $n$. This is obvious for $n=0$, as $\psi$ and $\psi'$ both sends the group-like element $1_A$
to the unique group-like element $1$ of $(\K[X],\Delta)$. Let us assume the result at all ranks $<n$.
For any $a\in A_n$:
\[\tdelta\circ \psi(a)=(\psi\otimes \psi)\circ \tdelta(a)=(\psi'\otimes \psi')\circ \tdelta(a)
=\tdelta \circ \psi'(a),\]
so $\psi(a)-\psi'(a)\in \ker(\tdelta)=\Vect(X)$: let us put $\psi(a)-\psi'(a)=\mu X$, with $\mu\in \K$. Then:
\[\mu=\varepsilon'(\psi(a)-\psi'(a))=\lambda(a)-\lambda(a)=0,\]
so $\psi_{\mid A_n}=\psi'_{\mid A_n}$. \\

2. $\Longrightarrow$. Then $\lambda=\varepsilon'\circ \psi$ is an algebra morphism by composition.
$\Longleftarrow$. Let us consider the two linear maps $\psi_1,\psi_2:A\otimes A\longrightarrow \K[X]$ defined
by $\psi_1=m\circ (\psi\otimes \psi)$ and $\psi_2=\psi\circ m$.
As $m$ is a coalgebra morphism, $\psi_1$ and $\psi_2$ are coalgebra morphisms.
Moreover, for any $a,b\in A$:
\begin{align*}
\varepsilon'\circ \psi_1(a\otimes b)&=\varepsilon'(\psi(a)\psi(b))\\
&=\varepsilon'\circ\psi(a)\varepsilon'\circ\psi(b)\\
&=\lambda(a)\lambda(b),\\
\varepsilon'\circ \psi_2(a\otimes b)&=\varepsilon'\circ \psi(ab)\\
&=\lambda(ab).
\end{align*}
As $\lambda$ is a character, $\varepsilon'\circ \psi_1$ and $\varepsilon'\circ \psi_2$ are equal. The coalgebra 
$A\otimes A$ being connected, by unicity in the first point, $\psi_1=\psi_2$, so $\psi$ is an algebra morphism.\\

3. $\Longrightarrow$. If $\psi$ is a morphism of bialgebras in cointeration, then $\varepsilon'\circ \psi=\varepsilon'$.
$\Longleftarrow$. The counit $\varepsilon'$ is a character of $A$, so $\psi$ is compatible with $m$ and $\Delta$.
Let us prove that for any $a\in A$, for any $k,l\geq 1$,
\[\delta \circ \psi(a)(k,l)=(\psi\otimes \psi)\circ \delta(a)(k,l).\]
Note that $\delta\circ \psi(a)(k,l)=\psi(a)(kl)$ by definition of the coproduct $\delta$ of $\K[X]$. If $k=1$:
\begin{align*}
\delta \circ \psi(a)(1,l)&=(\varepsilon'\otimes \id)\circ \delta \circ \psi(a)(l)\\
&=\psi(a)(l),\\
(\psi\otimes \psi)\circ \delta(a)(1,l)&=(\varepsilon'\circ \psi\otimes \psi)\circ \delta(a)(l)\\
&=(\varepsilon'\otimes \psi)\circ \delta(a)(l)\\
&=\psi(a)(l).
\end{align*}
Let us assume the result at rank $k$. As $\psi$ is compatible with $m$ and $\Delta$, and by the induction hypothesis:
\begin{align*}
\psi(a)((k+1)l)&=\psi(a)(kl+1l)\\
&=\Delta\circ \psi(a)(kl,1l)\\
&=(\delta \otimes \delta)\circ \Delta\circ \psi(a)(k,l,1,l)\\
&=m_{1,3,24}\circ (\delta \otimes \delta)\circ \Delta\circ \psi(a)(k,1,l)\\
&=(\psi\otimes \psi\otimes\psi)\circ m_{1,3,24}\circ (\delta \otimes \delta)\circ \Delta(a)(k,1,l)\\
&=(\psi\otimes \psi\otimes\psi)(\Delta \otimes \id)\circ \delta(a)(k,1,l)\\
&=(\psi\otimes \psi)\circ \delta(a)(k+1,l).
\end{align*}
Hence, for any $\delta\circ \psi=(\psi\otimes \psi)\circ \delta$. \end{proof}

Let us apply this to $\QSym$:

\begin{prop}
There exists a unique morphism $H:\QSym\longrightarrow \K[X]$ such that:
\begin{enumerate}
\item $H:(\QSym,\squplus,\Delta)\longrightarrow (\K[X],m,\Delta)$ is a bialgebra morphism.
\item $H:(\QSym,\squplus,\delta)\longrightarrow (\K[X],m,\delta)$ is a bialgebra morphism.
\end{enumerate}
For any composition $(a_1,\ldots,a_n)$:
\[H(a_1,\ldots,a_n)=H_n(X)=\frac{X(X-1)\ldots (X-n+1)}{n!}.\]
\end{prop}

\begin{proof}
Let us first prove that $H:(\QSym,\Delta)\longrightarrow (\K[X],\Delta)$ is a coalgebra morphism.
For any composition $(a_1,\ldots,a_n)$, for any $k,l\geq 0$:
\begin{align*}
\Delta \circ H((a_1,\ldots,a_n))(k,l)&=H(a_1,\ldots,a_n)(k+l)\\
&=H_n(k+l)\\
&=\binom{k+l}{n},\\
(H\otimes H)\circ \Delta((a_1,\ldots,a_n))(k,l)&=\sum_{i=0}^n H_i(k)H_{n-i}(l)\\
&=\sum_{i=0}^n \binom{k}{i}\binom{l}{n-i}\\
&=\binom{k+l}{n}, 
\end{align*}
so $\Delta \circ H=(H\otimes H)\circ \Delta$. For any composition $(a_1,\ldots,a_n)$:
\[\varepsilon'\circ H((a_1,\ldots,a_n))=H_n(1)=\delta_{n,1}+\delta_{n,0}=\varepsilon'((a_1,\ldots,a_n)),\]
so $H$ is a morphism of bialgebras in cointeraction. \end{proof}

Here is an application:

\begin{prop}
Let $q\in \K$. The following maps are bialgebra endomorphisms:
\begin{align*}
\fun(\theta_q)&:\left\{\begin{array}{rcl}
\WQSym&\longrightarrow&\WQSym\\
u&\longrightarrow&\displaystyle  \sum_{\substack{f:\underline{\max(u)}\longrightarrow \underline{k},\\
\mbox{\scriptsize surjective, non decreasing}}} H_{\sharp f^{-1}(1)}(q)\ldots H_{\sharp f^{-1}(\max(f))}(q) f\circ u,
\end{array}\right.\\
\fdeux(\theta_q)&:\left\{\begin{array}{rcl}
\QSym&\longrightarrow&\QSym\\
(a_1,\ldots,a_n)&\longrightarrow&\displaystyle \sum_{1\leq i_1<\ldots<i_k<n} H_{i_1}(q)H_{i_2-i_1}(q)\ldots H_{n-i_k}(q)\\
&&\hspace{1cm}(a_1+\ldots+a_{i_1},\ldots,a_{i_1+\ldots+i_k+1}+\ldots+a_n),
\end{array}\right.\\
\Theta_q&:\left\{\begin{array}{rcl}
\K[X]&\longrightarrow&\K[X]\\
P(X)&\longrightarrow&P(qX).
\end{array}\right.
\end{align*}
Moreover, the following diagram commutes:
\[\xymatrix{\WQSym\ar@{->>}[r]\ar[d]_{\fun(\theta_q)}&\QSym\ar[r]^{H}\ar[d]_{\fdeux(\theta_q)}&\K[X]\ar[d]^{\Theta_q}\\
\WQSym\ar@{->>}[r]&\QSym\ar[r]_{H}&\K[X]}\]
\end{prop}

\begin{proof}
The formulas for $\fun(\theta_q)$ and $\fdeux(\theta_q)$ are directly obtained, and it is obvious that 
$\Theta_q$ is a Hopf algebra endomorphism. It remains to prove that $H\circ \fdeux(\theta_q)=\Theta_q\circ H$.
Let $(a_1,\ldots,a_n)$ be a composition.
\begin{align*}
\varepsilon'\circ H\circ \fdeux(\theta_q)(a_1,\ldots,a_n)&=\varepsilon'\circ \fdeux(\theta_q)(a_1,\ldots,a_n)\\
&=H_n(q),\\
\varepsilon'\circ \Theta_q\circ H(a_1,\ldots,a_n)&=\varepsilon'\circ \Theta_q (H_n(X))\\
&=H_n(q),
\end{align*}
so $\varepsilon'\circ H\circ \fdeux(\theta_q)=\varepsilon'\circ \Theta_q\circ H$.  
By Proposition \ref{prop64}, $H\circ \fdeux(\theta_q)=\Theta_q\circ H$. \end{proof}

\begin{remark}
If $q=-1$, $H_k(-1)=(-1)^k$ for any $k$, so:
\begin{align*}
\fun(\theta_{-1})&:\left\{\begin{array}{rcl}
\WQSym&\longrightarrow&\WQSym\\
u&\longrightarrow&\displaystyle (-1)^{\max(u)} \sum_{\substack{f:\underline{\max(u)}\longrightarrow \underline{k},\\
\mbox{\scriptsize surjective, non decreasing}}} f\circ u,
\end{array}\right.\\
\fdeux(\theta_{-1})&:\left\{\begin{array}{rcl}
\QSym&\longrightarrow&\QSym\\
(a_1,\ldots,a_n)&\longrightarrow&\displaystyle (-1)^n\sum_{1\leq i_1<\ldots<i_k<n}
(a_1+\ldots+a_{i_1},\ldots,a_{i_1+\ldots+i_k+1}+\ldots+a_n).
\end{array}\right.\\
\end{align*}
\end{remark}

\begin{example} If $a,b,c\geq 1$, we put:
\begin{align*}
(1^a)&=(\underbrace{1,\ldots,1}_a),\\
(1^a2^b)&=(\underbrace{1,\ldots,1}_a,\underbrace{2,\ldots,2}_b),\\
(1^a2^b3^c)&=(\underbrace{1,\ldots,1}_a,\underbrace{2,\ldots,2}_b,\underbrace{3,\ldots,3}_c).
\end{align*}
Then:
\begin{align*}
\fun(\theta_q)((1^a))&=q(1^a),\\
\fun(\theta_q)((1^a2^b))&=q^2(1^a2^b)+\frac{q(q-1)}{2}(1^{a+b}),\\
\fun(\theta_q)((1^a3^b2^c))&=q^3(1^a3^b2^c)+\frac{q^2(q-1)}{2}(1^a2^b1^c)+\frac{q^2(q-1)}{2}(1^a2^{b+c})
+\frac{q(q-1)(q-2)}{6}(1^{a+b+c}),\\
\fdeux(\theta_q)((a))&=q(a),\\
\fdeux(\theta_q)((a,b))&=q^2(a,b)++\frac{q(q-1)}{2}(a+b),\\
\fdeux(\theta_q)((a,b,c))&=q^3(a,b,c)+\frac{q^2(q-1)}{2}(a+b,c)+\frac{q^2(q-1)}{2}(a,b+c)+\frac{q(q-1)(q-2)}{6}(a+b+c).
\end{align*}
\end{example}

\subsection{Fock functors and characters}

The monoids of characters of a twisted bialgebra and of its images by the Fock functor are related:

\begin{prop}
Let $\calP$ be a twisted bialgebra. 
\begin{enumerate}
\item We denote by $\chara(\fun(\calP))$ the monoid of characters of the bialgebra $\fun(\calP)$. 
The following map is an injective map of monoids:
\[\rho:\left\{\begin{array}{rcl}
\chara(\calP)&\longrightarrow&\chara(\fun(\calP))\\
f&\longrightarrow&\rho(f)=\displaystyle \bigoplus_{n\geq 0} f[\underline{n}].
\end{array}\right.\]
\item $\rho$ induces a monoid isomorphism between $\chara(\calP)$ and 
 the monoid of characters $\chara(\fdeux(\calP))$ of the bialgebra $\fdeux(\calP)$.
\end{enumerate}\end{prop}

\begin{proof}
1. \textit{First step.} Let $\sigma:A\longrightarrow B$ be a bijection between two finite sets.
As $f$ is a species morphism, the following diagram is commutative:
\[\xymatrix{\calP[A]\ar[r]^{\calP[\sigma]} \ar[d]_{f[A]}&\calP[B]\ar[d]^{f[B]}\\
\com[A]\ar[r]_{\com[\sigma]}&\com[B]
}\]
As $\com[\sigma]=\id_\K$, $f[B]\circ \calP[\sigma]=f[A]$.\\

\textit{Second step}. Let us prove that $\rho$ is well-defined. 
If $f\in \chara(\calP)$, then if $x\in \calP[\underline{m}]$
and $y\in \calP[\underline{n}]$, by the first step:
\begin{align*}
\rho(f)(xy)&=f[\underline{m+n}]\circ \calP[\sigma_{m,n}]\circ m(x\otimes y)\\
&=f[\underline{m}\sqcup \underline{n}]\circ m(x\otimes y)\\
&=f[\underline{m}](x)f[\underline{n}](y)\\
&=\rho(f)(x)\rho(f)(y).
\end{align*}
So $\rho(f)$ is indeed a character of $\fun(\calP)$.\\

Let us now prove that $\rho$ is a monoid morphism. Let $f,g\in \chara(\calP)$.
For any $x \in \calP[\underline{n}]$, by the first step:
\begin{align*}
\rho(f)*\rho(g)(x)&=\sum_{I\subseteq \underline{n}}
(f[\underline{\sharp I}]\otimes g[\underline{n-\sharp I}])\circ
(\calP[\sigma_I]\otimes \calP[\sigma_{\underline{n}\setminus I}])\circ \Delta_{I,\underline{n}\setminus I}(x)\\
&=\sum_{I\subseteq \underline{n}} (f[I]\otimes g[\underline{n}\setminus I])
\circ \Delta_{I,\underline{n}\setminus I}(x)\\
&=(f*g)[n](x)\\
&=\rho(f*g)(x).
\end{align*}
So $\rho(f)*\rho(g)=\rho(f*g)$. \\

Let us now prove that $\rho$ is injective. If $\rho(f)=\rho(g)$, then for any $n\geq 0$,
$f[\underline{n}]=g[\underline{n}]$. For any finite set $A$, let us denote by $n$ the cardinality of $A$
and let us choose $\sigma:A\longrightarrow \underline{n}$ be a bijection. By the first step:
\[f[A]=f[\underline{n}]\circ \calP[\sigma]=g[\underline{n}]\circ \calP[\sigma]=g[A].\]
Hence, $f=g$. \\

2. Let $f\in \chara(\calP)$. For any $x\in \calP[\underline[n]]$, $\sigma \in \mathfrak{S}_n$:
\begin{align*}
\rho(f)(x-\calP[\sigma](x))&=f[\underline{n}](x)-f[\underline{n}](\calP[\sigma](x))\\
&=f[\underline{n}](x)-\calP[\sigma](f[\underline{n}](x))\\
&=f[\underline{n}](x)-f[\underline{n}](x)\\
&=0.
\end{align*}
Hence, $\rho(f)$ induces a character $\overline{\rho}(f)$ of $\fdeux(\calP)$ defined by:
\begin{align*}
&\forall x\in \calP[\underline{n}],& \overline{\rho}(f)(\overline{x})&=\rho(f)(x)=f[\underline{n}](x).
\end{align*}
This defines a monoid morphism $\overline{\rho}$ from $\chara(\calP)$ to $\chara(\fdeux(\calP))$.
It is injective: indeed, if $\overline{\rho}(f)=\overline{\rho}(g)$, then $\rho(f)=\rho(g)$, 
so $f=g$. Let us prove it is surjective. Let us consider $g\in \chara(\fdeux(\calP))$.
Let $A$ be a finite set. We define $f[A]:\calP[A]\longrightarrow \K$ by:
\begin{align*}
&\forall x\in \calP[A],& f[A](x)=g(\overline{\calP[\sigma](x)}),
\end{align*}
where $\sigma:A\longrightarrow \underline{n}$ is any bijection. This does not depend on the choice of $\sigma$:
if $\tau$ is another bijection, then, as $\sigma \circ \tau^{-1}\in \mathfrak{S}_n$:
\begin{align*}
g(\overline{\calP[\tau](x)})&=g(\overline{\calP[\sigma \circ \tau^{-1}]\circ \calP[\tau](x)})
=g(\overline{\calP[\sigma](x)}).
\end{align*}
Hence, $f[A]$ is well-defined. If $\tau:A\longrightarrow B$ is a bijection, for any $x \in \calP[A]$,
choosing a bijection $\sigma:A\longrightarrow \underline{n}$:
\begin{align*}
f[B](\calP[\tau](x))&=g(\calP[\sigma \circ \tau^{-1}]\circ \overline{\calP[\tau](x)})\\
&=g(\overline{\calP[\sigma](x)})\\
&=\com[\tau]\circ f[A](x).
\end{align*}
So $f$ is a species morphism. Let $x\in \calP[A]$, $y\in \calP[B]$. Choosing bijections 
$\sigma:A\longrightarrow \underline{m}$ and $\tau:A\longrightarrow \underline{n}$:
\begin{align*}
f(xy)&=g(\overline{\calP[\sigma \sqcup \tau](xy)})\\
&=g(\overline{\calP[\sigma](x)}\:\overline{\calP[\tau](y)})\\
&=g(\overline{\calP[\sigma](x)})g(\overline{\calP[\tau](y)})\\
&=f(x)f(y).
\end{align*}
So $f\in \chara(\calP)$. If $x\in \calP[\underline{n}]$:
\begin{align*}
\overline{\rho}(f)(\overline{x})&=f[\underline{n}](x)=g(\overline{\calP[\id_{\underline{n}}](x)})=g(\overline{x}).
\end{align*}
Therefore, $\overline{\rho}(f)=g$. \end{proof}

\begin{remark} It may happen that $\rho$ is not surjective. Let us consider for example the twisted
algebra of posets $\calpos$. We define a character $g$ on $\fun(\calpos)$ by the following:
if $P=(\underline{n},\leq_P)$ is a poset, 
\[g(P)=\begin{cases}
1\mbox{ if $\forall i,j\in \underline{n}$, $i\leq_P j\Longrightarrow i\leq j$},\\
0\mbox{ otherwise}.
\end{cases}\] 
Let us assume that there exists a character $f\in \chara(\calpos)$ such that $\rho(f)=g$.
Then $f[\underline{2}](\tddeux{$1$}{$2$})=1$ and $f[\underline{2}](\tddeux{$2$}{$1$})=0$.
Moreover, by the first step of the preceding proof, for $\sigma=(12)$:
\[f[\underline{2}](\tddeux{$2$}{$1$})=f[\underline{2}]\circ \calpos[\sigma](\tddeux{$1$}{$2$})
=f[\underline{2}](\tddeux{$1$}{$2$}).\]
This is a contradiction. So $\rho$ is not surjective.
\end{remark}

\subsection{Hopf algebraic structure from the first Fock functor}

\begin{theo}\label{theo68}
\begin{enumerate}
\item Let $\calP$ be a connected twisted bialgebra. We define another coproduct $\blacktriangle$ 
on $\fun(\calP)$ by the following:
\begin{align*}
&\forall x\in \calP[\underline{n}],&\blacktriangle(x)&=\sum_{k=0}^n (\calP[\sigma_{\underline{k}}]
\otimes \calP[\sigma_{\underline{n}\setminus \underline{k}}])\circ \Delta_{\underline{k},
\underline{n}\setminus \underline{k}}(x).
\end{align*}
Then $(\fun(\calP),m,\blacktriangle)$ is an infinitesimal bialgebra, in the sense of Loday and Ronco \cite{LodayRonco}.
\item If $\calP$ is a connected twisted bialgebra, then $\fun(\calP)$ is freely generated
as an algebra by the space of primitive elements of $\blacktriangle$.
\item If $\calP$ is a cofree twisted coalgebra, then $\fun(\calP)$ is a cofree coalgebra.
\end{enumerate}
\end{theo}

\begin{proof}
Let $x\in \calP[\underline{n}]$. Then:
\begin{align*}
(\blacktriangle\otimes \id)\circ \blacktriangle(x)
&=(\id \otimes \blacktriangle)\circ \blacktriangle(x)
=\sum_{0\leq k\leq l\leq n} (\calP[\sigma_{\underline{k}}]\otimes \calP[\sigma_{\underline{l}\setminus \underline{k}}]
\otimes \calP[\sigma_{\underline{n}\setminus \underline{l}}])\circ 
\Delta_{\underline{k},\underline{l}\setminus \underline{k},\underline{n}\setminus \underline{l}}(x).
\end{align*}
So $\blacktriangle$ is coassociative. Let $x\in \calP[\underline{m}]$ and $y\in\calP[\underline{n}]$.
\begin{align*}
\blacktriangle(xy)&=\sum_{0\leq k\leq m+n}
(\calP[\sigma_{\underline{k}}]\otimes \calP[\sigma_{\underline{m+n}\setminus \underline{k}}])
\circ \Delta_{\underline{k},\underline{m+n}\setminus \underline{k}}\circ \calP[\sigma_{m,n}]\circ m(x\otimes y)\\
&=\sum_{0\leq k\leq m}
(\calP[\sigma_{\underline{k}}]\otimes \calP[\sigma_{\underline{m+n}\setminus \underline{k}}])
\circ \Delta_{\underline{k},\underline{m+n}\setminus \underline{k}}\circ \calP[\sigma_{m,n}]\circ m(x\otimes y)\\
&+\sum_{m\leq k\leq m+n}
(\calP[\sigma_{\underline{k}}]\otimes \calP[\sigma_{\underline{m+n}\setminus \underline{k}}])
\circ \Delta_{\underline{k},\underline{m+n}\setminus \underline{k}}\circ \calP[\sigma_{m,n}]\circ m(x\otimes y)\\
&-(\calP[\sigma_{\underline{m}}]\otimes \calP[\sigma_{\underline{m+n}\setminus \underline{m}}])
\circ \Delta_{\underline{m},\underline{m+n}\setminus \underline{m}}\circ \calP[\sigma_{m,n}]\circ m(x\otimes y)\\
&=\sum_{0\leq k\leq m}
(\calP[\sigma_{\underline{k}}]\otimes \calP[\sigma_{\underline{m+n}\setminus \underline{k}}])\circ 
(\Delta_{\underline{k},\underline{m}\setminus \underline{k}}(x)
\Delta_{\emptyset,\underline{n}}(y))\\
&+\sum_{m\leq k\leq m+n}
(\calP[\sigma_{\underline{k}}]\otimes \calP[\sigma_{\underline{m+n}\setminus \underline{k}}])
(\Delta_{\underline{m},\emptyset}(x)
\Delta_{\underline{k-m},\underline{m+n-k}\setminus \underline{k-m}}(y))\\
&-(\calP[\sigma_{\underline{m}}]\otimes \calP[\sigma_{\underline{m+n}\setminus \underline{m}}])
(\Delta_{\underline{m},\emptyset}(x)
\Delta_{\emptyset,\underline{n}}(y))\\
&=\sum_{0\leq k\leq m}
(\calP[\sigma_{\underline{k}}]\otimes \calP[\sigma_{\underline{m+n}\setminus \underline{k}}])(\Delta_{\underline{k},\underline{m}\setminus \underline{k}}(x)
1\otimes y)\\
&+\sum_{m\leq k\leq m+n}
(\calP[\sigma_{\underline{k}}]\otimes \calP[\sigma_{\underline{m+n}\setminus \underline{k}}])(x\otimes 1
\Delta_{\underline{k-m},\underline{m+n-k}\setminus \underline{k-m}}(y))\\
&-(\calP[\sigma_{\underline{m}}]\otimes \calP[\sigma_{\underline{m+n}\setminus \underline{m}}])
(x\otimes y)\\
&=\blacktriangle(x)(1\otimes y)+(x\otimes 1)\blacktriangle(y)-x\otimes y.
\end{align*}
Hence, $\fun(\calP)$ is an infinitesimal bialgebra.\\

2. This is a direct consequence of the first point, by Loday and Ronco's rigidity theorem \cite{LodayRonco}.\\

3. We assume that $\calP=\cot(\calQ)$. We define another product $*$ on $\fun(\calP)$ by:
\begin{align*}
&\forall x=q_1\ldots q_k, y=q_{k+1}\ldots q_{k+l} \in \calP,&
x*y&=p_1\ldots p_k q_{k+1}^+\ldots q_{k+l}^+,
\end{align*}
where we use the following notation: if $q_i \in \calQ[I]$, then $q_i^+=\calQ[(\sigma_{m,n})_{\mid I}](q_i)$.
This product is associative. Moreover, by definition of the coproduct of $\cot(\calQ)$:
\begin{align*}
\Delta(x*y)&=\sum_{i=0}^k \calP[\sigma'_i](p_1\ldots p_i)\otimes \calP[\sigma''_i](p_{i+1}\ldots p_k
q_1^+\ldots q_l^+)\\
&+\sum_{i=k}^l \calP[\sigma'_i](p_1\ldots p_kq_1^+\ldots q_i^+)\otimes \calP[\sigma''_i](q_{i+1}^+\ldots q_l^+)\\
&-q_1\ldots q_k\otimes q_{k+1}\ldots q_{k+l}\\
&=\Delta(x)*(1\otimes y)+(x\otimes 1)*\Delta(y)-x\otimes y,
\end{align*}
where $\sigma'_i$, $\sigma''_i$ are suitable bijections. Hence, $(\fun(\calP),*,\Delta)$
is an infinitesimal bialgebra. By Loday and Ronco's rigidity theorem, $(\fun(\calP),\Delta)$ is cofree.
\end{proof}

\begin{remark}
We obtain again the well-known result that $\FQSym$ and $\WQSym$ are both free and cofree bialgebras.
\end{remark}

\section{Applications}

\subsection{Hopf algebras and bialgebras of graphs}

We put $\bfh_{\calgr}=\fun(\calgr)$ and $H_{\calgr}=\fdeux(\calgr)$. 
The Hopf algebra $\bfh_{\calgr}$ has for basis the set of graphs whose vertices
are a set composition of $\underline{n}$, for a given $n$. Its product is given by shifted concatenation, for example:
\[\tdtroisun{$\{1,2\}$}{$\{4\}$}{\hspace{-2mm}$\{3\}$}\hspace{4mm}\cdot \tddeux{$\{1,3\}$}{$\{2\}$}\hspace{4mm}
=\:\tdtroisun{$\{1,2\}$}{$\{4\}$}{\hspace{-2mm}$\{3\}$}\hspace{4mm}\tddeux{$\{5,7\}$}{$\{6\}$}\hspace{4mm}.\]
Its coproduct is given by decompositions of the set of vertices into two parts with standardization, for example:
\begin{align*}
\Delta(\tdun{$\{1\}$}\hspace{2mm})&=\tdun{$\{1\}$}\hspace{2mm}\otimes 1+1\otimes \tdun{$\{1\}$}\hspace{2mm},\\
\Delta(\tddeux{$\{1\}$}{$\{2,3\}$}\hspace{4mm})&=\tddeux{$\{1\}$}{$\{2,3\}$}\hspace{4mm}\otimes 1
+\tdun{$\{1\}$}\hspace{2mm}\otimes \tdun{$\{1,2\}$}\hspace{4mm}+
\tdun{$\{1,2\}$}\hspace{4mm}\otimes \tdun{$\{1\}$}\hspace{2mm}+1\otimes \tddeux{$\{1\}$}{$\{2,3\}$}\hspace{4mm},\\
\Delta(\hspace{2mm}\tdtroisun{$\{1,3\}$}{$\{4\}$}{\hspace{-2mm}$\{2\}$}\hspace{4mm})
&=\hspace{2mm}\tdtroisun{$\{1,3\}$}{$\{4\}$}{\hspace{-2mm}$\{2\}$}\hspace{4mm}\otimes 1
+\tddeux{$\{1,3\}$}{$\{2\}$}\hspace{4mm}\otimes \tdun{$\{1\}$}\hspace{2mm}
+\tddeux{$\{1,2\}$}{$\{3\}$}\hspace{4mm}\otimes \tdun{$\{1\}$}\hspace{2mm}+
\tdun{$\{1\}$}\hspace{2mm}\tdun{$\{2\}$}\hspace{2mm}\otimes \tdun{$\{1,2\}$}\hspace{4mm}\\
&+\tdun{$\{1\}$}\hspace{2mm}\otimes \tddeux{$\{1,3\}$}{$\{2\}$}\hspace{4mm}
+\tdun{$\{1\}$}\hspace{2mm}\otimes \tddeux{$\{1,2\}$}{$\{3\}$}\hspace{4mm}
+\tdun{$\{1,2\}$}\hspace{4mm}\otimes \tdun{$\{1\}$}\hspace{2mm}\tdun{$\{2\}$}\hspace{2mm}+
1\otimes \hspace{2mm}\tdtroisun{$\{1,3\}$}{$\{4\}$}{\hspace{-2mm}$\{2\}$}\hspace{4mm}.
\end{align*}
The second coproduct is given by a contraction-extraction process; for example:
\begin{align*}
\delta(\tdun{$A$})&=\tdun{$A$}\otimes \tdun{$A$},\\
\delta(\tddeux{$A$}{$B$})&=\tddeux{$A$}{$B$}\otimes \tdun{$A$}\tdun{$B$}+
\tdun{$A\sqcup B$} \hspace{6mm} \otimes \tddeux{$A$}{$B$},\\
\delta(\tdtroisun{$A$}{$C$}{$B$})&=\tdtroisun{$A$}{$C$}{$B$}\otimes \tdun{$A$}\tdun{$B$}\tdun{$C$}
+\tddeux{$A\sqcup B$}{$C$}\hspace{6mm}\otimes \tddeux{$A$}{$B$}\tdun{$C$}+
\tddeux{$A\sqcup C$}{$B$}\hspace{6mm}\otimes \tddeux{$A$}{$C$}\tdun{$B$}+
\tdun{$A\sqcup B\sqcup C$}\hspace{12mm}\otimes \tdtroisun{$A$}{$C$}{$B$},
\end{align*}
where in the first computation, $(A)=\underline{n}$; in the second one, $(A,B)\in \comp[\underline{n}]$ and, in the last one,
$(A,B,C)\in \comp[\underline{n}]$.
This is a variant of the noncommutative Hopf algebra of graphs of \cite{Foissychrom}.\\

The Hopf algebra $H_{\calgr}$ has for basis the set of decorated graphs, that is to say graphs with a map
from their set of vertices to the set of positive integers. The product is given by concatenation, for example,
if $a$, $b$, $c$, $d$ and $e\geq 1$:
\[\tdtroisun{$a$}{$b$}{$c$}\cdot \tddeux{$d$}{$e$}=\tdtroisun{$a$}{$c$}{$b$}\tddeux{$d$}{$e$}.\]
The first product is given by partitions of the vertices into two subsets and the second one 
by an extraction-contraction process. For example, if $a$, $b$, $c\geq 1$:
\begin{align*}
\Delta(\tdun{$a$})&=\tdun{$a$}\otimes 1+1\otimes \tdun{$a$},\\
\Delta(\tddeux{$a$}{$b$})&=\tddeux{$a$}{$b$}\otimes 1+\tdun{$a$}\otimes \tdun{$b$}+
\tdun{$b$}\otimes \tdun{$a$}+1\otimes \tddeux{$a$}{$b$},\\
\Delta(\tdtroisun{$a$}{$c$}{$b$})&=\tdtroisun{$a$}{$c$}{$b$}\otimes 1
+\tddeux{$a$}{$b$}\otimes \tdun{$c$}+\tddeux{$a$}{$c$}\otimes \tdun{$b$}+
\tdun{$b$}\tdun{$c$}\otimes \tdun{$a$}\\
&+\tdun{$c$}\otimes \tddeux{$a$}{$b$}+\tdun{$b$}\otimes \tddeux{$a$}{$c$}+\tdun{$a$}\otimes \tdun{$b$}\tdun{$c$}+
1\otimes \tdtroisun{$a$}{$c$}{$b$};\\ \\
\delta(\tdun{$a$})&=\tdun{$a$}\otimes \tdun{$a$},\\
\delta(\tddeux{$a$}{$b$})&=\tddeux{$a$}{$b$}\otimes \tdun{$a$}\tdun{$b$}+
\tdun{$a+b$} \hspace{4mm} \otimes \tddeux{$a$}{$b$},\\
\delta(\tdtroisun{$a$}{$c$}{$b$})&=\tdtroisun{$a$}{$c$}{$b$}\otimes \tdun{$a$}\tdun{$b$}\tdun{$c$}
+\tddeux{$a+b$}{$c$}\hspace{4mm}\otimes \tddeux{$a$}{$b$}\tdun{$c$}+
\tddeux{$a+c$}{$b$}\hspace{4mm}\otimes \tddeux{$a$}{$c$}\tdun{$b$}+
\tdun{$a+b+c$}\hspace{9mm}\otimes \tdtroisun{$a$}{$c$}{$b$}.
\end{align*}
This is a decorated variant of the commutative Hopf algebra of graphs of \cite{Foissychrom,Hivert}.

\subsection{Morphisms from graphs to $\WQSym$, $\QSym$ and $\K[X]$}

The Fock functors applied to the the morphism $\phi_{chr_q}$ gives a commutative diagram of morphisms
compatible with the product and both coproducts:\\
\[\xymatrix{\bfh_{\calgr}\ar@{->>}[r]\ar[d]_{\fun(\phi_{chr_q})}&H_{\calgr}\ar[d]_{\fdeux(\phi_{chr_q})}\ar[dr]^{P_{chr_q}}&\\
\WQSym\ar@{->>}[r]&\QSym\ar@{->>}[r]_H&\K[X]}\]
where $P_{chr_q}=H\circ \fdeux(\phi_{chr_q})$. For any graph in $\bfh_{\calgr}$ of degree $n$:
\begin{align*}
\fun(\phi_{chr_1})(G)&=\sum_{c\in VC(G)} c(1)\ldots c(n),\\
\fdeux(\phi_{chr_1})(G)&=\sum_{c\in VC(G)} (\sharp c^{-1}(1),\ldots,\sharp c^{-1}(\max(c))).
\end{align*}
For example:
\begin{align*}
\fun(\phi_{chr_1})(\tdun{$\{1\}$}\hspace{2mm})&=(1),\\
\fun(\phi_{chr_1})(\tddeux{$\{1,3\}$}{$\{2\}$}\hspace{4mm})&=(121)+(212),\\
\fun(\phi_{chr_1})(\hspace{2mm}\tdtroisun{$\{1,4\}$}{$\{2\}$}{\hspace{-2mm}$\{3\}$}\hspace{3mm})
&=(1231)+(1321)+(2132)+(2312)+(3123)+(3213)+(1221)+(2112).
\end{align*}
This is a colored version of the noncommutative chromatic symmetric function of \cite{Gebhard}.
The morphism $\fdeux(\phi_{chr_1})$ is a colored version  of Stanley's chromatic symmetric series \cite{Stanley2,Rosas}, 
as it is explained in \cite{Foissychrom}. For example, if $a,b,c\geq 1$:
\begin{align*}
\fdeux(\phi_{chr_1})(\tdun{$a$})&=(a),\\
\fdeux(\phi_{chr_1})(\tddeux{$a$}{$b$})&=(a,b)+(b,a),\\
\fdeux(\phi_{chr_1})(\tdtroisun{$a$}{$c$}{$b$})&=(a,b,c)+(a,c,b)+(b,a,c)+(b,c,a)+(c,a,b)+(c,b,a)+(a,b+c)+(b+c,a).
\end{align*}
Reformulating in terms of formal series in a infinite set $\{X_1,X_2,\ldots\}$ of indeterminates,
with the help of the polynomial realization of $\QSym$:
\begin{align*}
\fdeux(\phi_{chr_1})(\tdun{$a$})&=\sum_{i} X_i^a,\\
\fdeux(\phi_{chr_1})(\tddeux{$a$}{$b$})&=\sum_{i\neq j} X_i^aX_j^b,\\
\fdeux(\phi_{chr_1})(\tdtroisun{$a$}{$c$}{$b$})
&=\sum_{i\neq j,i\neq k,j\neq k} X_i^a X_j^b X_k^c+\sum_{i\neq j} X_i^a X_j^{b+c}.
\end{align*}

Composing $\fdeux(\phi_{chr_1})$ with $H$, we obtain the chromatic polynomial $P_{chr_1}$ \cite{BL}. For example:
\begin{align*}
P_{chr_1}(\tdun{$a$})&=H_1(X)=X,\\
P_{chr_1}(\tddeux{$a$}{$b$})&=2H_2(X)=X(X-1)\\
P_{chr_1}(\tdtroisun{$a$}{$c$}{$b$})&=6H_3(X)+2H_2(X)=X(X-1)^2.
\end{align*}

\begin{prop} If $q\neq 0$, for any graph $G \in H_{\calgr}$:
\begin{align*}
P_{chr_q}(G)(X)&=q^{\deg(G)} P_{chr_1}(G)\left(\frac{X}{q}\right).
\end{align*} \end{prop}

\begin{proof}
Indeed:
\begin{align*}
P_{chr_q}(G)(X)&=H\circ \fdeux(\theta_{q^{-1}})\circ \fdeux(\phi_{chr_1})\circ \fdeux(\iota_q)(G)(X)\\
&=q^{\deg(G)}H\circ \fdeux(\theta_{q^{-1}})\circ \fdeux(\phi_{chr_1})(G)(X)\\
&=q^{\deg(G)} \Theta_{q^{-1}}\circ H \circ \fdeux(\phi_{chr_1})(G)(X)\\
&=q^{\deg(G)} \Theta_{q^{-1}} \circ P_{chr_1}(G)(X)\\
&=q^{\deg(G)} P_{chr_1}(G)\left(\frac{X}{q}\right).
\end{align*} \end{proof}

\begin{prop}
Let $q\in \K$. 
\begin{enumerate}
\item For any graph $G\in \calgr[\underline{n}]$:
\begin{align*}
\fun(\phi_q)(G)&=q^{\deg(G)} \sum_{w\in W(G)}w,
\end{align*}
where $w$ is the set of packed words of length $n$ such that for any $i\in \underline{n}$,
if $i$ and $j$ are in the same vertex of $G$, then $w(i)=w(k)$. 
\item For any decorated graph $G$, denoting by $d_1,\ldots,d_k$ the decoration of its vertices:
\[\fdeux(\phi_q)(G)=q^{\deg(G))} (d_1)\squplus \ldots \squplus (d_k).\]
\item For any decorated graph $G$:
\[H\circ \fdeux(\phi_q)(G)=q^{\deg(G)}X^{\deg(G)}.\]
\end{enumerate}\end{prop}

\begin{proof}
The two first points are immediate. If $G$ is a decorated graph:
\begin{align*}
H\circ \fdeux(\phi_q)(G)&=q^{\deg(G))}H( (d_1)\squplus \ldots \squplus (d_k))\\
&=q^{\deg(G)}H((a_1))\ldots H((a_k))\\
&=q^{\deg(G)}X^k\\
&=q^{\deg(G)}X^{\deg(G)}.
\end{align*} \end{proof}

As $\phi_{chr_1}=\phi_1\leftarrow \lambda_{chr_1}$:

\begin{cor}
For any graph $G\in H_{\calgr}$:
\begin{align*}
\fdeux(\phi_{chr_1})(G)
&=\sum_{\sim\triangleleft G} \lambda_{chr_1}(G\mid \sim) \prod^\squplus
_{\mbox{\scriptsize $C$ equivalence classe of $\sim$}} (C),\\
P_{chr_1}(G)&=\sum_{\sim\triangleleft G} \lambda_{chr_1}(G\mid \sim) X^{\cl(\sim)}.
\end{align*}
\end{cor}

Applying the Fock functor to $\Gamma$:

\begin{prop}
The following maps are Hopf algebra automorphisms:
\begin{align*}
\fun(\Gamma)&:\left\{\begin{array}{rcl}
\bfh_{\calgr}&\longrightarrow&\bfh_{\calgr}\\
G&\longrightarrow&\displaystyle \sum_{\sim\triangleleft G}\sharp\{\mbox{acyclic orientations of $G\mid \sim$}\} G/\sim, 
\end{array}\right.\\
\fdeux(\Gamma)&:\left\{\begin{array}{rcl}
H_{\calgr}&\longrightarrow&H_{\calgr}\\
G&\longrightarrow&\displaystyle \sum_{\sim\triangleleft G} \sharp\{\mbox{acyclic orientations of $G\mid \sim$}\}G/\sim.
\end{array}\right.
\end{align*}
Their inverse are given by:
\begin{align*}
\fun(\Gamma)^{-1}&:\left\{\begin{array}{rcl}
\bfh_{\calgr}&\longrightarrow&\bfh_{\calgr}\\
G&\longrightarrow&\displaystyle \sum_{\sim\triangleleft G} (-1)^{\deg(G)+\cl(\sim)}
\sharp\{\mbox{acyclic orientations of $G\mid \sim$}\}G/\sim,
\end{array}\right.\\
\fdeux(\Gamma)^{-1}&:\left\{\begin{array}{rcl}
H_{\calgr}&\longrightarrow&H_{\calgr}\\
G&\longrightarrow&\displaystyle
\sum_{\sim\triangleleft G}(-1)^{\deg(G)+\cl(\sim)} \sharp\{\mbox{acyclic orientations of $G\mid \sim$}\}G/\sim.
\end{array}\right.
\end{align*}
Moreover, the following diagram commutes:
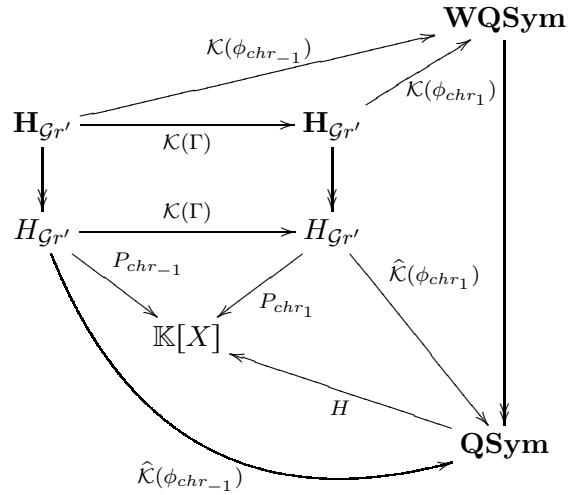
\begin{align*}
&\xymatrix{&&&\WQSym\ar@{->>}[dddd]\\
\bfh_{\calgr}\ar[rr]_{\fun(\Gamma)}\ar@{->>}[d]\ar[rrru]^{\fun(\phi_{chr_{-1}})}
&&\bfh_{\calgr}\ar@{->>}[d]\ar[ru]_{\fun(\phi_{chr_1})}&\\
H_{\calgr}\ar[rr]^{\fun(\Gamma)}\ar[dr]^{P_{chr_{-1}}}\ar@/_4pc/[ddrrr]_{\fdeux(\phi_{chr_{-1}})}
&&H_{\calgr}\ar[ld]^{P_{chr_1}}\ar[rdd]^<(.2){\fdeux(\phi_{chr_1})}&\\
&\K[X]&&\\
&&&\QSym\ar[ull]^H}
\end{align*}
\end{prop}

\subsection{Hopf algebras and bialgebras on finite topologies}

We put $\bfh_{\caltop}=\fun(\caltop)$ and $H_{\caltop}=\fdeux(\caltop)$. 
The Hopf algebra $\bfh_{\caltop}$ has for basis the set of finite topologies on a set $\underline{n}$. 
Its product is given by shifted concatenation. For example:
\[\tdtroisun{$\{1,2\}$}{$\{4\}$}{\hspace{-2mm}$\{3\}$}\hspace{4mm}\cdot \tddeux{$\{1,3\}$}{$\{2\}$}\hspace{4mm}
=\:\tdtroisun{$\{1,2\}$}{$\{4\}$}{\hspace{-2mm}$\{3\}$}\hspace{4mm}\tddeux{$\{5,7\}$}{$\{6\}$}\hspace{4mm}.\]
Its coproduct is given by partitions of the underlying sets in an open and a closed set, for example:
\begin{align*}
\Delta(\tdun{$\{1\}$}\hspace{2mm})&=\tdun{$\{1\}$}\hspace{2mm}\otimes 1+1\otimes \tdun{$\{1\}$}\hspace{2mm},\\
\Delta(\tddeux{$\{1\}$}{$\{2,3\}$}\hspace{4mm})&=\tddeux{$\{1\}$}{$\{2,3\}$}\hspace{4mm}\otimes 1
+\tdun{$\{1\}$}\hspace{2mm}\otimes \tdun{$\{1,2\}$}\hspace{4mm}+1\otimes \tddeux{$\{1\}$}{$\{2,3\}$}\hspace{4mm},\\
\Delta(\hspace{2mm}\tdtroisun{$\{1,3\}$}{$\{4\}$}{\hspace{-2mm}$\{2\}$}\hspace{4mm})
&=\hspace{2mm}\tdtroisun{$\{1,3\}$}{$\{4\}$}{\hspace{-2mm}$\{2\}$}\hspace{4mm}\otimes 1
+\tddeux{$\{1,3\}$}{$\{2\}$}\hspace{4mm}\otimes \tdun{$\{1\}$}\hspace{2mm}
+\tddeux{$\{1,2\}$}{$\{3\}$}\hspace{4mm}\otimes \tdun{$\{1\}$}\hspace{2mm}
+\tdun{$\{1,2\}$}\hspace{4mm}\otimes \tdun{$\{1\}$}\hspace{2mm}\tdun{$\{2\}$}\hspace{2mm}+
1\otimes \hspace{2mm}\tdtroisun{$\{1,3\}$}{$\{4\}$}{\hspace{-2mm}$\{2\}$}\hspace{4mm},\\
\Delta(\tdtroisdeux{$\{1,3\}$}{$\{4\}$}{$\{2\}$}\hspace{4mm})&=\tdtroisdeux{$\{1,3\}$}{$\{4\}$}{$\{2\}$}\hspace{4mm}
\otimes 1+\tddeux{$\{1,2\}$}{$\{3\}$}\hspace{4mm}\otimes \tdun{$\{1\}$}\hspace{2mm}+
\tdun{$\{1,2\}$}\hspace{4mm}\otimes \tddeux{$\{2\}$}{$\{1\}$}\hspace{2mm}+
1\otimes \tdtroisdeux{$\{1,3\}$}{$\{4\}$}{$\{2\}$}\hspace{4mm}.
\end{align*}
The second coproduct is given by a contraction-extraction process; for example:
\begin{align*}
\delta(\tdun{$A$})&=\tdun{$A$}\otimes \tdun{$A$},\\
\delta(\tddeux{$A$}{$B$})&=\tddeux{$A$}{$B$}\otimes \tdun{$A$}\tdun{$B$}+
\tdun{$A\sqcup B$} \hspace{6mm} \otimes \tddeux{$A$}{$B$},\\
\delta(\tdtroisun{$A$}{$C$}{$B$})&=\tdtroisun{$A$}{$C$}{$B$}\otimes \tdun{$A$}\tdun{$B$}\tdun{$C$}
+\tddeux{$A\sqcup B$}{$C$}\hspace{6mm}\otimes \tddeux{$A$}{$B$}\tdun{$C$}+
\tddeux{$A\sqcup C$}{$B$}\hspace{6mm}\otimes \tddeux{$A$}{$C$}\tdun{$B$}+
\tdun{$A\sqcup B\sqcup C$}\hspace{12mm}\otimes \tdtroisun{$A$}{$C$}{$B$},\\
\delta(\tdtroisdeux{$A$}{$B$}{$C$})&=\tdtroisdeux{$A$}{$B$}{$C$}\otimes \tdun{$A$}\tdun{$B$}\tdun{$C$}
+\tddeux{$A\sqcup B$}{$C$}\hspace{6mm}\otimes \tddeux{$A$}{$B$}\tdun{$C$}+
\tddeux{$A$}{$B\sqcup C$}\hspace{6mm}\otimes \tdun{$A$}\tddeux{$B$}{$C$}+
\tdun{$A\sqcup B\sqcup C$}\hspace{12mm}\otimes \tdtroisdeux{$A$}{$B$}{$C$}.
\end{align*}
where in the first computation, $A=\underline{n}$; in the second one, $(A,B)\in \comp[\underline{n}]$ and, in the last one,
$(A,B,C)\in \comp[\underline{n}]$.
This is the Hopf algebra of finite topologies of \cite{FFM,FM,FMP,FoissyEhrhart}.\\

The Hopf algebra $H_{\caltop}$ has for basis the set of finite topologies, up to homeomorphism.
We represent such objects by Hasse graphs of the underlying quasi-order, where the cardinality of the
equivalence classes of the quasi-order are represented on the graph. 
The product is given by disjoint union, for example,
if $a$, $b$, $c$, $d$ and $e\geq 1$:
\[\tdtroisun{$a$}{$b$}{$c$}\cdot \tddeux{$d$}{$e$}=\tdtroisun{$a$}{$c$}{$b$}\tddeux{$d$}{$e$}.\]
The first product is given by partitions into a closed set and an open set, and the second one 
by an extraction-contraction process. For example, if $a$, $b$, $c\geq 1$:
\begin{align*}
\Delta(\tdun{$a$})&=\tdun{$a$}\otimes 1+1\otimes \tdun{$a$},\\
\Delta(\tddeux{$a$}{$b$})&=\tddeux{$a$}{$b$}\otimes 1+\tdun{$a$}\otimes \tdun{$b$}+1\otimes \tddeux{$a$}{$b$},\\
\Delta(\tdtroisun{$a$}{$c$}{$b$})&=\tdtroisun{$a$}{$c$}{$b$}\otimes 1
+\tddeux{$a$}{$b$}\otimes \tdun{$c$}+\tddeux{$a$}{$c$}\otimes \tdun{$b$}+
+\tdun{$c$}\otimes \tddeux{$a$}{$b$}+1\otimes \tdtroisun{$a$}{$c$}{$b$},\\
\Delta(\tdtroisdeux{$a$}{$b$}{$c$})&=\tdtroisdeux{$a$}{$b$}{$c$}\otimes 1+\tddeux{$a$}{$b$}\otimes \tdun{$c$}
+\tdun{$a$}\otimes \tddeux{$b$}{$c$}+1\otimes \tdtroisdeux{$a$}{$b$}{$c$};\\ \\
\delta(\tdun{$a$})&=\tdun{$a$}\otimes \tdun{$a$},\\
\delta(\tddeux{$a$}{$b$})&=\tddeux{$a$}{$b$}\otimes \tdun{$a$}\tdun{$b$}+
\tdun{$a+b$} \hspace{4mm} \otimes \tddeux{$a$}{$b$},\\
\delta(\tdtroisun{$a$}{$c$}{$b$})&=\tdtroisun{$a$}{$c$}{$b$}\otimes \tdun{$a$}\tdun{$b$}\tdun{$c$}
+\tddeux{$a+b$}{$c$}\hspace{4mm}\otimes \tddeux{$a$}{$b$}\tdun{$c$}+
\tddeux{$a+c$}{$b$}\hspace{4mm}\otimes \tddeux{$a$}{$c$}\tdun{$b$}+
\tdun{$a+b+c$}\hspace{9mm}\otimes \tdtroisun{$a$}{$c$}{$b$},\\
\delta(\tdtroisdeux{$a$}{$b$}{$c$})&=\tdtroisdeux{$a$}{$b$}{$c$}\otimes \tdun{$a$}\tdun{$b$}\tdun{$c$}
+\tddeux{$a+b$}{$c$}\hspace{4mm}\otimes \tddeux{$a$}{$b$}\tdun{$c$}
+\tddeux{$a$}{$b+c$}\hspace{4mm}\otimes \tdun{$a$}\tddeux{$b$}{$c$}+
\tdun{$a+b+c$}\hspace{9mm}\otimes \tdtroisdeux{$a$}{$b$}{$c$}.
\end{align*}

As $\caltop$ is a cofree twisted bialgebra, by Theorem \ref{theo68}:

\begin{prop}
$\bfh_{\caltop}$ is a free and cofree bialgebra. 
\end{prop}

\subsection{Morphisms from finite topologies to $\WQSym$, $\QSym$ or $\K[X]$}

The Fock functors applied to the the morphism $\phi_{ehr_q}$ gives a commutative diagram of morphisms
compatible with the product and both coproducts:\\
\[\xymatrix{\bfh_{\caltop}\ar@{->>}[r]\ar[d]_{\fun(\phi_{ehr_q})}&H_{\caltop}\ar[d]_{\fdeux(\phi_{ehr_q})}\ar[dr]^{P_{ehr_q}}&\\
\WQSym\ar@{->>}[r]&\QSym\ar@{->>}[r]_H&\K[X]}\]
where $P_{ehr_q}=H\circ \fdeux(\phi_q)$. 
For any finite topology $T$ in $\bfh_{\caltop}$ of degree $n$:
\begin{align*}
\fun(\phi_{ehr_1})(T)&=\sum_{c\in L(T)} c(1)\ldots c(n),&
\fun(\phi_{ehr_{-1}})(T)&=\sum_{c\in L'(T)} c(1)\ldots c(n).
\end{align*}
For example:
\begin{align*}
\fun(\phi_{ehr_1})(\tdun{$\{1\}$}\hspace{2mm})&=(1),\\
\fun(\phi_{ehr_1})(\tddeux{$\{1,3\}$}{$\{2\}$}\hspace{4mm})&=(121),\\
\fun(\phi_{ehr_1})(\hspace{2mm}\tdtroisun{$\{1,4\}$}{$\{2\}$}{\hspace{-2mm}$\{3\}$}\hspace{3mm})
&=(1231)+(1321)+(1221),\\
\fun(\phi_{ehr_1})(\tdtroisdeux{$\{1,3\}$}{$\{4\}$}{$\{2\}$}\hspace{4mm})&=(1312);\\
\fun(\phi_{ehr_{-1}})(\tdun{$\{1\}$}\hspace{2mm})&=(1),\\
\fun(\phi_{ehr_{-1}})(\tddeux{$\{1,3\}$}{$\{2\}$}\hspace{4mm})&=(121)+(111),\\
\fun(\phi_{ehr_{-1}})(\hspace{2mm}\tdtroisun{$\{1,4\}$}{$\{2\}$}{\hspace{-2mm}$\{3\}$}\hspace{3mm})
&=(1231)+(1321)+(1221)+(1211)+(1121)+(1111),\\
\fun(\phi_{ehr_{-1}})(\tdtroisdeux{$\{1,3\}$}{$\{4\}$}{$\{2\}$}\hspace{4mm})&=(1312)+(1211)+(1221)+(1111).
\end{align*}
The morphism $\fun(\phi_{ehr_1})$ is the noncommutative version of the strict quasisymmetric Ehrhart functions \cite{BR},
as it is  explained in \cite{FoissyEhrhart},
whereas  $\fdeux(\phi_{ehr_1})$ is the strict  quasisymmetric Ehrhart function. For example, if $a,b,c\geq 1$:
\begin{align*}
\fdeux(\phi_{ehr_1})(\tdun{$a$})&=(a),\\
\fdeux(\phi_{ehr_1})(\tddeux{$a$}{$b$})&=(a,b),\\
\fdeux(\phi_{ehr_1})(\tdtroisun{$a$}{$c$}{$b$})&=(a,b,c)+(a,c,b)+(a,b+c),\\
\fdeux(\phi_{ehr_1})(\tdtroisdeux{$a$}{$b$}{$c$})&=(a,b,c);\\ \\
\fdeux(\phi_{ehr_{-1}})(\tdun{$a$})&=(a),\\
\fdeux(\phi_{ehr_{-1}})(\tddeux{$a$}{$b$})&=(a,b)+(a+b),\\
\fdeux(\phi_{ehr_{-1}})(\tdtroisun{$a$}{$c$}{$b$})&=(a,b,c)+(a,c,b)+(a,b+c)+(a+b,c)+(a+c,b)+(a+b+c),\\
\fdeux(\phi_{ehr_{-1}})(\tdtroisdeux{$a$}{$b$}{$c$})&=(a,b,c)+(a+b,c)+(a,b+c)+(a+b+c).
\end{align*}
Reformulating in terms of formal series in a infinite set $\{X_1,X_2,\ldots\}$ of indeterminates:
\begin{align*}
\fdeux(\phi_{ehr_1})(\tdun{$a$})&=\sum_i X_i^a,\\
\fdeux(\phi_{ehr_1})(\tddeux{$a$}{$b$})&=\sum_{i<j} X_i^a X_j^b,\\
\fdeux(\phi_{ehr_1})(\tdtroisun{$a$}{$c$}{$b$})&=\sum_{i<j<k} X_i^aX_j^bX_k^c+\sum_{i<j<k} X_i^aX_j^cX_k^b+
\sum_{i<j} X_i^aX_j^{b+c},\\
\fdeux(\phi_{ehr_1})(\tdtroisdeux{$a$}{$b$}{$c$})&=\sum_{i<j<k} X_i^aX_j^bX_k^c.
\end{align*}
Composing $\fdeux(\phi_{ehr_1})$  with $H$, we obtain the strict Ehrhart polynomial $P_{ehr_1}$. For example:
\begin{align*}
P_{ehr_1}(\tdun{$a$})&=H_1(X)=X,\\
P_{ehr_1}(\tddeux{$a$}{$b$})&=H_2(X)=\frac{X(X-1)}{2},\\
P_{ehr_1}(\tdtroisun{$a$}{$c$}{$b$})&=2H_3(X)+H_2(X)=\frac{X(X-1)(2X-1)}{6},\\
P_{ehr_1}(\tdtroisdeux{$a$}{$b$}{$c$})&=H_3(X)=\frac{X(X-1)(X-2)}{6}.
\end{align*}
Composing $\fdeux(\phi_{ehr_{-1}})$ with $H$, we obtain the Ehrhart polynomial $P_{ehr_{-1}}$. For example:
\begin{align*}
P_{ehr_{-1}}(\tdun{$a$})&=H_1(X)=X,\\
P_{ehr_{-1}}(\tddeux{$a$}{$b$})&=H_2(X)+H_1(X)=\frac{X(X+1)}{2},\\
P_{ehr_{-1}}(\tdtroisun{$a$}{$c$}{$b$})&=2H_3(X)+3H_2(X)+H_1(X)=\frac{X(X+1)(2X+1)}{6},\\
P_{ehr_{-1}}(\tdtroisdeux{$a$}{$b$}{$c$})&=H_3(X)+2H_2(X)+H_1(X)=\frac{X(X+1)(X+2)}{6}.
\end{align*}

\begin{prop}[Duality principle]
For any finite topology $T \in H_{\caltop}$:
\begin{align*}
P_{ehr_q}(T)(X)&=q^{\cl(T)} P_{ehr_1}(T)\left(\frac{X}{q}\right).
\end{align*} \end{prop}

\begin{proof}
Indeed:
\begin{align*}
P_{ehr_q}(T)(X)&=H\circ \fdeux(\theta_{q^{-1}})\circ \fdeux(\phi_{ehr_1})\circ \fdeux(\iota_q)(T)(X)\\
&=q^{\cl(T)}H\circ \fdeux(\theta_{q^{-1}})\circ \fdeux(\phi_{ehr_1})(T)(X)\\
&=q^{\cl(T)} \Theta_{q^{-1}}\circ H \circ \fdeux(\phi_{ehr_1})(T)(X)\\
&=q^{\cl(T)} \Theta_{q^{-1}}\circ P_{ehr_1}(T)(X)\\
&=q^{\cl(T)} P_{ehr}(T)\left(\frac{X}{q}\right).\qedhere
\end{align*} \end{proof}

\begin{prop} 
Let $q\in \K$. For any finite topology $T$:
\[H\circ \fdeux(\phi_q)(T)=\lambda_{ho}(T)X^{\cl(T)}.\]
\end{prop}

\begin{proof}
In order to lighten the proof, we write $F_q=H\circ \fdeux(\phi_q)$. For any $q'\in \K$,
$\phi_q\circ \iota_{q'}=\theta_{q'}\circ \phi_q$, so:
\begin{align*}
q'^{\cl(T)} F_q(T)&=F_q \circ \fdeux(\iota_{q'})(T)\\
&=H\circ \fdeux(\theta_{q'})\circ  \fdeux(\phi_q)(T)\\
&=\Theta_{q'}\circ H\circ  \fdeux(\phi_q)(T)\\
&=F_q(T)(q'X),
\end{align*}
so $F_q(T)$ is homogeneous of degree $\cl(T)$.  We put $F_q(T)=\mu(T)X^{\cl(T)}$. Then:
\begin{align*}
\mu(T)&=\varepsilon'(F_q(T))\\
&=\varepsilon'\circ H\circ \fdeux(\phi_q)(T)\\
&=\varepsilon'\circ \fdeux(\phi_q)(T)\\
&=\varepsilon'\circ \phi_q(T)\\
&=q^{\cl(T)}\lambda_{ho}(T).
\end{align*}
So $F_q(T)=q^{\cl(T)} \lambda_{ho}(T)$. \end{proof}

As $\phi_{ehr_1}=\phi_1\leftarrow \lambda_{ehr_1}$:

\begin{cor}
For any finite topology $T\in H_{\caltop}$:
\begin{align*}
P_{ehr_1}(T)&=\sum_{\sim\triangleleft G} \lambda_{ehr_1}(G\mid \sim) \lambda_{ho}(T/\sim)X^{\cl(\sim)}.
\end{align*}
\end{cor}

Applying the Fock functor to $\Gamma$:

\begin{prop}
The following maps are Hopf algebra automorphisms:
\begin{align*}
\fun(\Gamma)&:\left\{\begin{array}{rcl}
\bfh_{\caltop}&\longrightarrow&\bfh_{\caltop}\\
T&\longrightarrow&\displaystyle \sum_{\sim\triangleleft T} T/\sim, 
\end{array}\right.\\
\fdeux(\Gamma)&:\left\{\begin{array}{rcl}
H_{\caltop}&\longrightarrow&H_{\caltop}\\
G&\longrightarrow&\displaystyle \sum_{\sim\triangleleft T} T/\sim.
\end{array}\right.
\end{align*}
Their inverse are given by:
\begin{align*}
\fun(\Gamma)^{-1}&:\left\{\begin{array}{rcl}
\bfh_{\caltop}&\longrightarrow&\bfh_{\caltop}\\
T&\longrightarrow&\displaystyle  \sum_{\sim\triangleleft T} (-1)^{\cl(T)+\cl(\sim)}T/\sim,
\end{array}\right.\\
\fdeux(\Gamma)^{-1}&:\left\{\begin{array}{rcl}
H_{\caltop}&\longrightarrow&H_{\caltop}\\
T&\longrightarrow&\displaystyle \sum_{\sim\triangleleft T} (-1)^{\cl(T)+\cl(\sim)}T/\sim.
\end{array}\right.
\end{align*}
Moreover, the following diagram commutes:
\begin{align*}
&\xymatrix{&&&\WQSym\ar@{->>}[dddd]\\
\bfh_{\caltop}\ar[rr]_{\fun(\Gamma)}\ar@{->>}[d]\ar[rrru]^{\fun(\phi_{ehr_{-1}})}
&&\bfh_{\caltop}\ar@{->>}[d]\ar[ru]_{\fun(\phi_{ehr_1})}&\\
H_{\caltop}\ar[rr]^{\fun(\Gamma)}\ar[dr]^{P_{ehr_{-1}}}\ar@/_4pc/[ddrrr]_{\fdeux(\phi_{ehr_{-1}})}
&&H_{\caltop}\ar[ld]^{P_{ehr_1}}\ar[rdd]^<(.2){\fdeux(\phi_{ehr_1})}&\\
&\K[X]&&\\
&&&\QSym\ar[ull]^H}
\end{align*}
\end{prop}

\bibliographystyle{amsplain}
\bibliography{biblio}

\end{document}